\documentclass[11pt,a4paper]{amsart}

\RequirePackage[OT1]{fontenc}
\RequirePackage{amsthm,amsmath}
\RequirePackage[numbers]{natbib}
\RequirePackage[colorlinks,citecolor=blue,urlcolor=blue]{hyperref}
\RequirePackage{cleveref}

\usepackage{graphicx}
\usepackage{tikz}
\usepackage{enumitem}

\usetikzlibrary {arrows.meta} 

\usepackage{graphicx}
\newtheorem{theorem}{Theorem}
\newtheorem{lemma}[theorem]{Lemma}
\newtheorem{proposition}[theorem]{Proposition}
\newtheorem{corollary}[theorem]{Corollary}
\newtheorem{conjecture}[theorem]{Conjecture}
\newtheorem{definition}[theorem]{Definition}
\newtheorem{remark}[theorem]{Remark}
\newtheorem{assumption}[theorem]{Assumption}
\newtheorem{problem}{Problem}

\usepackage[utf8]{inputenc}
\usepackage[english]{babel}
\usepackage{latexsym}
\usepackage{amssymb}

\newcommand{\N}{\mathbb{N}}
\newcommand{\Q}{\mathbb{Q}}

\newcommand{\Z}{\mathbb{Z}}
\newcommand{\X}{\mathbf{X}}
\newcommand{\R}{\mathbb{R}}
\newcommand{\C}{\mathbb{C}}
\newcommand{\E}{\mathbb{E}}
\newcommand{\T}{\mathbb{T}}
\newcommand{\U}{\mathbb{U}}
\newcommand{\V}{\mathrm{V}}

\renewcommand{\P}{\mathbb{P}}
\newcommand{\Tr}{\mathrm{Tr}}
\renewcommand{\Re}{\mathrm{Re} \,}
\renewcommand{\Im}{\mathrm{Im} \,}
\newcommand{\1}{\mathbf{1}}

\newcommand{\x}{\boldsymbol{\xi}}

\newcommand{\dist}{\operatorname{dist}}

\numberwithin{equation}{section}
\numberwithin{theorem}{section}

\usepackage{fullpage}

\author{Janne Junnila$^1$ \and Gaultier Lambert$^2$  \and Christian Webb$^3$}
\date{ \today \\
$^1$University of Helsinki, Department of Mathematics and Statistics, \texttt{janne@junnila.me}\\%
$^2$University of Zurich, Institute of Mathematics,  \texttt{glambert@kth.se}\\%
$^3$University of Helsinki, Department of Mathematics and Statistics, {\tt christian.webb@helsinki.fi}\\[2ex]%
}

\title{Multiplicative chaos measures from thick points of log-correlated fields}

\begin{document}

\maketitle

\begin{abstract}
We prove that multiplicative chaos measures can be constructed from extreme level sets or \emph{thick points} of the underlying logarithmically correlated field. We develop a method which covers the whole subcritical phase and only requires asymptotics of suitable exponential moments for the field. As an application, we establish that these estimates hold for the logarithm of the absolute value of the characteristic polynomial of a Haar distributed random unitary matrix (CUE), using known asymptotics for Toeplitz determinant with (merging) Fisher-Hartwig singularities. Hence, this proves a conjecture of Fyodorov and Keating concerning the fluctuations of the volume of thick points of the CUE characteristic polynomial.
\end{abstract}


\section{Introduction} \label{sec:intro}

\subsection{Background and motivation}
Log-correlated fields are a class of stochastic processes that have recently appeared in various models of probability theory and mathematical physics (see e.g. \cite{ABBRS,Berestycki,BF,CMN,CN,CFLW,CLZ,FK,HKO,Kenyon,RV} and references therein). Formally, a (Gaussian) log-correlated field $\X$ is a centered Gaussian process on a metric space $\Omega$ with covariance
\begin{equation}\label{eq:cov}
C_{\X}(x,y):=\E \X(x)\X(y)=\log \dist(x,y)^{-1}+h(x,y)
\end{equation}
where $h: \Omega \times\Omega \to \R$ is continuous.
More specifically, we focus on the case where $\Omega\subset\R^d$ is either a bounded open set, in which case $\dist(x,y)=|x-y|$ denotes the Euclidean distance or a $d$-dimensional (smooth) Riemannian manifold, with or without a boundary, in which case $\dist$ denotes the intrinsic metric. 
We allow this generality because we are interested in the 2d Gaussian free field restricted to the unit circle $\T = \R/2\pi \Z$, that is, the Gaussian field with covariance,
\begin{equation}\label{eq:covT}
C_\X(\theta,x)=\log |e^{i\theta}-e^{ix}|^{-1} , \qquad \theta,x \in\T
\end{equation}
and its approximation given by the log characteristic polynomial of the circular unitary ensemble; cf.~Theorem~\ref{thm:cue}. 

Since the covariance \eqref{eq:cov} blows up on the diagonal, $C_\X(x,x)=\infty$, $\X$ must  be understood as a random generalized function ($\X$ is a Gaussian random element in a Sobolev space of negative regularity index; see e.g. \cite[Section 2]{JSW1} for a review of precise definitions).
Despite  log-correlated fields not being honest random functions, some of their geometric properties, such as extrema and extreme level sets can be understood to a degree.  Namely, if $\X_N$ is an approximation of $\X$ with relevant spatial scale $1/N$ (e.g. a discretization on a lattice of mesh $1/N$, a smoothing at scale $1/N$, or possibly a more complicated approximation coming from a model of statistical mechanics), in several instances, it has been proven that e.g. given a compact set $K\subset \Omega$ with non-empty interior, $\max_{x\in K}\X_N(x)=(1+o(1))\sqrt{2d}\log N$ (as $N\to\infty$) and that for $\gamma\in(0,\sqrt{2d})$
\begin{equation}\label{eq:tp}
\frac{\log |\{x\in K:\X_N(x)>\gamma \log N\}|}{\log N}=-(1+o(1))\frac{\gamma^2}{2}.
\end{equation}
This subset is known as the set of $\gamma$-thick points of the field and $|\cdot |$ denotes its Lebesgue measure (or volume form if we are on a manifold). A discussion of such claims in a rather general setting can be found in \cite[Section 3]{CFLW}.
In this context $\gamma = \sqrt{2d}$ is called the critical value and it plays a prominent role in this theory. 

A successful  approach to describe the geometry of log-correlated fields  is through studying multiplicative chaos (GMC) measures associated to $\X$. 
More precisely, under some mild assumptions on the approximations $\X_N$ ($\E_N$ denotes the expectation with respect to the law of $\X_N$), it is known that for $\gamma>0$, 
\begin{equation}\label{eq:GMC}
\mu_{N,\gamma}(x)dx:=\frac{e^{\gamma \X_N(x)}}{\E_N e^{\gamma \X_N(x)}}dx
\end{equation}
converges (with respect to the vague topology on $\Omega$) to a limiting measure $\mu_{\X,\gamma}$ as $N\to \infty$ (either almost surely, in probability, or in distribution -- depending on the type of approximation in question). 
Here $d x$ denotes the Lebesgue measure if $\Omega\subset\R^d$, or the volume from if $\Omega$ is a $d$-dimensional Riemannian manifold. 
We refer to e.g. \cite{Berestycki} and \cite[Section 2]{CFLW} for general convergence statements and further references. 
These measures are intimately related to the geometric properties mentioned above: the fact that a non-trivial limit exists only for $0<\gamma<\sqrt{2d}$ while for $\gamma\geq \sqrt{2d}$ the limit is zero is closely related to the leading order asymptotics of the maximum. Also it is known that in the $N\to \infty$ limit, the random measure $\mu_{N,\gamma}$ concentrates on the set of $\gamma$-thick points \eqref{eq:tp}; see e.g.~\cite[Section 3]{CFLW} for general claims and 
\cite[Theorem 1.3]{ABB17} or \cite[Proposition 1.6.]{Lam19} for statements in case of the circular ensembles.

It is expected that multiplicative chaos measures have a deeper connection to fluctuations of thick points and such results have been obtained in particular cases.
In \cite{GKS,BL,Jego,AHK}, the authors prove such results for respectively, branching Brownian motion, the discrete two-dimensional free field, local times of Brownian motion, and a log-correlated field which is a probabilistic model for the logarithm of the modulus of the Riemann zeta function on the critical line. These works nevertheless rely heavily on the specific properties of these models (e.g.~the branching structure in  branching Brownian motion, the Markov property of the free field, a martingale structure, etc). A general approach to this question, working even in the generality of Gaussian approximations of generic log-correlated fields is lacking, though one expects that a general result holds. Indeed, in the theoretical physics literature, such results have been conjectured to hold universally for any reasonable approximation of a logarithmically correlated random field. As an example, we quote the following conjecture of Fyodorov and Keating \cite[Section 2 (d)]{FK}.

\begin{conjecture}[Fyodorov and Keating (2014)]
\label{con:fk}
Let $U_N$ be a random matrix distributed according to the Haar measure on the group of $N\times N$ unitary matrices (known as the circular unitary ensemble or CUE) and let $p_N(\theta)=\det(1-U_N e^{-i\theta})$ be the characteristic polynomial of $U_N$ restricted to the unit circle
$\{e^{i\theta} : \theta \in \T\}$. 
Then as $N\to\infty$, for any $0<\gamma<1$, the random variable
\begin{equation} \label{ttmass}
\frac{\frac{1}{2\pi}\int_0^{2\pi}\mathbf 1\{ \log |p_N(\theta)|>\gamma\log N\}d\theta}{N^{-\gamma^2}\frac{1}{\sqrt{\pi \log N}}\frac{G(1+\gamma)^2}{2\gamma G(1+2\gamma)}\frac{1}{\Gamma(1-\gamma^2)}}
\end{equation}
converges in distribution to a positive random variable with density 
\[
\mathcal P_\gamma(x)= \gamma^{-2}x^{-1-\gamma^{-2}} e^{-x^{-\gamma^{-2}}}\1\{x>0\}.
\]
Here $G(z), z\in\C$ denotes the Barnes G-function (cf.~Lemma~\ref{le:morris}).
\end{conjecture}

There is another related conjecture \cite[Section 2.3]{FK} for the centering and fluctuations of  $\max_{\theta \in \T}\log |p_N(\theta)|$ as $N\to\infty$. In particular, the limit corresponds to the sum of two independent Gumbel random variables with a specific mean.
For circular $\beta$-ensembles for general $\beta>0$, the centering term has been obtained  in \cite{CMN} and the fluctuations have been described in \cite{PZ} as the sum of Gumbel and an independent \emph{derivative martingale}.  Identifying the law of this random variable in terms of GMC is still an open problem.

We return to the normalization of \eqref{ttmass}, the origin of this conjecture and the relationship to GMC after stating our main result for the CUE; Theorem~\ref{thm:cue}.
Let us point out that if a random variable $\Xi$ has density $\mathcal P_\gamma$, then  $\Xi^{-1/\gamma^2}$ is exponentially distributed with rate $1$.
Indeed, by making the change of variables $t=x^{-1/\gamma^2}$, one can check that for any bounded continuous $f:(0,\infty)\to \R$, 
\[
\E (f(\Xi^{-1/\gamma^2}))=\int_0^\infty f(x^{-1/\gamma^2}) \mathcal P_\gamma(x)dx=\int_0^\infty f(t) e^{-t}dt . 
\]

\smallskip

The main purpose of this article is to develop a robust machinery which allows to describe the fluctuations of the sets of thick points under some concrete assumptions on the covariance \eqref{eq:cov} of the field $\X$ and the asymptotics of exponential moments of its approximations $\X_N$. 
The relevant notation and assumptions are described in Section~\ref{sec:gen} and our main result valid for a general approximation scheme is presented in Section~\ref{sec:main}. 


\subsection{Main results}

In the setting of the previous section, let us introduce the measures, for $\gamma>0$, 
\begin{equation} \label{def:nu}
\nu_{N,\gamma}(dx) : =\frac{\mathbf 1\{\X_N(x)\geq \gamma \log N+g(x)\}}{\P_N(\X_N(x)\geq \gamma \log N)}dx 
\end{equation}
where $g: \Omega \to \R$ is a tuneable (continuous) function. 
Recall that $\mu_{\X,\gamma}$ denotes the GMC measures associated with  $\X$, that is, the (distributional) limit of the exponential measures~\eqref{eq:GMC}.

 One of our main results is studying this measure $\nu_{N,\gamma}(dx)$ in the setting of the circular unitary ensemble. More precisely, we have the following result.

\begin{theorem} \label{thm:cue}
Let $\X_N(x)=\sqrt{2}\log |\det(I-e^{-i x}U_N)|$ for $x\in\T$, where $U_N$ is a Haar distributed random $N\times N$ unitary matrix.

For any $\gamma\in(0,\sqrt{2})$ and  $g\in \mathcal{C}(\T \to \R)$, 
the random measure $\nu_{N,\gamma}$ converges in distribution as $N\to\infty$ (with respect to the weak topology)   to $e^{-\gamma g}\mu_{\X,\gamma}$. 
\end{theorem}

In relation to Conjecture \ref{con:fk}, let $\X_N =\sqrt{2} \log |p_N|$ so as to fix the critical value to $\sqrt{2}$  (which is the standard value in the literature) instead of 1. 
{It is known (e.g.~\cite[Theorem 7.5.1]{FMN16}) that for any $\gamma>0$, as $N\to\infty$, 
\begin{equation} \label{moddev}
\P_N(\X_N(\theta)\geq \gamma \log N) \sim {\frac{ N^{-\gamma^2/2} G(1+\gamma/\sqrt 2)^2}{\gamma \sqrt{2\pi \log N} G(1+\sqrt 2 \gamma)}}. 
\end{equation}
We offer a proof of this result in Lemma \ref{lem:prob} as our approach involves certain generalizations of these asymptotics. To verify the assumptions of Lemma \ref{lem:prob}, see \eqref{eq:barnesasy} for the appearance of the Barnes functions.}

Hence, as a consequence of Theorem~\ref{thm:cue} with $g=0$, for any $\gamma\in(0,\sqrt{2})$, 
\[
\frac{\frac{1}{2\pi}\int_0^{2\pi}\mathbf 1\{\X_N(\theta)>\gamma\log N\}d\theta}{\frac{N^{-\gamma^2/2} G(1+\gamma/\sqrt 2)^2}{\gamma \sqrt{2\pi \log N} G(1+\sqrt 2 \gamma)}} \to \mu_{\X,\gamma}(\T)
\]
in distribution as $N\to\infty$, where $\X$ is the restriction of the 2d Gaussian free field to  $\T$ (namely has the covariance \eqref{eq:covT}). 
In fact, our method also establishes that the moments of order $\beta \in (0,1]$ also converge; cf.~Theorem~\ref{thm:main} below.
In particular, $\E \mu_{\X,\gamma}(d \theta) = \frac{d \theta}{2\pi}$ with our convention and this should be compared to \eqref{ttmass} (upon multiplying by  $\Gamma(1-\gamma^2/2)$ and replacing $\gamma$ by $\sqrt{2}\gamma$). 
It is yet another conjecture from  \cite{FB} that the probability density function of the random variable $ \Gamma(1-\gamma^2/2) \mu_{\X,\gamma}(\T)$ can be calculated explicitly, and it is given by~$\mathcal P_\gamma$.
This conjecture has been proved independently using different methods in \cite{Remy,CN}. 
Hence, we obtain the following result;

\begin{corollary}
Conjecture \ref{con:fk} is true.
\end{corollary}

It is also of general interest to obtain similar results for smoothing (by convolution) of general Gaussian log-correlated fields.
Let $\Omega\subset \R^d$ be an open set, $h : \Omega\times \Omega\to \R$ be locally $\alpha$-Hölder continuous ($\alpha\in(0,1]$) and let $\X$ be a log-correlated field on $\Omega$ with covariance kernel \eqref{eq:cov}.

\begin{theorem} \label{thm:gauss}
Let  $\rho\in \mathcal{C}_c(\R^d\to\R_+)$ be a probability density function $(\int_{\R^d}\rho(x)dx=1)$ and let  $\rho_\delta=\delta^{-d}\rho(\delta^{-1}\cdot)$ for any $\delta>0$. 
Let $K \subset \Omega$ be a compact set. 
For a small $c=c_{K,\rho}>0$, define a stochastic process on $(0,c]\times K$ by 
\begin{equation}\label{eq:convapp}
\X_\delta := \rho_\delta*\X = \int_\Omega \X(x) \rho_\delta(\cdot-x) dx .  
\end{equation}
For any $\gamma\in (0,\sqrt{2})$ and  $g\in \mathcal{C}(\Omega \to \R)$, 
the random measure $\nu_{N,\gamma}$ $($with $\X_\delta$, $\delta=1/N$ in place of $\X_N$$)$ converges in probability as $N\to\infty$ $($with respect to the vague topology$)$ to $e^{-\gamma g}\mu_{\X,\gamma}$.
\end{theorem}

For $\delta>0$, \eqref{eq:convapp} denotes the convolution of the random Schwartz distribution $\X$ with a continuous compactly supported  test function, so it is almost surely well-defined (the constant $c_{K,\rho}$ depends only on the support of the mollifier $\rho$ and on $\dist(K,\partial\Omega)$).
This provides a natural family of approximations of $\X$ as $\delta\to0$. 

\smallskip

We will prove Theorems~\ref{thm:cue} and~\ref{thm:gauss} together by developing a general framework which guarantees that the random measures $\nu_{N,\gamma}$ and the usual GMC approximation $\mu_{N,\gamma}$ \eqref{eq:GMC} have the same limit (in probability in the subcritical phase $\gamma\in[0,\sqrt{2})$) under general assumptions on the approximation of $\X$; see Theorem~\ref{thm:main} below. 
These assumptions are stated in the next section in terms of  exponential moments of the field $\X_N$. Importantly, this allows us to consider non-Gaussian approximations such as the log characteristic polynomial of a random matrix model. 
We demonstrate the applicability of our method for the CUE in Sections \ref{sec:CUE} and obtain Theorem~\ref{thm:cue}. 
Most of the relevant asymptotic estimates are already available in the literature (see e.g. \cite{CK,DIK,DIK2}) based on the relationship with orthogonal polynomials and Riemann-Hilbert problems with Fisher-Hartwig singularities.  
This method and the appropriate results are reviewed in Section \ref{sec:CUE} and we carefully emphasize the (technical) modifications which are required to derive Theorem~\ref{thm:cue}.

While the CUE is arguably the most basic random matrix ensemble, we believe that our approach can be adapted to a large class of unitary invariant ensembles, such as the one considered in \cite{CFLW} for a related problem, modulo additional technicalities while performing the steepest descent analysis of the  corresponding Riemann-Hilbert problems.

\subsection{Notation and General assumptions} \label{sec:gen}
Throughout this article, the dimension $d\in\N$ is fixed. 
Recall that $\Omega\subset\R^d$ is either a bounded open set (as in Theorem~\ref{thm:gauss} in which case $\dist(x,y)=|x-y|$ denotes the Euclidean distance) or a $d$-dimensional smooth Riemannian manifold (as in Theorem~\ref{thm:cue} with
$d=1$, $\Omega=\T$ and $\dist(\theta,x) = |e^{i\theta}-e^{ix}|$). 
In the latter case, we assume that for any compact $K \subset \Omega$, the metric tensor $G$ satisfies $c_K \mathrm{I}_d \le G \le C_K \mathrm{I}_d $. 
In both cases, since we consider the vague topology of convergence of measures, 
one can always assume (by a partition of unity) that $\Omega\subset\R^d$ is a ball (the previous condition on $G$ guarantees that $\dist$ is locally equivalent to the Euclidean distance). 
Thus, to ease notation, we will now denote $\dist(x,y) = |x-y|$. 

The interpretation of $\X$ being a log-correlated field with covariance structure \eqref{eq:cov} is that for any test  function $f\in\mathcal{C}_c^\infty(\Omega\to\R)$, 
\begin{equation} \label{pair}
\langle \X,f\rangle : = \int f(x)\X(x) dx 
\end{equation}
is a centered Gaussian random variable with variance 
\(
\iint C_{\X}(x,y) f(x)f(y) dx dy
\); cf.~\eqref{eq:cov}. 
One can extend \eqref{pair} to more general functions (by density). For instance,  this quantity is (almost surely) well-defined for any $f\in L_c^\infty(\Omega)$ (see \eqref{eq:Lc} for a definition of this space). 

\medskip

In general, we consider two layers of approximation;
\begin{itemize}[leftmargin=5mm]
\item For $N\in\N$,  $\X_N:\Omega\to \R$ is a lower-semicontinuous function such that for any $f\in \mathcal{C}^\infty_c(\Omega)$, 
$\langle \X_N,f\rangle \to \langle \X,f\rangle $ in distribution as $N\to \infty$.
\item Define $\X_{N,\delta}(x) := \langle \X_N, \rho_{\delta,x}\rangle $ for $\delta\in (0,1]$ and $x\in\Omega$, where 
$(\delta,x,y) \in (0,1] \times \Omega \times \Omega \mapsto \rho_{\delta,x}(y)$ is a continuous approximation of the identity ($\rho_{\delta,x}\ge0$ and $\rho_{\delta,x}  \to \boldsymbol\delta_x $ vaguely as $\delta\to0$ where $\boldsymbol\delta_x $ is a Dirac measure at $x\in\Omega$). 
\end{itemize}

\smallskip

One views $\X_N$ as an approximation of the field $\X$ coming e.g.~from a statistical mechanics model on a (microscopic scale) $1/N$, and $\X_{N,\delta}$ as a smoothing of $\X_N$ on a scale $\delta>1/N$. 
In particular, $(\X_N)_{N\in\N}$ need not be defined on the same probability space, while the second layer of approximation $\X_{N,\delta}$ is defined on the same probability space as  $\X_N$ for every $N\in\N$.
We assume that $\X_N$ has mild regularity (e.g.~lower-semicontinuity such as the log characteristic polynomial in the context of Theorem~\ref{thm:cue}) so that the level set of \emph{$\gamma$-thick points} as in \eqref{def:nu} is well-defined. 
Let us also record that in the context of Theorem~\ref{thm:gauss}, one works with only the second level of approximation and we restrict ourself to a basic class of mollifiers,
$\rho_{\delta,x} = \rho_\delta(\cdot-x)$, with $\rho$ as in the statement.
Instead of a continuous approximation, one could consider a discretization of $\X$ on a graph embedded in $\Omega$ with mesh size $1/N$. The methods of this paper can be adapted to this case with similar assumptions, modulo some rather involved notational changes, so we do not pursue such arguments in this paper.

We define the kernels for $x,z \in \Omega$ and $\epsilon, \delta\in (0,1]$, 
\begin{equation} \label{kernel}
\begin{aligned}
C_{\X,\delta}(x,z) & :=  \int C_\X(u,z) \rho_{\delta,x}(u) d u  ,  \\
C_{\X,\delta,\epsilon}(x,z) &:= \int C_\X(u,v) \rho_{\delta,x}(u) \rho_{\epsilon,z}(v)du d v 
\end{aligned}
\end{equation}

In the sequel, we require that $\rho_{\delta,x}$ is a suitable smoothing kernel (depending on the covariance structure of the log-correlated field $\X$) in the following sense;

\begin{assumption}\label{ass:cov}
For any compact set $K\subset \Omega$, there exists $c_K>0$ so that
\[\begin{aligned}
C_{\X,\delta}(x,z)  = - \log\big( |x-z| \vee \delta \big)  +O(1) ,\qquad 
C_{\X,\delta,\epsilon}(x,z)   = - \log\big( |x-z| \vee \delta \big) +  O(1)
\end{aligned}\]
where both error terms are uniform for $(x,y) \in K^2$ and $\{(\epsilon,\delta) \in (0,c_K]^2 : \epsilon\le \delta \}$. 
\end{assumption}

Depending on the model, one is lead to consider different smoothings, which is why we work in an abstract setup and formulate general assumptions.  
However, the reader might keep in mind the following two schemes;

\begin{itemize}[leftmargin=5mm]
\item Let $(\phi_k)_{k\in\N}$ be a orthonormal basis of $L^2(\Omega)$ satisfying the following conditions: 

1) For $N\in\N$
\[
\X_N = \sum_{k\in\N} \x_{N,k} \phi_k  
\quad\text{where $(\x_{N,k})_{k\in\N}$ are (real) random variables such that $\sum_{k\in\N} \E\x_{N,k}^2 <\infty$}
\]
and for $k\in\N$, $\x_{N,k}\to\x_k$ in distribution as $N\to\infty$ where $\x_k\sim \mathcal{N}_{0,\lambda_k}$ are independent. 

2) The kernel \eqref{eq:cov} satisfies
\[
C_{\X}(x,y) = \sum_{k\in\N} \lambda_k \phi_k(x)\phi_k(y) . 
\]
Then, with $\rho_{\delta,x} = \sum_{\delta k \le 1} \overline{\phi_k(\cdot)} \phi_k(x)$, we have
\[
\X_{N,\delta}(x)= \langle \X_N, \rho_{\delta,x}\rangle =  \sum_{k\delta \le 1} \boldsymbol\xi_{N,k} \phi_k(x)
\]
\item  $\rho_{\delta,x}=\delta^{-d}\rho(\delta^{-1}(\cdot-x))$ for $x\in \Omega$ and $\delta\in(0,1]$ in which case we obtain  convolution approximations as in Theorem~\ref{thm:gauss}.  
\end{itemize}

For instance, in the context of Theorem~\ref{thm:cue}, 
$\X_N=\sqrt{2}\log |\det(I-e^{i\cdot}U_N)|$, where $U_N$ is a Haar distributed random $N\times N$ unitary matrix.
We can decompose this field in the Fourier basis,
\[
\X_N(x) =  \Re \bigg(\sum_{k\in\N} \frac{\Tr U_N^k}{k/\sqrt 2} e^{ikx} \bigg) , \qquad x\in\T .
\]
Moreover, $\frac{\Tr U_N^k}{k} \to \mathcal{N}_{k}$ in distribution as $N\to\infty$, 
where $(\mathcal{N}_{k})_{k\in\N}$ are independent complex Gaussians with variance $1/k$; see e.g.~\cite{DE01}. Then, we consider the approximation for $\delta\in (0,1]$, 
\begin{equation} \label{trseries}
\X_{N,\delta}(x) 
=\int_{\T} \sum_{|k|\leq \delta^{-1}}e^{ik(\theta-x)}\X_N(\theta)\frac{d\theta}{2\pi} 
=  \Re \bigg( \sum_{k\delta\le 1} \frac{\Tr U_N^k}{k/\sqrt 2} e^{ikx} \bigg), \qquad x\in\T . 
\end{equation}
These are random trigonometric polynomials. One could also consider other regularization schemes such as a convolution with the Poisson kernel on $\T$, 
$\rho_{\delta}(x) = \sum_{k\in\N}  e^{-\delta k+  ikx}$, in which case one ends up with regularizations of the form 
\[
\sum_{k\in\N} \frac{\Tr U_N^k}{k/\sqrt 2} e^{-\delta k+  ikx} 
=\sqrt{2}\log |\det(I-e^{-\delta+ix}U_N)|. 
\]

Our main assumptions are formulated in terms of exponential moments of $\X_N$ and $\X_{N,\delta} = \langle \X_N, \rho_{\delta,x}\rangle$. Since these assumptions are rather elaborate, we need to introduce further notations. 

For a (large) $R>0$ and a (small) $c>0$, let 
\begin{equation}
\label{eq:Dr}
\mathcal{D}_R :=\big\{  \zeta\in\C : \Re \zeta \in[0,\sqrt{2d}], |\Im \zeta| \le R \big\} \end{equation}
and for a compact\footnote{If $\Omega$ is compact manifold, we just let $K=\Omega$.} $K\subset\Omega$,
\begin{equation}
\label{eq:Ac}
\mathcal{A}_c\subset \big\{(x,y)\in K^2: |x-y| \ge c  \big\}.
\end{equation}
We also write $\mathcal{D}_\infty$ for the infinite strip $[0,\sqrt{2d}]\times \R$.
Note that the value $\sqrt{2d}$ corresponds to the GMC critical point. 

\medskip

Let $\mathrm{f}_{\delta,z} = \sum_{j=1}^q\xi_j \rho_{\delta_j,z_j}$
for $\delta=(\delta_1,...,\delta_q)\in (0,1]^q$, $z=(z_1,...,z_q) \in \Omega^q$ and 
$\xi = (\xi_1, \cdots, \xi_q)\in \R^q$, $q\in\N$. 
Then, for $\zeta_1, \zeta_2\in \mathcal{D}_\infty$, $x_1,x_2\in \Omega$ with $x_1\neq x_2$, we write
\begin{align} \label{mom1}
\E_N\big[ e^{ \zeta_1 \X_N(x_1) + \zeta_2 \X_N(x_2)+ \sum_{j=1}^q\xi_j \X_{N,\delta_j}(z_j)} \big] 
=   \Psi_{N,\delta}^{(\zeta_1,\zeta_2,\xi)}(x_1, x_2;z) \exp\Big( \tfrac{\zeta_1^2+ \zeta_2^2}{2} \log N \Big) \, \\ \notag
\exp\Big(\zeta_1\zeta_2 C_\X(x_1,x_2)+ \zeta_1 \int C_\X(x_1,u) \mathrm{f}_{\delta,z}(u) d u 
+ \zeta_2 \int C_\X(x_2,u) \mathrm{f}_{\delta,z}(u) d u 
+ \frac12\E \langle \X , \mathrm{f}_{\delta,z}  \rangle^2\Big) .
\end{align}
One should interpret this equation as saying that the RHS exists\footnote{
Note that as part of the Assumption~\ref{ass:cov},  $\E \langle \X , \mathrm{f}_{\delta,z}  \rangle^2 <\infty$ for any $z\in\Omega^q$ and $\delta\in(0,1]^q$. In particular, $ \mathrm{f}_{\delta,z} $ are admissible test functions for the field $\X$.} and defining $\Psi_N$ for the relevant choices of parameters. This quantity encodes the joint exponential moments of $\X_N$ and its approximation at different scales $\delta_j$ as  
$\langle \X_N , \mathrm{f}_{\delta,z}  \rangle = \sum_{j=1}^q \xi_j  \X_{N,\delta_j}(z_j) $
where $(\xi_j)_{j=1}^q$ are parameters.
We use the following shorthand notation; 
\[
 \Psi_{N}^{(\zeta_1,\zeta_2)}(x_1,x_2) =  \Psi_{N,\delta}^{(\zeta_1,\zeta_2,0)}(x_1, x_2; z)  \quad \text{if} \quad \xi=0.
 \]
and
\[
 \Psi_{N}^{(\zeta_1)}(x_1) =  \Psi_{N,\delta}^{(\zeta_1,0,0)}(x_1, x_2; z)  \quad \text{if} \quad  \zeta_2=\xi=0 .
 \]

The function $\Psi_N$ controls  the effects coming from any non-trivial (or non-Gaussian) behavior on microscopic scales in the model and we assume it satisfies the following properties; 
\begin{assumption}  \label{ass:meso}
The following assumptions hold for $\Psi_N$; for a fixed $q\in\N_0$.
\begin{itemize}[leftmargin=5mm]
\item[(1)] For $x_1,x_2\in \Omega$, $x_1\neq x_2$, $z\in\Omega^q$, $\xi\in\R^q$ and $\delta\in(0,1]^q$, the function $(\zeta_1,\zeta_2)\mapsto \Psi_{N, \delta}^{(\zeta_1,\zeta_2,\xi)}(x_1,x_2,z)$ is analytic in $\mathcal{D}_\infty^2$. 
\item[(2)]  
There exists a continuous  function
$\Psi : (\zeta,x) \in  \mathcal{D}_\infty \times \Omega \mapsto \Psi(\zeta,x)$ 
such that for any compact set $K\subset \Omega$,
$\Psi(0,x)=1$ and $ \Psi (\gamma,x) > 0$ for 
$(\gamma,x) \in  [0,\sqrt{2d}] \times K$. 
It holds for any (small) $\eta>0$ and (large) $R>0$,  and  any fixed $\xi\in \R^q$
\begin{equation} \label{PsiN}
\Psi_{N,\delta}^{(\zeta_1,\zeta_2,\xi)}(x_1, x_2;z)  = \Psi(\zeta_1,x_1) \Psi(\zeta_2,x_2) \big(1+\underset{N\to\infty}{o(1)} \big)
\end{equation}
uniformly for $\zeta_1,\zeta_2 \in \mathcal{D}_R$, $x_1,x_2 \in K$ with $|x_1-x_2| \ge N^{\eta-1}$, $z\in K^q$ and $\delta_1,\cdots, \delta_q \ge N^{-1+\eta}$. 
\end{itemize}
\end{assumption}

In case $\zeta_2 = 0$, \eqref{PsiN} holds locally uniformly for $(x_1,\zeta_1) \in \Omega \times  \mathcal{D}_\infty$; in particular it follows that the limiting function $\Psi$ is continuous and also analytic for $\zeta\in \mathcal{D}_\infty$. 
However, the fact that $\Psi$ does not vanish in  $[0,\sqrt{2d}] \times K$ is a non-trivial hypothesis.
From Assumption~\eqref{ass:meso}, we also recover that $\X_N$ is indeed a suitable approximation of $\X$; see e.g.~Lemma~\ref{lem:approx} below.

Let us further comment on this assumption. 
On the one hand,  if $\zeta_1=\zeta_2=0$, this implies that for $z\in \Omega^q$, $\xi\in\C^q$ and $\delta_1,\cdots, \delta_q \ge N^{-1+\eta}$, 
\[
\Psi_{N,\delta}^{(0,0,\xi)}(z) =  \E_N\big[ e^{\sum_{j=1}^q \xi_j \X_{N,\delta}(z_j)} \big] \sim \exp\big(\tfrac12\E \langle \X , \mathrm{f}_{\delta,z}  \rangle^2\big) .
\]
This means that on arbitrary mesoscopic scales $\delta\geq N^{-1+\eta}$, the regularized process  $\X_{N,\delta}$ is essentially Gaussian with covariance kernel 
as in \eqref{kernel}. 
On the other-hand, if $\zeta_2=\xi=0$, then  it holds locally uniformly for $(x,\zeta) \in \Omega \times  \mathcal{D}_\infty$,
\begin{equation} \label{W1}
\Psi_{N}^{(\zeta)}(x)  =  \E_N\big[ e^{ \zeta \X_N(x)} \big] \exp\big(- \tfrac{\zeta^2}{2} \log N\big) \sim  \Psi(\zeta,x).
\end{equation}

Note that it follows from this that  $\Psi(0,x)=1$  for all $x\in\Omega$. In particular, making this assumption in Assumption \ref{ass:meso} was not actually needed. {The interpretation we suggest the reader keeps in mind is that}  this function $\Psi$ encodes the non-trivial microscopic structure of $\X_N$.
Then, \eqref{mom1}-- \eqref{PsiN} imply that to leading order, $\X_N$ has variance $\log N$, correlation kernel $C_\X(x_1,x_2)$ for  $x_1\neq x_2$ and that this microscopic information decorrelates at any mesoscopic distances $N^{-1+\eta}$.

For instance,  in the context of Theorem~\ref{thm:cue}, by Lemma~\ref{le:morris},  the CUE characteristic polynomials satisfy for $(\zeta,\theta) \in  \mathcal{D}_\infty \times \T$, as $N\to\infty$, 
\begin{equation}\label{eq:barnesasy}
\E\big[ e^{\zeta\X_N(\theta)} \big]
= \E\big[ |p_N(\theta)|^{\sqrt{2}\zeta} \big]
\sim \Psi(\zeta,\theta) N^{\frac{\zeta^2}{2}} , \qquad
\Psi(\zeta,\theta) = \frac{G(1+\zeta/\sqrt 2)^2}{G(1+\sqrt2\zeta)},
\end{equation}
where $G(z), z\in\C$ denotes the Barnes G-function ($G(x)> 0$ if $x>0$.)
In the CUE case, these quantities are independent of $\theta\in\T$ because of the (distributional) rotational invariance of the eigenvalues of $U_N$. 
Moreover, the function $\Psi$ plays a key role in the \emph{moderate deviations} of the field $\X_N$ as reflected by \eqref{moddev}. 

\smallskip

Our method requires a slightly stronger version of this assumption when the singularities $(x_1,x_2)$ are macroscopically separated. 
In this regime, we must allow the imaginary parts of $\zeta_1,\zeta_2$ to grow like a (small) power of $\E \X_N^2 \sim \log N$.

\begin{assumption}\label{ass:macro}
The following assumptions hold for $\Psi_N$; for a fixed $q\in\N_0$.  
\begin{itemize}[leftmargin=5mm]
\item[(1)] 
For any (small) $c,\eta>0$, let $R_N=(\log N)^\eta$ and recall \eqref{eq:Dr}--\eqref{eq:Ac}. 
For fixed $\xi \in \R^q$, $\delta \in (0,1]^q$, $z\in \Omega^q$,
we have uniformly in $\zeta_1,\zeta_2\in \mathcal D_{R_N}$ and $(x,y)\in \mathcal{A}_c$, 
\[
\Psi_{N,\delta}^{(\zeta_1,\zeta_2,\xi)}(x_1,x_2;\xi) =\Psi(\zeta_1,x_1)\Psi(\zeta_2,x_2)\big(1+\underset{N\to\infty}{o(1)}\big) .
\]
\item[(2)]

For any compact $K\subset \Omega$, there exist constants $C=C_K$ and $\varkappa=\varkappa_K$ such that
\[
|\Psi(\zeta,x)|\leq C e^{ |\zeta|^\varkappa} , \qquad\text{$\zeta \in \mathcal{D}_\infty,\,  x\in K$.}
\]
\end{itemize}
\end{assumption}


\subsection{General results} \label{sec:main}
Recall that in terms of the (log-correlated)  field $\X_N$, we define the random measures
\[
\mu_N(dx)=\frac{e^{\gamma \X_N(x)}}{\E_N e^{\gamma \X_N(x)}}dx
\qquad\text{and}\qquad
\nu_N(dx)=\frac{\mathbf 1\{\X_N(x)\geq \gamma \log N+g(x)\}}{\P_N(\X_N(x)\geq \gamma \log N)}dx.
\]
Note that we have dropped the subscript $\gamma$ and $g\in \mathcal{C}(\Omega \to \R)$ is a tunable function. 
Our main statement is; 
\begin{theorem} \label{thm:main}
Fix $\gamma \in (0,\sqrt{2d})$ and $g\in \mathcal{C}(\Omega \to \R)$. 
Under Assumption~\ref{ass:cov}, \ref{ass:meso} and \ref{ass:macro},
for any $f\in\mathcal{C}_c^\infty(\Omega)$, 
\[
\lim_{N\to\infty}\E_N\big[| \nu_N(f) - \mu_N(e^{-\gamma g}f)| \big]=0.
\]
\end{theorem}

Our result shows that, under general hypothesis and in the whole subcritical phase, the fluctuations of the volume of the set of $\gamma$-thick points  and the exponential measure \eqref{eq:GMC} have the same limit in probability as $N\to\infty$.
If it is already known that the exponential measure converges to GMC\footnote{We emphasize that  Theorem~\ref{thm:main} applies even if the limit of the exponential measures are not GMC.} as $N\to\infty$, then we obtain the following immediate consequence; 

\begin{corollary} \label{cor:main}
Under Assumptions~\ref{ass:cov}, \ref{ass:meso} and \ref{ass:macro}, if $\mu_N \to \mu_{\X,\gamma}$ in distribution $($or in probability$)$ for $\gamma \in (0,\sqrt{2d})$  as $N\to\infty$, with respect to the vague topology, then $\nu_N\to e^{-\gamma g}\mu_{\X,\gamma}$ in the same sense. 
\end{corollary}


\subsection{Organization of the paper.} 
In Section~\ref{sec:strat}, we present the general strategy of the proof of Theorem~\ref{thm:main} which is based on the \emph{modified second moment method} by introducing a \emph{barrier} and the new ideas underlying this result. 
In Section~\ref{sec:proof}, we provide the details of the proofs, which are largely based on technical estimates using characteristic functions and Fourier analysis. 
In Section~\ref{sec:gauss}, we verify that a convolution approximation to a Gaussian log-correlated field satisfies the relevant assumptions.
This allows us to obtain Theorem~\ref{thm:gauss} by applying Corollary~\ref{cor:main}. 
Finally in Section~\ref{sec:CUE}, we verify that the log characteristic polynomial of the circular unitary ensembles satisfies the assumptions from Section~\ref{sec:gen}, thus we obtain Theorem~\ref{thm:cue}.  
In particular, the (exponential) moments conditions from Assumptions~\ref{ass:meso} and~\ref{ass:macro} are checked based on the determinantal structure of the CUE and the relationship to the Riemann-Hilbert problem for orthogonal polynomials on the unit circle. 
Finally, in the Appendix~\ref{sec:Fourier}, we provide further background on the (uniform) Gaussian approximation of a distribution function in terms of the asymptotics of its characteristic function. 

\smallskip

In the sequel, we use the following convention for the Fourier transform: for $n\in\N$ and $f\in\mathcal{C}_c^\infty(\R^n)$, 
\begin{equation*}
\widehat f(\xi) = \int_{\R^n}f(x) e^{-i\xi\cdot x}dx 
\qquad \text{and} \qquad
f(x)=\frac{1}{(2\pi)^n}\int_{\R^n}\widehat f(\xi) e^{i \xi\cdot x}d\xi.
\end{equation*}
We will use that the Fourier transform extends to a linear operator on $L^1(\R^n)$,  $L^2(\R^n)$ and on Schwartz distributions with the usual properties.

Recall that $\T = \R/2\pi\Z$ and for a function $V \in \mathcal{C}(\T\to\C)$, we define its Fourier coefficients,
\[
\widehat f_k = \int_\T f(\theta) e^{-i k\theta} \frac{d\theta}{2\pi}
\qquad \text{so that} \qquad
f(x)= \sum_{k\in\Z} \widehat f_k   e^{ik\theta} . 
\]

We also find it convenient to introduce notation for the space of functions that are compactly supported and essentially bounded:
\begin{equation}\label{eq:Lc}
L_c^\infty(\Omega)=\{f\in L^\infty(\Omega): \mathrm{supp}(f)\subset \Omega \text{ compact}\}.
\end{equation}


\section{Main steps of the proof of Theorem~\ref{thm:main}}  \label{sec:strat}

\subsection{Notation}

We fix the parameter $\gamma \in (0,\sqrt{2d})$, a continuous function $g:\Omega\to\R$ (cf.~Definition \ref{def:nu}) and a small parameter $\eta>0$ (``small" depending only on $\gamma<\sqrt{2d}$). 
We rely on the convention from Section~\ref{sec:gen}.
Let $K\subset \Omega$ be compact and $C=C_{\eta,K,R}$ denote a constant which changes from line to line (this constant also depends on the function $h$ from \eqref{eq:cov} and $g$ from \eqref{def:nu}).
If $C$ depends on another parameter, this will be emphasized. 

\medskip
We define the {\bf barrier} for any $\ell, L\in\N$ with $\ell <L$, 
\[
\mathcal{B}_{\ell , L}= \big\{   \X_{N, e^{-k}}   \le ( \gamma + \eta) k : k\in [\ell, L] \big\} . 
\]
We view $\mathcal{B}_{\ell , L}$ both as an event and a random subset of $\Omega$. When considering it as an event, we write for $x\in \Omega$ 
\[
\mathcal{B}_{\ell , L}(x)= \big\{   \X_{N, e^{-k}}(x)   \le ( \gamma + \eta) k : k\in [\ell, L] \big\}. 
\]
We also write $\mathcal B_{\ell,L}^{\mathsf c}$ for the complementary set (or event).

\medskip

Our proof of Theorem \ref{thm:main} is divided in three steps. 

\subsection{Step 1, Truncation.}
This step aims to show that the barrier is a typical event in the sense that the following bounds hold.

\begin{proposition} \label{prop:L1}
For $\eta\in (0,1\wedge\gamma)$, let $L_{N}(\eta) := \lfloor \log N^{1-\eta}  \rfloor$ and $f\in L^\infty_c(\Omega)$. There exists a constant $C= C_{\eta,g}$ such that  for all $\ell, N\in\N$, 
\[
\E_N[ \mu_N(f\1_{\mathcal{B}_{\ell , L_N}^{\bf c}})] , \E_N[ \nu_N(f\1_{\mathcal{B}_{\ell , L_N}^{\bf c}})] 
\le C\|f\|_\infty e^{-\frac{\eta^2}{2}\ell} . 
\]
\end{proposition}

The proof of this proposition is given in Section~\ref{sec:L1}. 
It relies on a simple union bound based on our assumptions. 
In particular, it is important to set up the barrier so that $e^{-L_N} = N^{-1+\eta}$ is a mesoscopic scale.

\medskip

By Proposition~\ref{prop:L1} and Cauchy-Schwarz's inequality,  it holds for any $f\in L_c^\infty(\Omega)$ 
\begin{equation} \label{step1}
\E_N\big[| \nu_N(f) - \mu_N(e^{-\gamma g}f)| \big] \le \left( \E_N \big[ | \nu_N(f \1_{\mathcal{B}_{\ell , L_N}}) - \mu_N(e^{-\gamma g}f\1_{\mathcal{B}_{\ell , L_N}})|^2 \big]\right)^{1/2} + O(e^{-\frac{\eta^2}{2}\ell} ), 
\end{equation}
where the implied constant is independent of $N$ (but may depend on $\gamma,\eta,f,g$). Moreover, expanding the square, we can write
\begin{equation} \label{decomp}
\E_N \big[ | \nu_N(f \1_{\mathcal{B}_{\ell , L_N}}) - \mu_N(e^{-\gamma g}f\1_{\mathcal{B}_{\ell , L_N}})|^2  \big]
=  \iint_{\Omega\times\Omega} f(x)f(y) \big(\Theta_{1, N,\ell}+ \Theta_{2, N,\ell} -2\Theta_{3, N,\ell}\big)  dxdy
\end{equation}
where for $(x,y) \in\Omega^2$, 
\begin{equation} \label{Theta}
\begin{aligned}
\Theta_{1,N,\ell}(x,y)  
&=  \E_N\big[\mathbf 1_{\mathcal{B}_{\ell , L_N}(x)} \mathbf 1_{ \mathcal{B}_{\ell , L_N}(y)}  \mu_N(x) \mu_N(y) \big] e^{-\gamma g(x)-\gamma g(y)}  \\
\Theta_{2,N,\ell}(x,y) 
&=  \E_N\big[ \mathbf 1_{\mathcal{B}_{\ell , L_N}(x)} \mathbf 1_{ \mathcal{B}_{\ell , L_N}(y)}  \nu_N(x) \nu_N(y) \big]  \\
\Theta_{3,N,\ell}(x,y) 
&=  \E_N\big[\mathbf 1_{\mathcal{B}_{\ell , L_N}(x)} \mathbf 1_{ \mathcal{B}_{\ell , L_N}(y)} \nu_N(x) \mu_N(y) \big] e^{-\gamma g(y)}. 
\end{aligned}
\end{equation}

The next two steps aim to show that (the integrals of) $\Theta_{j,N,\ell}$ for $j\in\{1,2,3\}$ all have the same limit  as $N\to\infty$ and then $\ell\to\infty$. 

\subsection{Step 2, Excluding the diagonal.}

Let us fix a compact $K\subset\Omega$; e.g.~ $K=\mathrm{supp}(f)$. 
Let us recall a direct consequence of \eqref{mom1} and Assumption~\ref{ass:meso}.
If we drop the barriers in the expression of  $\Theta_1$, then 
\begin{equation} \label{W}
\E_N\big[\mu_N(x) \mu_N(y) \big]
= \frac{ \E_N\big[ e^{ \gamma \X_N(x) + \gamma \X_N(y)}\big]}{ \E_N\big[ e^{ \gamma \X_N(x)}\big]\E_N \big[ e^{\gamma \X_N(y)}\big]}
= \frac{\Psi_N(x,y)}{\Psi_N(x)\Psi_N(y)} e^{\gamma^2C_\X(x,y)}
\end{equation}
so that  for $x,y \in  K$ with $|x-y| \ge N^{-1+\eta}$, 
\begin{equation} \label{W2}
\E_N\big[\mu_N(x) \mu_N(y) \big] =  |x-y|^{-\gamma^2}  e^{\gamma^2 h(x,y)}\big(1+\underset{N\to\infty}{o(1)} \big)   
\end{equation}
where $h$ is as in \eqref{eq:cov}.

This estimate is \emph{sharp} up to microscopic scales: using \eqref{W1} we verify that
$\E_N\big[\mu_N(x)^2\big] = C_{\gamma,x} N^{\gamma^2}$ which is on the same scale as the above estimate if $|x-y|\simeq N^{-1}$.  Moreover, for $\gamma\in[\sqrt{d},\sqrt{2d})$,  the RHS of \eqref{W2}  
is not locally integrable on $\Omega^2$, which is why we introduce the barrier in the first place.
In fact, using the barrier and Markov's inequality, it holds for any $\alpha \ge 0$, 
\[\begin{aligned}
\E_N\big[\mathbf 1_{\mathcal{B}_{\ell , L_N}(x) }\mu_N(x)^2\big]
\le \E_N\big[\mu_N(x)^2\big] \frac{\E_N\big[ e^{2\gamma \X_N(x) -\alpha \X_{N,\delta}(x)} \big]}{\E_N\big[ e^{2\gamma \X_N(x)} \big]}  \delta^{-\alpha(\gamma+\eta)}
\end{aligned}\]
where we choose $\delta = e^{-L_N} = N^{-1+\eta} $. 
Then, according to Assumption~\ref{ass:meso} (with $\zeta_1=2\gamma$, $\zeta_2=0$, $q=1$, $\xi = -\alpha$ (in the numerator) or $\xi= 0$ (in the denominator) and $z=x$), this implies that for $x\in K$‚
\[
\E_N\big[\mathbf 1_{\mathcal{B}_{\ell , L_N}(x) }\mu_N(x)^2\big]
\le  C_{\alpha} N^{\gamma^2}  \delta^{-\tfrac{\alpha^2}{2}+\alpha(\gamma-\eta)} .
\]

Hence, choosing $\alpha=\gamma-\eta$, we obtain the uniform bound for $x\in K$‚
\begin{equation} \label{L2}
\E_N\big[\mathbf 1_{\mathcal{B}_{\ell , L_N}(x) }\mu_N(x)^2\big]
\le  C N^{\gamma^2-(\gamma-\eta)^2\frac{1-\eta}{2} }  .
\end{equation}

Now, by Cauchy-Schwarz's inequality, this estimate implies that for $x,y\in K$, 
\[
\Theta_{1,N,\ell}(x,y)  \le  C  N^{\gamma^2-(\gamma-\eta)^2\frac{1-\eta}{2} } 
\]
so that for $\gamma < \sqrt{2d}$, 
\[
\iint_{K^2} \1_{\{|x-y| \le N^{\eta-1}\}} \Theta_{1,N,\ell}(x,y)dxdy 
\le C N^{\gamma^2-(\gamma-\eta)^2\frac{1-\eta}{2}+d(\eta-1)}.
\]
Thus, by choosing $\eta>0$ small enough, this quantity  converges to 0 as $N\to\infty$.

Similarly, according to Lemma~\ref{lem:prob}, we have for $x\in K$,
\[
\E_N\big[\nu_N(x)^2\big]   =\frac{\P_N[\X_N(x) \ge \gamma \log N+ g(x)] }{\P_N[\X_N(x) \ge \gamma \log N]^2 } \le  C (\log N)^{1/2} N^{\gamma^2/2} . 
\]
Hence, by Cauchy-Schwarz's inequality again (ignoring the barrier for $\nu_N$-terms), we obtain for $j \in \{1,2,3\}$ and $\eta>0$ small enough (depending only on $\gamma<\sqrt{2d}$)
\begin{equation} \label{Thetabd}
\iint_{K^2} \1_{\{|x-y| \le N^{\eta-1}\}} \Theta_{j,N,\ell}(x,y)dxdy 
\le C N^{- \epsilon}
\end{equation}
for some $\epsilon=\epsilon_{\eta,\gamma}>0$. 

\medskip

By using again the barrier, we obtain the following result in Section~\ref{sec:meso}.

\begin{proposition} \label{prop:meso}
Let $0<\eta<\min(\gamma,1)$,  $\ell\in\N$ and $K\subset \Omega$ be a compact. There exists a constant $C=C_{\eta,\gamma,K,g}$ such that for $ j\in\{1,2,3\}$ and for $x,y\in K$ with $N^{\eta-1} \le |x-y| \le e^{-\ell}$, 
\begin{equation} \label{Thetaest}
\Theta_{j,N,\ell}(x,y)  \le C  |x-y| ^{-\frac{(\gamma+\eta)^2-\eta^2}{2}}. 
\end{equation}
\end{proposition}

The proof of Proposition~\ref{prop:meso} relies on Assumption~\ref{ass:meso} and Fourier analytic arguments. 
In contrast to \eqref{W2}, the RHS of \eqref{Thetaest} is integrable near the diagonal (at least for small enough $\eta$) and the constant $C$ is independent of the parameter $\ell$. 
Hence, by combining the bounds \eqref{Thetabd} and \eqref{Thetaest}, we conclude that if the parameter $\eta$ is small enough, then for any  $\ell, N\in\N$, 
\[
\iint_{K^2} \1_{\{|x-y| \le e^{-\ell}\}} \Theta_{j,N,\ell}(x,y)dxdy 
\le C e^{-\frac{\eta^2}{2}\ell} . 
\]

Going back to \eqref{step1}--\eqref{decomp}, this shows that 
\begin{equation} \label{step2}
\begin{aligned}
&\E_N\big[| \nu_N(f) - \mu_N(e^{-\gamma g}f)| \big] \\
&\le \left(
\iint\1_{\{|x-y| \ge e^{-\ell}\}} f(x)f(y)\big(\Theta_{1, N,\ell}+ \Theta_{2, N,\ell} -2\Theta_{3, N,\ell}\big)(x,y)  dxdy
+ O(e^{-\frac{\eta^2}{2}\ell} ) \right)^{1/2}. 
\end{aligned}
\end{equation}

This completes the second step of the proof. 
The last step consists basically in computing the pointwise limit of $\Theta_{j,N,\ell}(x,y) $ for $j\in\{1,2,3\}$ when the points $(x,y)$ are macroscopically separated. 

\subsection{Step 3, Macroscopic regime}

As we already suggested,
we expect that the limits of $\Theta_{j,N,\ell}$ for $j\in\{1,2,3\}$ are all of the same form. More precisely, we will show that 
\begin{equation} \label{lim}
\lim_{N\to\infty} \Theta_{j, N,\ell}(x,y) 
= \frac{\P_{\mathcal{N}(\mathbf{m}_{x,y}, \Sigma_{x,y})}\big[ X_{2k-1} , X_{2k}   \le ( \gamma + \eta) k : k\in [\ell, c\ell]  \big]}{|x-y|^{\gamma^2} e^{\gamma g(x) +\gamma g(y)}} e^{\gamma^2 h(x,y)} + O(e^{-\frac{\eta^2}{2}\ell} ) . 
\end{equation}
where $\mathcal{N}(\mathbf{m}_{x,y}, \Sigma_{x,y})$ is a multivariate Gaussian measure with mean $\mathbf{m}_{x,y} \in \R^{2c\ell} $, covariance matrix $\Sigma_{x,y} \in \R^{2c\ell \times 2c\ell} $ and $c\in\N$ is a large enough constant (depending only on $(\gamma,\eta)$). 
If these limits were to hold uniformly over $\{|x-y| \ge e^{-\ell}\}$, 
by taking first $N\to\infty$ and then $\ell\to\infty$, we could conclude that
\[
\lim_{N\to\infty} \E_N\big[| \nu_N(f) - \mu_N(e^{-\gamma g}f)| \big]  = 0
\]
which would complete the proof of Theorem~\ref{thm:main}. 

While it would be possible to obtain \eqref{lim} with the required uniformity, we proceed differently since this is technically slightly simpler.  More precisely, in Section~\ref{sec:macro}, we prove the following results. 

\begin{proposition} \label{prop:macro}
Fix $\ell , L \in \N$ with $L>\ell$.
For $j\in\{1,2\}$, it holds pointwise for $x,y \in \Omega$ with $x\neq y$,
\begin{equation} \label{Tmacro}
\limsup_{N\to\infty} \Theta_{j, N,\ell}(x,y) 
\le   \frac{\P_{\mathcal{N}(\mathbf{m}_{x,y}, \Sigma_{x,y})}\big[ X_{2k-1} , X_{2k}   \le ( \gamma + \eta) k : k\in [\ell, L]  \big]}{e^{\gamma^2 C_\X(x,y)} e^{\gamma g(x) +\gamma g(y)}} .
\end{equation}
The functions $(x,y)  \mapsto \mathbf{m}_{x,y} \in \R^{2L}$ and $(x,y)  \mapsto \Sigma_{x,y} \in \R^{2L\times 2L}$ taking values in positive-definite matrices are both continuous functions on $\Omega\times\Omega$ and given by Definition~\ref{def:mS}. 

Moreover, for any $D>0$, if $L = R \ell$ with $R$ sufficiently large (depending on $\eta,\gamma, D$), then 
\begin{equation*} \label{Upsilon0}
\Theta_{3,N,\ell} \ge  \Upsilon_{1,N,\ell}   -  \Upsilon_{2,N,\ell} 
\end{equation*}
where 
\begin{equation} \label{Upsilon2}
\limsup_{N\to\infty} \sup\big\{ \Upsilon_{2,N,\ell}(x,y) : x,y \in K ;  |x-y|  \ge e^{-\ell} \big\}  \le   C_{\eta,K,R} \, e^{-D\ell}
\end{equation}
and it holds pointwise for any $x,y\in \Omega$ with $x\neq y$, 
\begin{equation} \label{T3macro}
\lim_{N\to\infty}  \Upsilon_{1,N,\ell}(x,y)
=   \frac{\P_{\mathcal{N}(\mathbf{m}_{x,y}, \Sigma_{x,y})}\big[ X_{2k-1} , X_{2k}   \le ( \gamma + \eta) k : k\in [\ell, L]  \big]}{e^{\gamma^2 C_\X(x,y)} e^{\gamma g(x)+\gamma g(y)}}.  
\end{equation}
\end{proposition}

The proof of Proposition~\ref{prop:macro} relies on Assumption~\ref{ass:macro} and it is the most involved probabilistic argument of this paper. 
While proving Proposition~\ref{prop:macro}, we also obtain the following bounds: for $j=1,2,$ and $\ell\in\N$, 
\begin{equation} \label{macrobd}
\sup\big\{ \Theta_{j,N,\ell}(x,y)  : \ x,y \in K ;  |x-y|  \ge e^{-\ell} \big\}   \le C_{\ell,K} . 
\end{equation}

By \eqref{step2} and using the bound \eqref{Upsilon2}, we obtain
\[
\begin{aligned}
&\E_N\big[| \nu_N(f) - \mu_N(e^{-\gamma g}f)| \big] \\
&\le \left(
\iint\1_{\{|x-y| \ge e^{-\ell}\}} f(x)f(y)\big(\Theta_{1, N,\ell}+ \Theta_{2, N,\ell} -2\Upsilon_{1,N,\ell}\big)(x,y)  dxdy
+ O(e^{-\frac{\eta^2}{2}\ell} ) \right)^{1/2}. 
\end{aligned}
\]
Moreover,
combining  \eqref{Tmacro} and \eqref{T3macro} shows that the pointwise limit (when $N\to \infty$) of the integrand above  is $\le 0$. 
Since it is uniformly bounded from above on the set $\{|x-y| \ge e^{-\ell}\}$, cf.~\eqref{macrobd},   by the (reverse) Fatou's lemma, we conclude that
\[
\limsup_{N\to\infty} \E_N\big[| \nu_N(f) - \mu_N(e^{-\gamma g}f)| \big] =  O(e^{-\frac{\eta^2}{4}\ell} )  . 
\]
The LHS is independent of the parameter $\ell$ so this quantity converges to 0 and it completes the overall argument. 
In the next section, we provide the details of the proof.


\section{Proof of Theorem \ref{thm:main}}  \label{sec:proof}

In this section, we give the complete proof of Theorem \ref{thm:main} following the  strategy explained in Section \ref{sec:strat}. 
We rely on the setting from Section~\ref{sec:gen} as well as on the notation from Section \ref{sec:strat}; in particular, on \eqref{def:nu}, \eqref{mom1}, \eqref{Theta} and the Assumptions~\ref{ass:cov}, \ref{ass:meso} and \ref{ass:macro}.

\subsection{Proof of Proposition~\ref{prop:L1}.}
\label{sec:L1}
 As a size estimate suffices for us here, we can simply bound $|f(x)|\leq \|f\|_\infty \mathbf 1_K(x)$ where $K=\mathrm{supp}(f)$, and by scaling, we can assume that $\|f\|_\infty=1$, so we will focus on the setting $f(x)=\mathbf 1_K(x)$ for some compact $K\subset \Omega$.  
Let us begin by bounding $\E_N[ \mu_N(\1_K\1_{\mathcal{B}_{\ell , L_N}^{\bf c}})]$ in terms of $\ell$.
By definition of the barrier and a union bound,
\[
\mu_N(\1_K\1_{\mathcal{B}_{\ell , L_N}^{\bf c}}) \le   \sum_{k=\ell}^{L_N}
\int_K \1\{\X_{N, e^{-k}}(x)   \ge ( \gamma + \eta) k\}
\mu_N(x) dx
\]

According to \eqref{mom1} and \eqref{kernel}, for $\beta, \gamma \in [0,\sqrt{2d}]$, $\alpha\in\R$ and $x\in\Omega$, 
\[
\frac{\E_N\big[e^{\beta \X_N(x)+\alpha \X_{N,\delta}(x)}\big]}{\E_N\big[e^{\gamma \X_N(x)}\big]}
= \frac{\Psi_{N,\delta}^{(\beta,0,\alpha)}(x,0;x)}{\Psi_N^{(\gamma)}(x)}
N^{\frac{\beta^2-\gamma^2}{2}} \exp\Big( \beta\alpha C_{\X,\delta}(x,x)+ \frac{\alpha^2}2 C_{\X,\delta,\delta}(x,x) \Big) 
\]
Then, as a consequence of the Assumptions~\ref{ass:cov} and \ref{ass:meso},
there is a constant $C= C_{K,\eta}$ so that for $\delta\ge N^{\eta-1}$, 
$\beta, \gamma , \alpha \in [0,\sqrt{2d}]$, 
\begin{equation} \label{mom2}
\frac{\E_N\big[e^{\beta \X_N(x)+\alpha \X_{N,\delta}(x)}\big]}{\E_N\big[e^{\gamma \X_N(x)}\big]}
\le C N^{\frac{\beta^2-\gamma^2}{2}} \delta^{-\beta\alpha-\alpha^2/2}
\end{equation}

By Markov's inequality, this bound with $\beta=\gamma$ implies that for any $\alpha>0$
\[\begin{aligned}
\E_N\big[ \mu_N(\1_K\1_{\mathcal{B}_{\ell , L_N}^{\bf c}}) \big] 
& \le \sum_{k=\ell}^{L_N} e^{-\alpha(\gamma+\eta)k} \int_K\frac{\E_N\big[ e^{ \gamma \X_N(x) +\alpha \X_{N,e^{-k}}(x)} \big]}{\E_N\big[ e^{ \gamma \X_N(x)} \big]}  dx \\
&\le C_\alpha  \sum_{k=\ell}^{L_N} e^{(\frac{\alpha^2}{2}-\alpha\eta)k} .
\end{aligned}\]
Choosing e.g.~$\alpha=\eta$, this sum converges and we obtain that for any $N, \ell\in\N$, 
\[
\E_N\big[ \mu_N(f\1_{\mathcal{B}_{\ell , L_N}^{\bf c}}) \big] \le C  e^{-\frac{\eta^2}{2}\ell} .
\]
This proves the first claim. We can proceed similarly for the second claim:
\begin{equation} \label{L1}
\E_N[ \nu_N(\1_K\1_{\mathcal{B}_{\ell , L_N}^{\bf c}})]
\le  \sum_{k=\ell}^{L_N}  e^{-\alpha(\gamma+\eta)k} \int_K
\frac{\E_N\big[\mathbf 1\{\X_N(x) \ge \gamma \log N +g(x) \}  e^{\alpha \X_{N,e^{-k}}(x)}
\big]}{\P_N[\X_N(x) \ge \gamma \log N]}  dx . 
\end{equation}

We can use the method from the Appendix~\ref{sec:1dapprox} to compute every term in this sum. 
For future purposes, we record a more general proposition. 
Recall that $\mathcal{A}_c:=\big\{ (x,y) \in K^2 : |x-y| \ge c  \big\}$ for a small $c>0$. 

\begin{proposition} \label{prop:mixed2}
For $\gamma,\zeta\in (0, \sqrt{2d})$,  $\alpha\in[0,\gamma]$, there exists a constant $C = C_{K,\eta,g}$ such that for any $c>0$ there is $N_c\in\N$ so that it holds for $(x,y) \in \mathcal{A}_c$, $N\ge N_c$ and for $\delta \in [N^{\eta-1},1]$,
\[
\frac{\E_N\big[ \mathbf 1\{ \X_N(x) \ge \gamma \log N +g(x)\} e^{\zeta \X_N(y) +\alpha \X_{N,\delta}(x)} \big]}{\P_N[\X_N(x) \ge \gamma \log N]  \E_N[e^{\zeta \X_N(y)}] }  \\
\le C \delta^{-\alpha\gamma}  \frac{\E_N\big[e^{\beta_{N,\delta} \X_N(x)+\zeta \X_N(y)+\alpha \X_{N,\delta}(x)}\big]}{\E_N\big[e^{\gamma \X_N(x)}\big] \E_N[e^{\zeta \X_N(y)}]} . 
\]
where $\beta_{N,\delta} =   \gamma + \alpha \frac{\log \delta}{\log N}$.

In case $\zeta=0$, the previous bound holds for all $N\in\N$ and $x\in K$.  
\end{proposition}

\begin{proof}
Let us denote $X_N = \frac{\X_N(x) - \gamma \log N}{\sqrt{\log N}}$, $\sigma_N = \sqrt{\log N}$ and define a new probability measure
\[
\frac{d\Q_{N,x,y}}{d\P_N} : =  \frac{e^{\zeta \X_N(y)+\alpha \X_{N,\delta}(x)}}{\E_N[e^{\zeta\X_N(y)+\alpha \X_{N,\delta}(x)}]}  .
\]

We want to compute the asymptotics of $\Q[\X_N(x) \ge \gamma \log N +g]$ where $\Q=\Q_{N,y,x}$ by using Lemma~\ref{lem:1}. 
The characteristic function of the random variable $X_N$ under the measure $\Q$ biased by $e^{\beta \X_N(x)}$ is given by, 
\[
\frac{\Q\big[ e^{i\chi X_N+ \beta \X_N(x)} \big]}{\Q\big[ e^{\beta  \X_N(x)} \big]} = 
\frac{\E_N\big[e^{\zeta \X_N(y)+ (\beta+ \tfrac{i\chi}{\sigma_N})\X_N(x)+\alpha \X_{N,\delta}(x)}\big]}{\E_N\big[e^{\zeta \X_N(y)+\beta \X_N(x)+\alpha \X_{N,\delta}(x)}\big]}  e^{- i \chi \gamma \sigma_N}, 
\]
for $\chi\in \R$.
Using \eqref{mom1} and \eqref{kernel}, we can rewrite this characteristic function as 
\begin{equation*} 
\frac{\Q\big[ e^{i\chi X_N+ \beta \X_N(x)} \big]}{\Q\big[ e^{\beta  \X_N(x)} \big]} 
= \frac{\Psi_N^{(\beta+ \frac{i\chi}{\sigma_N},\zeta,\alpha)}(x,y;x)}{\Psi_N^{(\beta,\zeta,\alpha)}(x,y;x)}
e^{-\tfrac{\chi^2}{2}+i \zeta\tfrac{\chi}{\sigma_N}C_\X(x,y)+i\alpha\tfrac{\chi}{\sigma_N}C_{\X,\delta}(x,x)-i\chi(\gamma-\beta)\sigma_N}
\end{equation*}
where we used that $((\beta+ \tfrac{i\chi}{\sigma_N})^2-\beta^2) \log N -2i \chi \gamma \sigma_N =-\chi^2 -2 i\chi(\gamma-\beta) \sigma_N$.

Then, it is natural to choose  $\beta =\beta_{N,\delta}=  \gamma + \alpha \frac{\log \delta}{\log N}$ so that $\beta_{N,\delta} \in [{\gamma} \frac\eta2,\gamma]$ for $\alpha\in[0,\gamma]$, $\delta \in[N^{ \eta-1},1]$  (if $N$ is sufficiently large) and 
\[
\tfrac{\alpha}{\sigma_N}C_{\X,\delta}(x,x)-(\gamma-\beta_{N,\delta})\sigma_N = \tfrac{\alpha}{\sigma_N}\big( C_{\X,\delta}(x,x)+ \log \delta\big) =  O(\sigma_N^{-1})
\]
 uniformly for  $(x,y) \in \mathcal{A}_c$, $\delta \in [N^{\eta-1},1]$; cf.~Assumption~\ref{ass:cov}. 
Hence, the characteristic function is of the form
\[
\frac{\Q\big[ e^{i\chi X_N+ \beta_{N,\delta}\X_N(x)} \big]}{\Q\big[ e^{\beta_{N,\delta}  \X_N(x)} \big]} 
= \psi_{N,\delta,x,y}\big(i\tfrac\chi{\sigma_N}\big) e^{-\tfrac{\chi^2}{2}} 
\]
and, in the regime where $\beta_{N,\delta}\to\beta$ as $N\to\infty$, 
\[
\psi_{N,\delta,x,y}(z) \to \frac{\Psi(\beta+z,x)}{\Psi(\beta,x)}  e^{\zeta z C_\X(x,y)}
\]
locally uniformly for $z\in \mathcal{S} = \{z \in\C  :  |\Re z| \le\mathcal W\}$  and uniformly for $(x,y) \in \mathcal{A}_c$; upon choosing $\mathcal W>0$ small enough (cf.~Assumption \ref{ass:meso}). 

Thus, the condition \eqref{ass1} holds (provided that $\beta_{N,\delta}\to\beta$ and $\beta>0$ by assumptions), then by applying Lemma~\ref{lem:1}, we conclude that uniformly for $(x,y) \in \mathcal{A}_c$ and locally uniformly for $g\in\R$, 
\[
\Q[\X_N(x) \ge \gamma \log N + g]
=\frac{\Q[e^{\beta_{N,\delta} \X_N(x)}] N^{-\gamma^2} \delta^{-\alpha\gamma}}{\beta\sqrt{2\pi \log N}}  e^{-\beta g}\big( 1+ \underset{N\to\infty}{o(1)} \big) .
\]

Hence, by definition of  $\Q$, this implies that for  $(x,y) \in \mathcal{A}_c$, there is a constant $C = C_{\eta, K,B}$ so that for $g\in [-B,B]$, $N\ge N_{c,\eta}$ and $\delta \in[N^{ \eta-1},1]$, 
\[
\E_N\big[ \mathbf 1\{ \X_N(x) \ge \gamma \log N +g\} e^{\zeta \X_N(y) +\alpha \X_{N,\delta}(x)} \big] \le C \frac{\E_N\big[e^{\beta_{N,\delta} \X_N(x)+\zeta \X_N(y)+\alpha \X_{N,\delta}(x)}\big]N^{-\gamma^2} \delta^{-\alpha\gamma}}{\sqrt{\log N}}  . 
\]
Note that by a compactness argument, this bound holds for any possible sequence $\delta(N) \in[N^{ \eta-1},1]$, since 
we have seen that $\beta_{N,\delta} \in [{\gamma}\frac\eta2,\gamma]$ and the bound holds along any subsequence such that $\beta_{N,\delta}\to\beta$ as $N\to\infty$. 

Using the asymptotics of Lemma~\ref{lem:prob}, this proves the claim for $\zeta\neq 0$. For $\zeta=0$, the argument is even simpler -- we do not need to worry about the variable $y$, so $c$ plays no role and indeed the estimates hold for all $N$ and $x\in K$. 
\end{proof}

By combining Proposition~\ref{prop:mixed2} with $\zeta=0$ and the asymptotics \eqref{mom2}, we obtain that for $\alpha\in[0,\gamma]$,  it holds uniformly for $x\in K$ and $\delta \ge N^{-1+\eta}$, 
\[\begin{aligned}
\frac{\E_N\big[\mathbf 1\{\X_N(x) \ge \gamma \log N +g(x) \}  e^{\alpha \X_{N,\delta}(x)}
\big]}{\P_N[\X_N(x) \ge \gamma \log N]}
&\le C \frac{\E_N\big[e^{\beta_{N,\delta}\X_N(x)+\alpha \X_{N,\delta}(x)}\big]}{\E_N\big[e^{\gamma \X_N(x)}\big]} \delta^{-\alpha\gamma}  \\
&\le C  N^{(\beta_{N,\delta}^2-\gamma^2)/{2}} \delta^{-\alpha-\beta_{N,\delta}\alpha^2/2-\alpha\gamma}.
\end{aligned}\] 
In particular, 
$(\beta_{N,\delta}^2-\gamma^2)\log N =  (2\alpha\gamma - \alpha^2 \kappa)\log \delta +O(1)$
where $\kappa = \frac{\log \delta^{-1}}{\log N} \in [0,1]$, 
so that for $\alpha\in[0,\gamma]$, any $k\le L_N$ and $x\in K$, 
\[
\frac{\E_N\big[\mathbf 1\{\X_N(x) \ge \gamma \log N +g(x) \}  e^{\alpha \X_{N,e^{-k}}(x)}
\big]}{\P_N[\X_N(x) \ge \gamma \log N]}
\le C e^{k(\alpha^2\frac{1+\kappa}{2}+\alpha\beta_{N,\delta})} 
\le    C e^{k(\frac{\alpha^2}{2}+\gamma\alpha)}  . 
\]

Using this bound and \eqref{L1}, choosing again $\alpha=\eta$,   we conclude that 
\[
\E_N[ \nu_N(\1_K\1_{\mathcal{B}_{\ell , L_N}^{\bf c}})]
\le C  \sum_{k=\ell}^{L_N}  e^{k(\frac{\alpha^2}{2}-\alpha\eta)}  \le Ce^{-\frac{\eta^2}{2}\ell} .
\]
This completes the proof of Proposition~\ref{prop:L1}.

\subsection{Mesoscopic regime: Proof of Proposition~\ref{prop:meso}}
\label{sec:meso}
The proof is divided in three subsections where we obtain the relevant bounds for  $\Theta_{j,N,\ell}$
for $j \in \{1,2,3\}$ respectively. 

\subsubsection{Bound for  $\Theta_{1}$} \label{sec:bd1} Since $g$ is locally uniformly bounded, the bound for $\Theta_{1,N,\ell}$ we are after will follow from bounding  
\[
\E_N [\mathbf 1_{\mathcal{B}_{\ell , L_N}(x)} \mathbf 1_{ \mathcal{B}_{\ell , L_N}(y)}  \mu_N(x) \mu_N(y)].
\]

Our strategy will be to replace the barriers here by a single judiciously chosen bound which we then estimate with Markov's inequality. More precisely, let us take $\delta=e^{-k}$ with $k=\lceil \log |x-y|^{-1}\rceil$ and then note that  
for $x,y\in K$ and any $\alpha>0$,  
\begin{align*}
\E_N [\mathbf 1_{\mathcal{B}_{\ell , L_N}(x)} \mathbf 1_{ \mathcal{B}_{\ell , L_N}(y)}  \mu_N(x) \mu_N(y)]
&\leq \frac{\E_N[\mathbf 1\{\X_{N,e^{-k}}(y)\leq (\gamma+\eta)k\}e^{\gamma \X_N(x)}e^{\gamma \X_N(y)}]}{\E_N [e^{\gamma \X_N(x)}]\E_N [e^{\gamma \X_N(y)}]}\\
&\leq   \frac{\E_N[e^{-\alpha(\X_{N,\delta}(y)-(\gamma+\eta)k)}e^{\gamma \X_N(x)}e^{\gamma \X_N(y)}]}{\E_N [e^{\gamma \X_N(x)}]\E_N [e^{\gamma \X_N(y)}]}\\
&=  \frac{\E_N\big[ e^{\gamma \X_N(x) + \gamma \X_N(y) - \alpha \X_{N, \delta}(y) } \big]}{\E_N\big[ e^{\gamma \X_N(x)}\big] \E_N\big[e^{\gamma \X_N(y}\big]}
\delta^{-\alpha(\gamma+\eta)} . 
\end{align*}

By definitions, we rewrite
\[
\frac{\E_N\big[ e^{\gamma \X_N(x) + \gamma \X_N(y) - \alpha \X_{N, \delta}(y) } \big]}{\E_N\big[ e^{\gamma \X_N(x)}\big] \E_N\big[e^{\gamma \X_N(y)}\big]}= 
\frac{ \Psi_{N,\delta}^{(\gamma,\gamma,\alpha)}(x,y;y) }{ \Psi_{N}^{(\gamma)}(x)\Psi_{N}^{(\gamma)}(y)  }
e^{\gamma^2 C_\X(x,y)+\frac{\alpha^2}{2}C_{\X,\delta,\delta}(y,y)-\alpha\gamma(C_{\X,\delta}(y,y)+C_{\X,\delta}(y,x))}.
\]
From Assumption \ref{ass:meso}, we see that the ratio of the $\Psi$-functions converges to 1 (uniformly in all parameters), while using Assumption~\ref{ass:cov} and the definition of~$k$, 
\begin{equation}\label{covbd}
C_\X(x,y), C_{\X,\delta}(y,y)  , C_{\X,\delta}(y,x) , C_{\X,\delta,\delta}(y,y) = k +O(1)
\end{equation}
with different errors (these errors are  bounded  
functions of $(x,y) \in K$ and the convergence holds uniformly if $\delta\to0$). 
This implies that for $\beta , \gamma , \alpha \in [0,\sqrt{2d}]$, 
\begin{align} \label{supmom}
\frac{\E_N\big[ e^{\beta \X_N(x) + \gamma \X_N(y) - \alpha \X_{N, \delta}(y) } \big]}{\E_N\big[ e^{\gamma \X_N(x)}\big] \E_N\big[e^{\gamma \X_N(y)}\big]}
\le C \delta^{-\gamma\beta-\frac{\alpha^2}{2}+\alpha(\gamma+\beta)} N^{\frac{\beta^2-\gamma^2}{2}} . 
\end{align}
Choosing $\beta=\alpha=\gamma$, we obtain 
\[
\frac{\E_N\big[ e^{\gamma \X_N(x) + \gamma \X_N(y) - \alpha \X_{N, \delta}(y) } \big]}{\E_N\big[ e^{\gamma \X_N(x)}\big] \E_N\big[e^{\gamma \X_N(y)}\big]}
\le C \delta^{\frac{\gamma^2}{2}} , 
\]
then we conclude that for $x,y\in K$ with $ \delta \simeq |x-y|\geq N^{-1+\eta}$, 
\[
\E_N [\mathbf 1_{\mathcal{B}_{\ell , L_N}(x)} \mathbf 1_{ \mathcal{B}_{\ell , L_N}(y)}  \mu_N(x) \mu_N(y)]
\le C \delta^{-\gamma(\gamma+\eta)+\frac{\gamma^2}{2}}
\leq C|x-y|^{-\frac{(\gamma+\eta)^2-\eta^2}{2}}.
\]
This estimate directly implies \eqref{Thetaest} for $\Theta_1$. 


\subsubsection{Bound for  $\Theta_{2}$.} \label{sec:bd2}
Recall that our goal is to bound 
\[
\E_N\left[\mathbf 1_{\mathcal{B}_{\ell , L_N}(x)} \mathbf 1_{ \mathcal{B}_{\ell , L_N}(y)}  \nu_N(x) \nu_N(y)\right]
\qquad\text{where}\quad
\nu_N(x) = \frac{\1\{\X_N(x) \ge \gamma \log N +g(x)\}}{\P_N[\X_N(x) \ge \gamma \log N]}
\]
and $g\in \mathcal{C}(\Omega)$. The strategy is analogous to that used in the previous section; we replace the barriers by $e^{-\gamma(\X_{N,\delta}(x)-(\gamma+\eta)k)}$ with the same choice of $k$ and $\delta=e^{-k}$ 
as in the previous section. The required estimate is then provided by the following result.

\begin{proposition} \label{prop:control2}
For any $\gamma,\eta>0$ and any compact set $K\subset \Omega$, there exists a constant $C=C_{\gamma,\eta,K}$ such that for  $x,y\in K$ with $ |x-y| \ge N^{-1+\eta}$ and $k=\lceil \log |x-y|^{-1}\rceil$, 
\[
\E_N\big[ \nu_N(x) \nu_N(y)  e^{- \gamma  \X_{N, e^{-k}}(x) }  \big] 
\le C |x-y|^{\frac{\gamma^2}{2}} . 
\]
\end{proposition}

Proposition~\ref{prop:control2} provides control of  $\Theta_{2}$ on mesoscopic scales.  
Just as in the previous section, we deduce from this proposition the estimate \eqref{Thetaest} for $\Theta_2$;
for $x,y\in K$ with $e^{-\ell} \geq  |x-y|  \geq N^{-1+\eta}$, 
\[ \begin{aligned}
\Theta_{2,N,\ell}(x,y) 
& \le C k^{\gamma(\gamma+\eta)}  \E_N\big[ \nu_N(x) \nu_N(y)  e^{- \gamma  \X_{N, e^{-k}}(x) }  \big]  \\
& \le C  |x-y| ^{-\frac{\gamma^2}{2}-\eta\gamma} .
\end{aligned}\]

In order to prove Proposition~\ref{prop:control2}, we rely on the following result.

\begin{lemma} \label{lem:fzero}
Assume that $\widehat{\mathrm{f}_N}$ is the Fourier transform of some probability measure $\mathrm f_N$ on $\R^2$: $\widehat{\mathrm  f_N}(\xi)=\int_{\R^2} e^{-i\xi\cdot u}\mathrm f_N(du)$, and can be written in the form
\[
\widehat{\mathrm{f}_N}(\xi_1,  \xi_2) = F_N(\xi) \exp\big( - \tfrac{\xi_1^2+ \xi_2^2}{2} - \xi_1\xi_2 \kappa_N \big)  
\]
for some $\kappa_N\in [0,1]$.   
We assume further that $\kappa:=\limsup_{N\to\infty} \kappa_N<1$. 

Let $R>0$ and $\chi \in L^1(\R \to\R)$ be a function with $\operatorname{supp}(\widehat\chi) \subset [-R,R]$. Moreover,  let  $\gamma_1,\gamma_2>0$ (possibly depending on $N$) be such that $\gamma_1,\gamma_2>\gamma$ for some fixed $\gamma>0$ and set $\phi_j = (e^{- \gamma_j \cdot}\1_{\R_+})\star \chi$. Then we have the following bound: for any sequence of strictly positive finite numbers $\sigma_N$, we have
\[
\limsup_{N\to\infty}\left|\int_{\R^2} \sigma_N^2\phi_1(\sigma_N u_1) \phi_2(\sigma_N u_2) \mathrm f_N(du)\right|\leq \frac{1}{2\pi \gamma^2}\|\widehat \chi\|_\infty^2 \frac{1}{\sqrt{1-\kappa^2}} \sup_{N\geq 1}\|F_N\|_{\infty,[-R\sigma_N,R\sigma_N]^2} . 
\]
\end{lemma}

\begin{proof}
Fourier transforming, we find by Plancherel's formula, that the integral in question equals  
\[
I_N =  \frac{1}{(2\pi)^2} \int_{\R^2}  \widehat\chi(\xi_1\sigma_N^{-1})  \widehat\chi(\xi_2\sigma_N^{-1})  \widehat{\mathrm{f}_N}(\xi_1,  \xi_2) \frac{d \xi_1}{\gamma_1 + i\xi_1 \sigma_N^{-1}} \frac{d \xi_2}{ \gamma_2+ i\xi_2 \sigma_N^{-1}} . 
\]
Since $\widehat\chi$ is continuous and has compact support in $[-R,R]$, using our assumption about the form of $\widehat{\mathrm{f}_N}$, we find
\[
|I_N| \leq \frac{1}{\gamma^2(2\pi)^2}\|\widehat \chi\|_\infty^2 \|F_N\|_{\infty,[-R\sigma_N,R\sigma_N]^2}\int_{\R^2} e^{-\frac{\xi_1^2+\xi_2^2}{2}-\xi_1\xi_2\kappa_N}d\xi_1d\xi_2.  
\]
The integrand on the RHS is the unnormalized density of a centered Gaussian distribution with covariance matrix $\Sigma = \left( \begin{smallmatrix}1 &\kappa_N \\ \kappa_N &1  \end{smallmatrix}\right)^{-1}$, so the integral equals $2\pi \sqrt{\det \Sigma}=2\pi /\sqrt{1-\kappa_N^2}$. We thus obtain the bound 
\[
\limsup_{N\to\infty}|I_N|\leq \frac{1}{2\pi \gamma^2}\|\widehat \chi\|_\infty^2 \frac{1}{\sqrt{1-\kappa^2}}\sup_{N\geq 1} \|F_N\|_{\infty,[-R\sigma_N,R\sigma_N]^2},
\]
which concludes the proof.
\end{proof}

\begin{proof}[Proof of Proposition~\ref{prop:control2}]
Let  $\gamma_1= \gamma + \frac{g(x)}{\log N}$, $\gamma_2= \gamma + \frac{g(y)}{\log N}$.
We assume that $N$ is sufficiently large so that $\gamma_1 , \gamma_2 \in (c,\sqrt{2d})$ for $x,y\in K$ where $c>0$ is a fixed small constant. 
Then, we define a new probability measure $\Q_{N,x,y}$ by
\[
\frac{d\Q_{N,x,y}}{d\P_N} :=  \frac{ e^{ \gamma_1  \X_N(x) +\gamma_2  \X_N(y) - \gamma  \X_{N,\delta}(x)}}{ \E_N\big[ e^{ \gamma_1  \X_N(x) +\gamma_2  \X_N(y) - \gamma  \X_{N, \delta}(x)} \big]} ,\qquad
\]
where $k$ and $\delta=e^{-k}$ are as in the statement of the proposition.

For convenience, we also denote $\sigma_N = \sqrt{\log N}$ and $(X_1,X_2) = \big(
\frac{ \X_N(x) - \gamma_1 \sigma_N^2}{\sigma_N} ,
\frac{ \X_N(y) - \gamma_2 \sigma_N^2}{\sigma_N}\big)$. 
Then, we can rewrite
\[\begin{aligned}
J_{N,x,y} :&=  \E_N\big[ \mathbf 1\{\X_N(x) \ge \gamma \log N +g(x) , \X_N(y) \ge \gamma \log N + g(y) \} e^{- \gamma  \X_{N, \delta}(x) } \big]   \\
&=   \E_N\big[ e^{ \gamma_1  \X_N(x) +\gamma_2  \X_N(y) - \gamma  \X_{N, \delta}(x)} \big]
\Q_{N,x,y} \big[ \{X_1 \ge 0 , X_2 \ge 0 \}  e^{- \gamma_1  \X_N(x) -\gamma_2  \X_N(y)}  \big] \\
&=   \E_N\big[ e^{ \gamma_1  \X_N(x) +\gamma_2  \X_N(y) - \gamma  \X_{N, \delta}(x)} \big]e^{-(\gamma_1^2 + \gamma_2^2)  \sigma_N^2 } \int_{[0,\infty]^2} \hspace{-.3cm} e^{-  \gamma_1 \sigma_N u_1 -  \gamma_2 \sigma_N u_2 } \mathrm{f}_N(du) , 
\end{aligned}\]
where $ \mathrm{f}_N$ is the law of $(X_1,X_2)$ under $\Q_{N,x,y}$.

Using \eqref{mom1}, Assumption \ref{ass:meso} and \eqref{covbd} as in Section \ref{sec:bd1} (recall also that  $k=\lceil \log |x-y|^{-1}\rceil$ and $\delta=e^{-k}$), we find that for $x,y\in K$ with $|x-y|\geq N^{-1+\eta}$, 
\begin{align*}
\E_N\big[ e^{ \gamma_1  \X_N(x) +\gamma_2  \X_N(y) - \gamma  \X_{N, \delta}(x)} \big]&\leq C e^{\frac{\gamma_1^2+\gamma_2^2}{2}\log N} |x-y|^{-\gamma_1\gamma_2} \delta^{-\frac{\gamma^2}{2}+\gamma(\gamma_1+\gamma_2)}\\
& \leq C e^{\frac{\gamma_1^2+\gamma_2^2}{2}\sigma_N^2}  |x-y|^{\frac{\gamma^2}{2}} , 
\end{align*}
since $\gamma(\gamma_1+\gamma_2)-\gamma_1\gamma_2 = \gamma^2 +O(\sigma_N^{-2}) $. 
Thus, we obtain
\begin{equation} \label{prob1}
J_{N,x,y} \le C    e^{-\frac{\gamma_1^2 + \gamma_2^2}{2}  \sigma_N^2} |x-y|^{\frac{\gamma^2}{2}} 
\int_{[0,\infty]^2} \hspace{-.3cm} e^{- \gamma_1 \sigma_N u_1 -  \gamma_2\sigma_N u_2 } \mathrm{f}_N(du)  . 
\end{equation}

Now, we would like to apply Lemma \ref{lem:fzero} in order to estimate this quantity. This requires the  asymptotics of the characteristic function
\[
\widehat{\mathrm{f}_N}(\xi_1,  \xi_2) =  \Q_{N,x,y}\big[ e^{-i(\xi_1 X_1+ \xi_2 X_2)}\big] , \qquad  \xi\in\R^2. 
\]
Namely, with $\zeta_j = \gamma_j - i \xi_j \sigma_N^{-1}$, we find from \eqref{mom1} and \eqref{kernel} that
\[ \begin{aligned}
\widehat{\mathrm{f}_N}(\xi_1,  \xi_2)
&= \frac{e^{ i \sigma_N(\gamma_1\xi_1+ \gamma_2\xi_2)}\E_N\big[ e^{ \zeta_1 \X_N(x) + \zeta_2 \X_N(y) - \gamma  \X_{N, \delta}(x) } \big]}{\E_N\big[ e^{ \gamma_1  \X_N(x) +\gamma_2  \X_N(y) - \gamma  \X_{N, \delta}(x)} \big]} \\
& = \frac{ \Psi_{N,\delta}^{(\zeta_1,\zeta_2,-\gamma)}(x,y;x)}{ \Psi_N^{(\gamma_1,\gamma_2,-\gamma)}(x,y;x)} \exp\left(\frac{1}{2}(\zeta_1^2-\gamma_1^2+\zeta_2^2-\gamma_2^2)\log N+ i \sigma_N(\gamma_1\xi_1+ \gamma_2\xi_2)\right)\\
&\quad \times \exp\big((\zeta_1\zeta_2-\gamma_1\gamma_2)C_\X(x,y)-\gamma(\zeta_1-\gamma_1)C_{\X,\delta}(x,x)-\gamma(\zeta_2-\gamma_2)C_{\X,\delta}(x,y)\big) .\\
& =  F_{N}(\xi_1,\xi_2) \exp\big( - \tfrac{\xi_1^2+ \xi_2^2}{2} - \xi_1\xi_2 \kappa_N \big)  
\end{aligned}\]
where $\kappa_N = \tfrac{k}{\sigma_N^2}$  and $F_{N}$ is defined implicitly by the previous equation. 

In particular, using \eqref{covbd}, we have 
\[
F_{N}(\xi_1,\xi_2) = \frac{ \Psi_{N,\delta}^{(\zeta_1,\zeta_2,-\gamma)}(x,y;x)}{ \Psi_N^{(\gamma_1,\gamma_2,-\gamma)}(x,y;x)}  \exp\big(O((1+\tfrac{|\xi_1|}{\sigma_N})(1+\tfrac{|\xi_2|}{\sigma_N}))\big) 
\]
where  the ratio of the $\Psi$-functions is uniformly bounded for $\xi_1,\xi_2 \in [-R\sigma_N,R\sigma_N]$ and $x,y\in K$, $|x-y|\geq N^{-1+\eta}$ (cf~Assumption~\ref{ass:meso})  and the implied constant is uniform in $\xi_1,\xi_2$.

This implies that 
\[ \begin{aligned}
\sup_{x,y\in K,|x-y|\geq N^{-1+\eta} }\sup_{\xi_1,\xi_2 \in [-R\sigma_N,R\sigma_N]} \big|F_{N}(\xi_1,\xi_2) \big| \le C , 
\end{aligned}\]
that is, $\mathrm{f}_N$ satisfies all the assumptions of Lemma~\ref{lem:fzero} (note that $0<\kappa_N \le 1-\eta$).

Now, let  $\chi:\R\to \R_+$ be an even mollifier such that $\int_{\R} \chi=1$ and $\widehat\chi$  has compact support. 
As in Lemma~\ref{lem:fzero}, set $\phi_j = (e^{- \gamma_j\cdot}\1_{\R_+})\star \chi$ for  $j=1,2$ and let
\[
I_N = \int_{\R^2} \phi_1(\sigma_N u_1)  \phi_2(\sigma_N u_2)  \mathrm{f}_N(du) 
\qquad\text{and}\qquad
\rho_j(u) = e^{\gamma_j u}\phi_j(u) =\int^u_{-\infty} e^{\gamma_j t} \chi(t) d t . 
\]
Then, observe that
\[ \begin{aligned}
\int_{[0,\infty]^2} \hspace{-.3cm} e^{- \gamma_1 \sigma_N u_1 - \gamma_2\sigma_N u_2 }\rho_1(\sigma_N u_1)\rho_2(\sigma_N u_2) \mathrm{f}_N(du)  \le I_N . 
\end{aligned}\] 
Since $\chi$ is even, we have 
$\displaystyle 
\rho_j(x)  \ge  \int_{0}^\infty e^{-\gamma_j t} \chi(t) d t $ for $x\ge 0$. Upon rescaling $\chi$, we can assume that this integral is at least 1/3 (recall that we assume that $\gamma_j>c>0$).   
These trivial bounds imply that 
\[
\int_{[0,\infty]^2} \hspace{-.3cm} e^{- \gamma_1 \sigma_N u_1 - \gamma_2 \sigma_N u_2 } \mathrm{f}_N(du) \le 9 I_N . 
\]
Hence, by Lemma~\ref{lem:fzero}, we conclude that 
\[
\limsup_{N\to\infty} \sigma_N^2 
\int_{[0,\infty]^2} \hspace{-.3cm} e^{- \gamma_1 \sigma_N u_1 - \gamma_2 \sigma_N u_2 } \mathrm{f}_N(du)  \le C .
\]
Going back to the estimate~\eqref{prob1}, replacing $\sigma_N =\sqrt{\log N}$ and  $\gamma_1= \gamma + \frac{g(x)}{\log N}$, $\gamma_2= \gamma + \frac{g(y)}{\log N}$,  this shows that
\[
J_{N,x,y} = \E_N\big[\mathbf 1 \{\X_N(x) \ge \gamma \log N +g(x) , \X_N(y) \ge \gamma \log N + g(y) \} e^{- \gamma  \X_{N, \delta}(x) } \big] 
\le C\frac{N^{-\gamma^2} |x-y|^{\frac{\gamma^2}{2}}}{\log N} . 
\]
Using the asymptotics from Lemma~\ref{lem:prob}, we conclude that
\[
\E_N\big[ \nu_N(x) \nu_N(y)  e^{- \gamma  \X_{N, \delta}(x) }  \big] 
=  \frac{J_{N,x,y} }{\P_N[\X_N(x) \ge \gamma \log N]   \P_N[\X_N(y) \ge \gamma \log N]  } 
=  \frac{O(J_{N,x,y}) }{N^{-\gamma^2}/\log N} 
\]
and the claim follows from our estimate for $J_{N,x,y}$.
\end{proof}


\subsubsection{Bound for  $\Theta_{3}$.} \label{sec:bd3}
Proceeding like in Subsection~\ref{sec:bd1}, in order to prove \eqref{Thetaest} for $\Theta_3$, by \eqref{Theta} and Markov's inequality, it suffices to obtain the following estimates;

\begin{proposition} \label{prop:control3}
For any $\gamma,\eta>0$ and any compact set $K\subset \Omega$, there exists a constant $C=C_{\gamma,\eta,K}$ such that for  $x,y\in K$ with $|x-y| \ge N^{-1+\eta}$ and $k=\lceil \log|x-y|^{-1}\rceil$, 
\[
\E_N\big[ \nu_N(x) \mu_N(y)  e^{- \gamma  \X_{N, e^{-k}}(y) }  \big] 
\le C_{\gamma,\eta,K} |x-y|^{\frac{\gamma^2}{2}}  .  
\]
\end{proposition}

\begin{proof} 
For $x,y \in K$, we define a new measure by  
\[
\frac{d \Q_{N,x,y}}{d\P_N} 
= \frac{e^{\gamma(\X_N(y) - \X_{N, \delta}(y)) } }{\E_N\big[e^{\gamma(\X_N(y) - \X_{N, \delta}(y))}\big]}, 
\]
where $k$ is as in the statement of the proposition and set $\delta=e^{-k}$ (in particular, the dependence on $x$ of this measure comes through $\delta$). Then, we can rewrite 
\begin{equation}\label{mixed3}
\E_N\big[ \nu_N(x) \mu_N(y)  e^{- \gamma  \X_{N, \delta}(y) }  \big]  
= \frac{\Q_{N,x,y}[ \X_N(x) \ge \gamma \log N + g(x)]}{\P_N[\X_N(x) \ge \gamma \log N]}\frac{\E_N\big[e^{\gamma(\X_N(y) - \X_{N, \delta}(y))}\big]}{\E_N[e^{\gamma \X_N(y)}]} . 
\end{equation} 

We use the method described in the Appendix~\ref{sec:1dapprox} to compute the first factor on the RHS of \eqref{mixed3}. 
To this end,
we let $X_N = \frac{\X_N(x)-\gamma\log N}{\sqrt{\log N}}$ and obtain the asymptotics of the characteristic function of the random variable $X_N$ under the measure $\Q_{N,x,y}$ biased by $e^{\beta \X_N(x)}$, that is, for $\chi\in\R$, 
\[
\frac{\Q_{N,x,y}\big[e^{i\chi X_N+\beta \X_N(x)}\big]}{\Q_{N,x,y}\big[e^{\beta \X_N(x)}\big]}
= \frac{\E_N\big[e^{(\beta+ i\chi/\sigma_N)\X_N(x)+\gamma(\X_N(y) - \X_{N, \delta}(y))}\big]}{\E_N\big[e^{\beta \X_N(x)+ \gamma(\X_N(y) -  \X_{N, \delta}(y))}\big]} 
e^{-i\chi\gamma\sigma_N}
\]
where $\sigma_N = \sqrt{\log N}$. 
This ratio of exponential moments can again be computed using \eqref{mom1},
\[
\frac{\Q_{N,x,y}\big[e^{i\chi X_N+\beta \X_N(x)}\big]}{\Q_{N,x,y}\big[e^{\beta \X_N(x)}\big]}
= \frac{\Psi_{N,\delta}^{(\beta+ i\chi/\sigma_N,\gamma,-\gamma)}(x,y;y)}{\Psi_{N,\delta}^{(\beta,\gamma,-\gamma)}(x,y;y)}
\exp\big(-\tfrac{\chi^2}2 +i\chi (\beta-\gamma)\sigma_N+i\gamma\chi \tfrac{C_\X(x,y) -C_{\X,\delta}(y,x)}{\sigma_N} \big) .
\]
We choose $\beta = \gamma + \frac{C_\X(x,y) -C_{\X,\delta}(y,x)}{\sigma_N^2} = \gamma + O(\sigma_N^{-2})$ since the log singularity cancels (cf.~Assumption~\ref{ass:cov} with $\delta= e^{-\lceil \log|x-y|^{-1}\rceil}$). 
Then, by Assumption \ref{ass:meso}, we have for small enough $\mathcal W>0$, $\zeta\in \mathcal{S} = \{z \in\C  :  |\Re z| < \mathcal W\} $, 
\[
\frac{\Psi_{N,\delta}^{(\beta+\zeta,\gamma,-\gamma)}(x,y;y)}{\Psi_{N,\delta}^{(\beta,\gamma,-\gamma)}(x,y;y)} \sim \frac{\Psi(\gamma +\zeta,x)}{\Psi(\gamma,x)}\qquad\text{as $N\to\infty$.}
\]
Hence, the characteristic function of $X_N$ under $\Q_{N,x,y}$ satisfies the condition \eqref{ass1} with $\epsilon_N =1/\sigma_N$ and limit $ \psi_{x,y}(z) = \frac{ \Psi^{(\gamma+z)}(x)}{\Psi^{(\gamma)}(x) }$. 
We emphasize that this limit is locally uniform in $\{ (x,y) \in K^2 : |x-y| \ge N^{-1+\eta} \}$ and $\{z \in\C  :  |\Re z| \le \mathcal W\}$ for a small $\mathcal W>0$. 
Hence, we can apply Lemma~\ref{lem:1} with $\beta=\gamma$; we obtain
\[
\Q_{N,x,y}[ \X_N(x) \ge \gamma \log N + g(x)]
=\frac{\Q_{N,x,y}\big[e^{\beta\X_N(x)}\big]}{\gamma \sqrt{2\pi \log N}} N^{-\gamma\beta} \big( e^{-\gamma g(x)}+ \underset{N\to\infty}{o(1)} \big) . 
\]
Using also formula \eqref{prob3}, by \eqref{mixed3} and the definition of the measure $\Q_{N,x,y}$, this shows that
\[ \begin{aligned}
\E_N\big[ \nu_N(x) \mu_N(y)  e^{- \gamma  \X_{N, \delta}(y) }  \big]  
& =  \underset{N\to\infty}{O(1)}  N^{\gamma(\gamma-\beta)}   \frac{\Q_{N,x,y}\big[e^{\beta\X_N(x)}\big]}{\E_N\big[e^{\gamma\X_N(x)}\big]} \frac{\E_N\big[e^{\gamma(\X_N(y) - \X_{N, \delta}(y))}\big]}{\E_N[e^{\gamma \X_N(y)}]} \\
& =  \underset{N\to\infty}{O(1)} N^{\gamma(\gamma-\beta)}  \frac{\E_N\big[ e^{\beta \X_N(x) + \gamma \X_N(y) - \gamma  \X_{N, \delta}(y) } \big]}{\E_N\big[ e^{\gamma \X_N(x)}\big] \E_N\big[e^{\gamma \X_N(y)}\big]} . 
\end{aligned}\]
Finally, using \eqref{supmom} with $\beta=\gamma + O(\sigma_N^{-2})$ and $\alpha=\gamma$, one obtains
\[
\E_N\big[ \nu_N(x) \mu_N(y)  e^{- \gamma  \X_{N, \delta}(y) }  \big]   \le C \delta^{\frac{\gamma^2}{2}+ O(\sigma_N^{-2})} N^{\frac{(\beta-\gamma)^2}{2}}
\le C  \delta^{\frac{\gamma^2}{2}}
\]
where we used that $\delta= e^{-\lceil \log|x-y|^{-1}\rceil}$ and  $|x-y| \ge N^{-1+\eta}$.   
This yields the claim.
\end{proof}


\subsection{Macroscopic regime: Proof of Proposition~\ref{prop:macro}} 
\label{sec:macro}

We turn to Step 3 of the proof which is the study of the limits of $\Theta_{j, N,\ell}$ for $j\in\{1,2,3\}$.
The proof of Proposition~\ref{prop:macro} is divided in four steps where we control the relevant asymptotics  of $\Theta_{j,N,\ell}$ and $ \Upsilon_{j,N}$ for $j\in\{1,2\}$ respectively.
The definitions of $ \Upsilon_{j,N}$ and the choice of $L$ are made in Subsection~\ref{sec:T3}.

\subsubsection{Upper-bound for $\Theta_{1}$} \label{sec:T1}
The goal of this subsection is to obtain \eqref{Tmacro} for $\Theta_{1}$. 
Let us denote by $\mathbb{P}^1_{N,x,y}$ the probability measure with density proportional to  $ e^{\gamma \X_N(x)+\gamma \X_N(y)}$ with respect to $\P_N$ and observe that by \eqref{Theta} and \eqref{W2}, it holds for $N$ sufficiently large, 
\begin{equation} \label{UB1}
\begin{aligned}
\Theta_{1,N,\ell} (x,y)  &=  \mathbb{P}^1_{N,x,y}\big[\mathcal{B}_{\ell , L_N}(x),\mathcal{B}_{\ell , L_N}(y)\big] \E_N\big[\mu_N(x) \mu_N(y) \big] e^{-\gamma g(x) -\gamma g(y)}    \\
&\le  \mathbb{P}^1_{N,x,y}\big[\mathcal{B}_{\ell,L}(x),\mathcal{B}_{\ell ,L}(y)\big]    e^{\gamma^2C_\X(x,y)-\gamma g(x) -\gamma g(y)}  \big(1+\underset{N\to\infty}{o(1)} \big) 
\end{aligned}
\end{equation}
for any fixed $\ell,L \in\N$ with $L>\ell$.

This already yields the bound \eqref{macrobd}; for $x,y\in K$ with $|x-y| \ge e^{-\ell}$, 
\begin{equation*}
\Theta_{1,N,\ell} (x,y)  \le C_{\ell}. 
\end{equation*}

The next proposition provides the required pointwise limit of the RHS of \eqref{UB1}. 
To state the result, we need the following definitions; 

\begin{definition} \label{def:mS}
Let $L\in\N$ and set $(z_{2j-1},\delta_{2j-1}) = (x,e^{-j})$, $(z_{2j},\delta_{2j}) = (y,e^{-j})$ for $j\in [L]$.
Let $\mathrm{f}_{\delta,z} := \sum_{j=1}^q\xi_j \rho_{\delta_j,z_j}$ for $\xi \in \R^q$ with $q=2L$.
For $x,y\in\Omega$ with $x\neq y$, we define a positive-definite matrix $\Sigma_{x,y}\in \R^{q\times q}$ by 
\[
\E \langle \X , \mathrm{f}_{\delta,z}  \rangle^2 = 
\langle \xi ,\Sigma_{x,y}\xi \rangle .
\]
Similarly we define $\mathbf{m}_{x,y} \in \R^q$ by 
\[
\gamma\int \big( C_\X(x,u) +  C_\X(y,u) \big)   \mathrm{f}_{\delta,z}(u) d u =
\langle \xi ,\mathbf{m}_{x,y}\rangle
\]
Note that $(x,y)\mapsto \Sigma_{x,y}, \mathbf{m}_{x,y} $ are both continuous functions on $\Omega\times\Omega$ which can be written explicitly in terms of the kernels \eqref{kernel}.
\end{definition}

We will use this notation throughout this section and we let, for $x,y\in\Omega$ with $x\neq y$, 
\begin{equation} \label{def:Y}
\mathbf{Y} = \big( \X_{N, e^{-k}}(x) , \X_{N, e^{-k}}(y) \big)_{k=1}^{L} . 
\end{equation}
We have the following result.

\begin{proposition} \label{prop:T1lim}
For $x,y\in\Omega$ with $x\neq y$, under $\mathbb{P}^1_{N,x,y}$, the random vector $\mathbf{Y}$ converges in distribution to a multivariate Gaussian law
$\mathcal{N}(\mathbf{m}_{x,y}, \Sigma_{x,y})$. 
\end{proposition}

\begin{proof}
By definitions, the Laplace transform of $\mathbf{Y}$ under the measure $\mathbb{P}^1$ is given by 
\[
\mathbb{P}^1_{N,x,y}\big[ e^{\xi \cdot  \mathbf{Y} } \big] = 
\frac{\E_N\big[  e^{\gamma \X_N(x)+\gamma \X_N(y) +\langle \X_N , \mathrm{f}_{\delta,z}  \rangle} \big]}{\E_N[e^{\gamma \X_N(x)+\gamma \X_N(y)}]} , \qquad \xi\in\R^{q} . 
\]
where $q=2L$, $\delta\in(0,1)^q$ and $z\in\Omega^q$ are given according to Definition~\ref{def:mS}.
Moreover, according to \eqref{mom1}, we can rewrite 
\[
\mathbb{P}^1_{N,x,y}\big[ e^{\xi \cdot  \mathbf{Y} } \big] = 
\frac{\Psi_{N,\delta}^{(\gamma,\gamma,\xi)}(x,y;z)}{\Psi_N^{(\gamma,\gamma)}(x,y)} 
\exp\Big(\gamma \langle \xi ,\mathbf{m}_{x,y}\rangle+ \langle \xi ,\Sigma_{x,y}\xi \rangle/2 \Big)
\]
From Assumption \ref{ass:meso}, we see that this ratio of the $\Psi$-functions converges to 1 (for any fixed $x,y\in\Omega$ with $x\neq y$ and $\xi\in\R^{q}$).
Hence, this implies that for any $\xi\in\R^{q} $
\begin{equation} \label{limitL}
\lim_{N\to\infty}\mathbb{P}^1_{N,x,y}\big[ e^{\xi \cdot  \mathbf{Y} } \big] =  
\E_{\mathcal{N}(\mathbf{m}_{x,y}, \Sigma_{x,y})}\big[ e^{\xi \cdot X } \big] . 
\end{equation}
Since the pointwise convergence of the Laplace transform implies convergence in distribution, this completes the proof. 
\end{proof}

By Proposition \ref{prop:T1lim}, we see that  for any fixed $L, \ell \in\N$ and $x,y\in\Omega$ with $x\neq y$, as $N\to\infty$, 
\[\begin{aligned}
\mathbb{P}^1_{N,x,y}\big[\mathcal{B}_{\ell , L}(x),\mathcal{B}_{\ell , L}(y)\big] 
& = \mathbb{P}^1_{N,x,y} \big[ \mathbf{Y} _{2k-1} , \mathbf{Y} _{2k}   \le ( \gamma + \eta) k : k\in [\ell, L]  \big] \\
& \to   \P_{\mathcal{N}(\mathbf{m}_{x,y}, \Sigma_{x,y})}\big[ X_{2k-1} , X_{2k}   \le ( \gamma + \eta) k : k\in [\ell, L]  \big] .
\end{aligned}\]
Then, with \eqref{UB1}, this completes the proof of \eqref{Tmacro} in case $j=1$. 


\subsubsection{Upper-bound for $\Theta_{2,N,\ell}$} \label{sec:T2}
We proceed as in Section~\ref{sec:T1} and introduce a new probability measure $\mathbb{P}^2_{N,x,y}$ given by for $x,y \in\Omega$, 
\[
\frac{d\mathbb{P}^2_{N,x,y}}{d\P_N} =  \frac{\1\{\X_N(x) \ge \gamma \log N +g(x)\}\1\{\X_N(y) \ge \gamma \log N +g(y)\}}{\P_N\big[ \X_N(x) \ge \gamma \log N +g(x) ,  \X_N(y) \ge \gamma \log N +g(y)  \big]}  .
\]
Using this notation, it holds for any fixed $L\ge \ell$  (and $N$ sufficiently large),
\begin{equation} \label{UB2}
\begin{aligned}
\Theta_{2,N,\ell}(x,y) 
& \le\E_N\big[ \mathbf 1\{\mathcal{B}_{\ell , L}(x), \mathcal{B}_{\ell , L}(y)\}  \nu_N(x) \nu_N(y) \big] \\
& = \mathbb{P}^2_{N,x,y}\big[\mathcal{B}_{\ell , L}(x), \mathcal{B}_{\ell , L}(y) \big]
\E_N\big[ \nu_N(x) \nu_N(y)  \big] . 
\end{aligned}
\end{equation}

Our next proposition provides the asymptotics of the quantity on the RHS of \eqref{UB2}. Recall the notation \eqref{eq:Ac}.

\begin{proposition} \label{prop:T2lim}
For $x\neq y$, under $\mathbb{P}^2_{N,x,y}$, the random vector  \eqref{def:Y} converges in distribution as $N\to\infty$ to a multivariate Gaussian $\mathcal{N}(\mathbf{m}_{x,y}, \Sigma_{x,y})$. Moreover, for any $c>0$, we have uniformly for $(x,y) \in \mathcal{A}_c$, 
\begin{equation} \label{VN}
\lim_{N\to\infty} \E_N\big[ \nu_N(x) \nu_N(y)  \big]  =  e^{\gamma^2 C_\X(x,y)} e^{- \gamma g(x) - \gamma g(y)} . 
\end{equation}
\end{proposition}

By \eqref{UB2}, Proposition~\ref{prop:T2lim} immediately implies that \eqref{Tmacro} in case $j=2$, that is,
\[
\limsup_{N\to\infty} \Theta_{2,N,\ell}(x,y) \le 
\P_{\mathcal{N}(\mathbf{m}_{x,y}, \Sigma_{x,y})}\big[ X_{2k-1} , X_{2k}   \le ( \gamma + \eta) k : k\in [\ell, L]  \big]  e^{\gamma^2 C_\X(x,y)} e^{- \gamma g(x) - \gamma g(y)} . 
\]
Moreover using the uniformity of the limit \eqref{VN}, this also establishes the  bound \eqref{macrobd}  in case $j=2$.
The proof of Proposition~\ref{prop:T2lim} is the only step which requires  the Assumption~\ref{ass:macro} and the results of the  Appendix~\ref{sec:2dapprox}. 
This is the most technically involved step in this paper.

\begin{proof}[Proof of Proposition~\ref{prop:T2lim}]
Let us introduce yet new notation. Let
$\sigma_N =\sqrt{\log N}$  and 
\[
X_x = \frac{ \X_N(x) -\gamma \log N -g(x)}{\sqrt{\log N}} ,\qquad x\in\Omega .
\]
Let $q=2L$.
We also consider another probability measure depending on $\xi\in\R^{q}$, 
\begin{equation} \label{Fmeasure}
\frac{d\mathbb{L}_{N,x,y,\xi}}{d\P_N} := \frac{ e^{\gamma \X_N(x)+ \gamma \X_N(y) +\xi \cdot  \mathbf{Y}}  }{ \E_N[e^{\gamma \X_N(x)+ \gamma \X_N(y)+\xi \cdot  \mathbf{Y}} ]} .
\end{equation}

Our strategy is to compute the Laplace transform of the
random vector $\mathbf{Y}$ under the measure $\mathbb{P}^2$, that is, using the previous notation, for $\xi \in\R^{q}$, 
\begin{equation} \label{LaplaceQ}
\begin{aligned}
\mathbb{P}^2_{N,x,y}\big[ e^{\xi \cdot  \mathbf{Y} } \big]  & =
\frac{\E_N\big[\mathbf 1\{X_x\ge 0, X_y \ge 0\}e^{\xi \cdot  \mathbf{Y}} \big]}{\P_N\big[X_x\ge 0, X_y \ge 0\big]}  \\
& = \frac{\mathbb{L}_{N,x,y,\xi}\big[\mathbf 1\{X_x\ge 0, X_y \ge 0\}  e^{-\gamma \sigma_N X_x- \gamma\sigma_N X_y } \big]}{\mathbb{L}_{N,x,y,0}\big[\mathbf 1\{X_x\ge 0, X_y \ge 0\}  e^{-\gamma \sigma_N X_x- \gamma \sigma_N X_y } \big]} 
\mathbb{P}^1_{N,x,y}\big[ e^{\xi \cdot  \mathbf{Y} } \big] 
\end{aligned}
\end{equation}
where $\mathbb{P}^1_{N,x,y} = \mathbb{L}_{N,x,y,0}$ as in Section~\ref{sec:T1}.  

To compute these quantities, we rely on classical Fourier-analytic arguments based on the asymptotics of Assumption~\ref{ass:macro}. 
These arguments are presented in detail in Subsection~\ref{sec:2dapprox}. 

Integrating by parts and making a change of variables, we find
\[ 
\begin{aligned}
& \mathbb{L}_{N,x,y,\xi}\big[\mathbf 1\{X_x\ge 0, X_y \ge 0\}  e^{-\gamma \sigma_N X_x- \gamma\sigma_N X_y } \big]  \\
&\qquad =  \gamma^2\iint_{[0,\infty)^2} 
e^{- \gamma u_1 -  \gamma u_2} \,
\mathbb{L}_{N,x,y,\xi}\big[X_x \in [0, \sigma_N^{-1}u_1] , X_y \in [0, \sigma_N^{-1}u_2]  \big]  d u_1 du_2 .
\end{aligned}
\end{equation*}
To compare with Subsection~\ref{sec:2dapprox}, we use   $\epsilon_N = 1/\gamma \sigma_N$. Then, choosing $L_N : = \log \sigma_N^{c}$ for a fixed $c>\frac{2}{\gamma}$, we have
\begin{equation}\label{F}
\begin{aligned}
& \mathbb{L}_{N,x,y,\xi}\big[\1\{X_x\ge 0, X_y \ge 0\}  e^{- \gamma\sigma_N X_x- \gamma \sigma_N X_y } \big]  \\
& =  \gamma^2\iint_{[0,L_N]^2} e^{- \gamma u_1 -  \gamma u_2}  \,
\mathbb{L}_{N,x,y,\xi}\big[X_x \in [0, \sigma_N^{-1}u_1] , X_y \in [0, \sigma_N^{-1}u_2]  \big]  d u_1 du_2  + \underset{N\to\infty}{O}(\sigma_N^{- \gamma c}). 
\end{aligned}
\end{equation}

Now, we can use the uniform approximation of Proposition~\ref{prop:approx}
to compute the leading term up to an error of order $o(\sigma_N^{-2})$ if we can verify the assumptions \eqref{ass2}--\eqref{ass3} on the characteristic function of $(X_x , X_y)$ under the measure $\mathbb{L}=\mathbb{L}_{N,x,y,\xi}$. 
Using the notation from Definition~\ref{def:mS}, the characteristic function of this random vector is given by for 
$\chi \in \mathcal \R^2$, 
\[
\mathbb{L}\big[ e^{i\chi_1 X_x + i\chi_2 X_y } \big]
=  \frac{\E_N\big[ e^{\zeta_1 \X_N(x)+ \zeta_2 \X_N(y) + \langle \X_N , \mathrm{f}_{\delta,z}\rangle}]}{ \E_N\big[e^{\gamma \X_N(x)+ \gamma \X_N(y) + \langle \X_N , \mathrm{f}_{\delta,z}\rangle}]} e^{-i \frac{\chi_1}{\sigma_N}(\gamma \sigma_N^2+g(x))-i \frac{\chi_2}{\sigma_N}(\gamma \sigma_N^2+  g(y))} 
\]
where  $\zeta_j = \gamma + \tfrac{i\chi_j}{\sigma_N} $  for $j\in\{1,2\}$. 
By  \eqref{mom1}, we can rewrite
\begin{align*}
\mathbb{L}\big[ e^{i\chi_1 X_x + i\chi_2 X_y } \big] = \frac{ \Psi_{N,\delta}^{(\zeta_1,\zeta_2,\xi)}(x,y;z)}{ \Psi_{N,\delta}^{(\gamma,\gamma,\xi)}(x,y;z)}
\exp\bigg(-  \tfrac{\chi_1^2+ \chi_2^2}{2}- C_\X(x,y) \frac{\chi_1 \chi_2}{\sigma_N^2} + i\frac{\chi_1}{\sigma_N} \varpi_1 + i\frac{\chi_2}{\sigma_N} \varpi_2  \bigg) 
\end{align*}
where we set
\[
\varpi_1 := \gamma C_\X(x,y) + \int C_\X(x,u) \mathrm{f}_{\delta,z}(u) d u  - g(x) ,\qquad
\varpi_2 := \gamma C_\X(x,y) + \int C_\X(y,u) \mathrm{f}_{\delta,z}(u) d u  - g(y) , 
\]
and we have used that
\[
\big( \tfrac{\zeta_1^2+ \zeta_2^2}{2} - \gamma^2  \big) \sigma_N^2 -i \gamma\sigma_N (\chi_1+\chi_2)
=  - \tfrac{\chi_1^2+ \chi_2^2}{2}  \qquad \text{and}\qquad
\gamma^2-\zeta_1\zeta_2 =  \tfrac{\chi_1 \chi_2}{\sigma_N^2} -i \gamma \tfrac{\chi_1+\chi_2}{\sigma_N}. 
\]

We claim that this characteristic function satisfies the assumptions of  Subsection~\ref{sec:2dapprox}  with  $\epsilon_N = \sigma_N^{-1}$ and 
\[
\psi_{N,x,y}(\chi_1,\chi_2)
= \frac{\Psi_{N,\delta}^{(\gamma+\chi_1,\gamma+\chi_2,\xi)}(x,y;z)}{\Psi_{N,\delta}^{(\gamma,\gamma,\xi)}(x,y;z)}  e^{ C_\X(x,y)\chi_1 \chi_2 +\chi_1 \varpi_1 + \chi_2 \varpi_2}  , \qquad \chi \in \mathcal{S}^2 
\]
where $\mathcal{S}^2 = \big\{z \in\C^2  :  |\Re z_j| \le \eta ; j\in\{1,2\} \big\}$.  
In particular, Assumption~\ref{ass:macro} guarantees that  \eqref{ass2} holds for a small ${\eta>0}$ 
 uniformly for $\lambda=(x,y) \in \mathcal{A}_c$; the parameters $\xi\in\R^{q}$ and $\delta\in (0,1)^q$ are fixed here. 
Its limit is given by 
\[
\psi_{x,y}(\chi) = \frac{\Psi(\gamma+\chi_1,x)\Psi(\gamma+\chi_2,y)}{\Psi(\gamma,x) \Psi(\gamma,y)}   e^{C_\X(x,y)\chi_1 \chi_2 +\chi_1 \varpi_1 + \chi_2 \varpi_2}  , \qquad \chi \in \mathcal S^2 , 
\]
so that \eqref{ass3} holds; 
\(
|\psi_{x,y}(i \chi )|  \le c e^{|\chi|^\eta} 
\)
for $\chi \in\R^2$ and $(x,y)\in \mathcal{A}_c$ -- here we used that  $\Psi(\gamma,x)> 0$ for $(\gamma, x) \in [0,\sqrt{2d}] \times K$ and $C_\X, \varpi_1,\varpi_2$ are continuous for $(x,y) \in \mathcal{A}_c$.

Hence, by applying Proposition~\ref{prop:approx}, we have uniformly in $u\in [0,L_N]^2$ and $(x,y) \in \mathcal{A}_c$, 
\[
\mathbb{L}\big[X_x \in [0, \sigma_N^{-1}u_1] , X_y \in [0, \sigma_N^{-1}u_2]  \big]  = \P_{\mathcal{N}_{(0,\mathrm{I})}}\big[X_1 \in [0,\sigma_N^{-1}u_1] , X_2\in [0, \sigma_N^{-1}u_2]  \big] + \underset{N\to\infty}{o(\sigma_N^{-2})} . 
\]
Going back to formula \eqref{F}, this implies that uniformly in $(x,y)\in \mathcal{A}_c$, 
\[ \begin{aligned}
& \mathbb{L}\big[\mathbf 1\{X_x\ge 0, X_y \ge 0\}  e^{-\gamma \sigma_N X_x- \gamma\sigma_N X_y } \big]  \\
&\qquad = \gamma^2 \iint_{[0,\infty)^2} e^{- \gamma u_1 - \gamma u_2} \,
\P_{\mathcal{N}_{(0,\mathrm{I})}}\big[X_1 \in [0,\sigma_N^{-1}u_1] , X_2\in [0, \sigma_N^{-1}u_2]  \big]    d u_1 du_2  + \underset{N\to\infty}{o(\sigma_N^{-2})} . 
\end{aligned}\]
Integrating by parts again, we obtain
\[ \begin{aligned}
\sigma_N^2 \mathbb{L}\big[\{X_x\ge 0, X_y \ge 0\}  e^{- \gamma \sigma_N X_x- \gamma\sigma_N X_y } \big] =  \frac{1}{ 2\pi \gamma^2} \iint_{[0,\infty)^2} e^{- u_1 -  u_2}  e^{-\frac{|u|^2}{2\gamma^2\sigma_N^2}}  d u_1 du_2  + \underset{N\to\infty}{o(1)}.
\end{aligned}\]

By Lebesgue's dominated convergence theorem, we conclude that for any $\xi\in\R^{q}$ and uniformly in $(x,y)\in \mathcal{A}_c$
\begin{equation} \label{limitF}
\lim_{N\to\infty}\sigma_N^2 \mathbb{L}_{N,x,y,\xi}\big[\{X_x\ge 0, X_y \ge 0\}  e^{- \gamma \sigma_N X_x- \gamma\sigma_N X_y } \big] =  \frac{1}{ 2\pi \gamma^2}.
\end{equation}

According to formula \eqref{LaplaceQ} and the asymptotics \eqref{limitL}, this shows that for any $\xi \in \R^{q}$ and for $x,y \in\Omega$ with $x \neq y$, 
\[
\lim_{N\to\infty} \mathbb{P}^2_{N,x,y}\big[ e^{\xi \cdot  \mathbf{Y} } \big]  = \lim_{N\to\infty} \mathbb{P}^1_{N,x,y}\big[ e^{\xi \cdot  \mathbf{Y} } \big]  =  \E_{\mathcal{N}(\mathbf{m}_{x,y}, \Sigma_{x,y})}\big[ e^{\xi \cdot X } \big] 
\]
This yields the first claim concerning the convergence in distribution of $\mathbf{Y}$ under $\mathbb{P}^2_{N,x,y}$. 

For the second claim, we can rewrite in a similar way,
\begin{equation*}
\begin{aligned}
& \E_N\big[ \nu_N(x) \nu_N(y)  \big]    = \frac{\P_N\big[X_x\ge 0, X_y \ge 0\big]}{\P_N[\X_N(x) \ge \gamma \log N]\P_N[\X_N(y) \ge \gamma \log N]} \\
&= \frac{\mathbb{L}_{N,x,y,0}\big[\{X_x\ge 0, X_y \ge 0\}  e^{- \gamma\sigma_N X_x- \gamma\sigma_N X_y } \big]}{\P_N[\X_N(x) \ge \gamma \log N]\P_N[\X_N(y) \ge \gamma \log N]}
\E_N[e^{\gamma \X_N(x)+ \gamma \X_N(y)}]e^{-2 \gamma^2\sigma_N^2 - \gamma g(x) - \gamma g(y)} . 
\end{aligned}
\end{equation*}
Using again the asymptotics \eqref{limitF} and Lemma~\ref{lem:prob}, we obtain uniformly in $(x,y)\in \mathcal{A}_c$,  
\[
\lim_{N\to\infty}\frac{\mathbb{L}_{N,x,y,0}\big[\{X_x\ge 0, X_y \ge 0\}  e^{- \gamma\sigma_N X_x- \gamma\sigma_N X_y } \big]}{\P_N[\X_N(x) \ge \gamma \log N]\P_N[\X_N(y) \ge \gamma \log N]}
e^{-\gamma^2\sigma_N^2} =  \frac{1}{\Psi(\gamma,x)\Psi(\gamma,y)} . 
\]
On the other hand, by Assumption \ref{ass:meso}, we also have uniformly for $(x,y)\in \mathcal{A}_c$, 
\[
\E_N[e^{\gamma \X_N(x)+ \gamma \X_N(y)}\big]  = \Psi(\gamma,x)\Psi(\gamma,y)
\exp\big( \gamma^2\sigma_N^2+ \gamma^2 C_\X(x,y) +\underset{N\to\infty}{o(1)}  \big).
\]
Hence we conclude that with the required uniformity
\[
\lim_{N\to\infty} \E_N\big[ \nu_N(x) \nu_N(y)  \big]   =  e^{\gamma^2 C_\X(x,y)-\gamma g(x)-\gamma g(y)} . \qedhere 
\]
\end{proof}


\subsubsection{Contribution from $\Theta_{3, N,\ell}$} \label{sec:T3}
Our goal is now to obtain the lower bound \eqref{Upsilon0}. 
We choose $L= R \ell$ with $R>1$ and define
\begin{equation} \label{Upsilon1}
\Upsilon_{1,N,\ell}(x,y) : =  \E_N\big[\mathbf 1\{\mathcal{B}_{\ell , L}(x) , \mathcal{B}_{\ell , L}(y)\} \nu_N(x) \mu_N(y) \big] e^{-\gamma g(y)}.
\end{equation}

Recall \eqref{Theta} and that $L_N = \lfloor (1-\eta) \log N \rfloor$. 
Then, by a union bound, we can bound for $x\neq y$, 
\[
\Upsilon_{1,N,\ell}  \le \Theta_{3,N,\ell} 
+  \sum_{\substack{L <k\le L_N \\ z\in \{x,y\}}}
\E_N\big[\mathbf 1\{ \X_{N, e^{-k}}(z) \ge ( \gamma + \eta) k\} \nu_N(x) \mu_N(y) \big] e^{-\gamma g(y)}.
\]

We also define for $\alpha,\delta>0$ and $x,y\in\Omega$, 
\[\begin{aligned}
\Upsilon_{2,N}^{1,\delta}(x,y)  &: =  \E_N\big[\nu_N(x) \mu_N(y) e^{\alpha\X_{N, \delta}(x)} \big]  \\
\Upsilon_{2,N}^{2,\delta}(x,y)  &: =  \E_N\big[\nu_N(x) \mu_N(y) e^{\alpha\X_{N, \delta}(y)} \big] . 
\end{aligned}\]
By Markov's inequality, using this notation, it holds for any fixed $\alpha>0$, 
\begin{equation} \label{def:Ups}
\Upsilon_{1,N,\ell}  \le \Theta_{3,N,\ell} +  \Upsilon_{2,N,L}  , \qquad
\Upsilon_{2,N,L} := e^{-\gamma g(y)}  \sum_{\substack{L <k\le L_N}}
e^{-\alpha( \gamma + \eta) k} \big(  \Upsilon_{2,N}^{1,e^{-k}}+  \Upsilon_{2,N}^{2,e^{-k}} \big) . 
\end{equation}

In Subsection~\ref{sec:mixed1}, we obtain the bound \eqref{Upsilon2} for the quantity $\Upsilon_{2,N,L}$ (provided that $R$ is large enough and $\alpha=\eta$). 
Then, in Subsection~\ref{sec:mixed2}, we compute the pointwise limit of $\Upsilon_{1,N,\ell}$. This last step completes the proof of Proposition~\ref{prop:macro}.
Both subsections are based on the tools presented in Appendix~\ref{sec:1dapprox}. 


\subsubsection{Control of $\Upsilon_{2,N,L}$.} \label{sec:mixed1}
First, by  Proposition~\ref{prop:mixed2} with  $c=e^{-\ell}$, $\zeta =\gamma$ and $\alpha\in[0,\gamma]$, we obtain for  $(x,y) \in \mathcal{A}_c$,  $\delta= e^{-k}$ with  $k\in \{1, \cdots, L_N\}$ and $N\ge N_{\ell}$, 
\begin{equation} \label{Upsbd}
\Upsilon_{2,N}^{1,\delta}(x,y)  \le  C_{\gamma,\eta,g}  e^{\alpha\gamma k}  \frac{\E_N\big[e^{\beta_{N,\delta} \X_N(x)+\gamma \X_N(y)+\alpha \X_{N,\delta}(x)}\big]}{\E_N\big[e^{\gamma \X_N(x)}\big] \E[e^{\gamma \X_N(y)}]} ,\qquad  \beta_{N,\delta} =   \gamma + \alpha \frac{\log \delta}{\log N}. 
\end{equation}

Similarly, we can bound  $\Upsilon_{2,N}^{2,\delta}$ in terms of exponential moments;

\begin{proposition} \label{prop:mixed1}
For  $\alpha \in [0,\gamma]$ and $c>0$, it holds uniformly
for $(x,y)\in \mathcal{A}_c$ and  $\delta \in [N^{-1+\eta},1]$,  
\[
\Upsilon_{2,N}^{2,\delta}(x,y)
= \E_N\big[\nu_N(x) \mu_N(y) e^{\alpha\X_{N, \delta}(y)} \big] 
=   \frac{\E_N\big[e^{\gamma \X_N(x)+\gamma \X_N(y)+\alpha \X_{N,\delta}(y)}\big]}{\E_N[e^{\gamma \X_N(x)}]\E_N\big[e^{\gamma \X_N(y)}\big]} 
e^{-\gamma g(x)}\big( 1+ \underset{N\to\infty}{o(1)} \big) . 
\] 
\end{proposition}

\begin{proof} We use the method described in Section~\ref{sec:1dapprox}; let $X_N = \big(\X_N(x) - \gamma \log N \big)\epsilon_N$, $\epsilon_N = 1/\sqrt{\log N}$ and define
\[
\frac{d\Q_{N,y}}{d\P_N} : =  \frac{e^{\gamma \X_N(y)+\alpha \X_{N,\delta}(y)}}{\E_N[e^{\gamma \X_N(y)+\alpha \X_{N,\delta}(y)}]}  .
\]

We want to compute the asymptotics of $\Q[\X_N(x) \ge \gamma \log N +g(x)]$ where $\Q=\Q_{N,y}$. 
The characteristic function of the random variable $X_N$ under the measure $\Q$ biased by $e^{\beta \X_N(x)}$ for $\beta\in[0,\sqrt{2d})$ is given by
\[
\frac{\Q\big[ e^{i\chi  X_N+ \beta  \X_N(x)} \big]}{\Q\big[ e^{\beta  \X_N(x)} \big]} = 
\frac{\E_N\big[e^{\gamma \X_N(y)+\zeta \X_N(x)+\alpha \X_{N,\delta}(y)}\big]}{\E_N\big[e^{\gamma \X_N(y)+\beta\X_N(x)+\alpha \X_{N,\delta}(y)}\big]}  e^{- i \gamma \chi/\epsilon_N}
\]
where $ \zeta= \beta + i\chi \epsilon_N$ and $\chi\in\R$. 
Using \eqref{mom1} and \eqref{kernel}, we can rewrite 
\[
\frac{\Q\big[ e^{i\chi  X_N+ \gamma  \X_N(x)} \big]}{\Q\big[ e^{\gamma  \X_N(x)} \big]} 
= \frac{\Psi_{N,\delta}^{(\zeta,\gamma,\alpha)}(x,y;y)}{\Psi_{N,\delta}^{(\beta,\gamma,\alpha)}(x,y;y)}
\exp\big(- \tfrac{\chi^2}{2} + i  \chi \epsilon_N \gamma C_\X(x,y) + i\chi \epsilon_N \big(\alpha C_{\X,\delta}(y,x) - (\gamma-\beta)\epsilon_N^{-2}\big) \big) ,
\]
using that $(\zeta^2-\beta^2) \log N = 2i \beta \chi / \epsilon_N -\chi^2$.
Hence, choosing $\beta = \gamma - \alpha C_{\X,\delta}(y,x) \epsilon_N^{2}$,
we have $\beta \in (0,\sqrt{2d}-\eta)$ if $\eta$ is sufficiently small and $N \ge N_{\eta,c}$ ($| C_{\X,\delta}(y,x) | \le C_c$ for $(x,y)\in \mathcal{A}_c$ and any $\delta\in(0,1]$; cf.~Assumption~\ref{ass:cov})
and this Laplace function satisfies the condition \eqref{ass1} with 
\[
\psi_{N,\delta,x,y}(z) 
= \frac{\Psi_{N,\delta}^{(\beta+z,\gamma,\alpha)}(x,y;y)}{\Psi_{N,\delta}^{(\beta,\gamma,\alpha)}(x,y;y)} e^{z \gamma C_\X(x,y)} , \qquad 
z\in \mathcal{S} = \{z \in\C  :  |\Re z| < \eta\} . 
\]
By Assumption~\ref{ass:meso}, 
\(
\psi_{N,\delta,x,y}(z) \to  \frac{\Psi(\gamma+z,x)}{\Psi(\gamma,x)} e^{z \gamma C_\X(x,y)}
\)
and uniformly for $(x,y)\in \mathcal{A}_c$, $\delta \in [N^{-1+\eta},1]$ (in particular, we allow both $\delta>0$ fixed and $\delta(N) \to 0$). 
Hence, by applying Lemma~\ref{lem:1}, we conclude that uniformly for $(x,y)\in \mathcal{A}_c$  and  $\delta \in [N^{-1+\eta},1]$,  
\[ \begin{aligned}
\E_N\big[\nu_N(x) \mu_N(y) e^{\alpha\X_{N, \delta}(y)} \big] 
&= \Q[\X_N(x) \ge \gamma \log N + g(x)] \frac{\E_N[e^{\gamma \X_N(y)+\alpha \X_{N,\delta}(y)}]}{\P_N[\X_N(x) \ge \gamma \log N] \E_N[e^{\gamma \X_N(y)}] } \\
& =  N^{-\gamma^2} \frac{\E_N\big[e^{\gamma \X_N(x)+\gamma \X_N(y)+\alpha \X_{N,\delta}(y)}\big]}{\P_N[\X_N(x) \ge \gamma \log N] \E_N[e^{\gamma \X_N(y)}] }\frac{e^{-\gamma g(x)}}{\gamma\sqrt{2\pi \log N}}  \big( 1+ \underset{N\to\infty}{o(1)} \big) . 
\end{aligned}\]
Using the asymptotics \eqref{prob3}, one readily recovers the claim. 
\end{proof}

According to \eqref{mom1} and 
for $(x,y) \in \mathcal{A}_c$, $\alpha,\beta,\gamma \in [0,\sqrt{2d}]$, {there exists a constant $C$ depending only on $\mathcal A_c$ such that for} $\delta>0$, 
\[\begin{aligned}
\frac{\E_N\big[ e^{ \beta \X_N(x) + \gamma \X_N(y)+ \alpha \X_{N,\delta}(x)}\big]}{\E_N\big[ e^{ \beta \X_N(x)+ \alpha \X_{N,\delta}(x)} \big]\E_N\big[e^{\gamma \X_N(y)}\big]} 
&=  \frac{\Psi_{N,\delta}^{(\beta,\gamma,\alpha)}(x,y;{x})}{\Psi_{N,\delta}^{(\gamma,0,\alpha)}(x)\Psi_{N}^{(\gamma)}(y)} 
\exp\Big( \beta\gamma C_\X(x,y) + \gamma \alpha  C_{\X,\delta}(y,x) \Big) \\
&\leq C |x-y|^{-(\beta+\alpha)\gamma}
\end{aligned}\]
By Assumption~\ref{ass:meso},  these {bounds}
hold uniformly in all the parameters if $\delta \in [N^{\eta-1},c]$.

Combining these {bounds} with $c=e^{-\ell}$ and \eqref{mom2},  there exists $N_{\ell} \in \N$ and a constant $C = C_{K,\eta}$  so that for uniformly for  $(x,y) \in \mathcal{A}_c$, $\alpha,\beta,\gamma \in [0,\sqrt{2d}]$, $N\ge N_\ell$ and $\delta \in [N^{\eta-1},c]$, 
\begin{equation}\label{mom3}
\frac{\E_N\big[ e^{ \beta \X_N(x) + \gamma \X_N(y)+ \alpha \X_{N,\delta}(x)}\big]}{\E_N\big[e^{\gamma \X_N(x)}\big]\E_N\big[e^{\gamma \X_N(y)}\big]} \le C N^{\frac{\beta^2-\gamma^2}{2}} \delta^{-\beta\alpha-\alpha^2/2} |x-y|^{-(\beta+\alpha)\gamma } .
\end{equation}

By Proposition~\ref{prop:mixed1} with $c=e^{-\ell}$,
the bound \eqref{mom3} (exchanging $x\leftrightarrow y$ and taking $\beta=\gamma$) implies that for $(x,y)\in \mathcal{A}_c$, $\alpha \in [0,\gamma]$, it holds  for $N\ge N_{\ell}$ and $k\in \{\ell, \cdots, L_N\}$, 
\[\begin{aligned}
\Upsilon_{2,N}^{2,e^{-k}}(x,y) \le C_{K,\eta,g}  |x-y|^{-\gamma^2 - \gamma \alpha} e^{k (\gamma\alpha+ \alpha^2/2)} . 
\end{aligned}\] 
In particular, choosing $\alpha=\eta \leq \gamma$, we conclude that 
\[
e^{-\alpha( \gamma + \eta) k} \Upsilon_{2,N}^{2,e^{-k}}(x,y) \le C_{K,\eta,g}  e^{2\ell
\gamma^2}  e^{-k \eta^2/2} . 
\]

Going back to \eqref{Upsbd}, using \eqref{mom3} with  $\beta=\beta_{N,e^{-k}}$ so that $(\beta-\gamma)k =  -\frac{\alpha k^2}{\log N}$, we obtain for $(x,y)\in \mathcal{A}_c$, $N\ge N_{\ell}$ and  $k\in \{1, \cdots, L_N\}$,
\[\begin{aligned}
\Upsilon_{2,N}^{1,e^{-k}}(x,y)   & \le  C  e^{\alpha(\gamma+\beta_N +\alpha/2)  k}  N^{\frac{\beta_N^2-\gamma^2}{2}} |x-y|^{-(\beta+\alpha)\gamma }  \\
&\le C e^{\alpha(\gamma +\alpha/2)  k  + \ell \gamma(\gamma+\alpha) } 
\end{aligned}\]
where we used that $c=e^{-\ell}$,
$\tfrac{\beta_N^2-\gamma^2}{2}\log N  = -\gamma \alpha k + \tfrac{\alpha^2}{2}\frac{k^2}{\log N}$. 
Hence, choosing $\alpha=\eta$, we also have 
\[
e^{-\alpha( \gamma + \eta) k} \Upsilon_{2,N}^{1,e^{-k}}(x,y) \le C e^{2\ell
\gamma^2}  e^{-k \eta^2/2} . 
\]

Combining these estimates in \eqref{def:Ups} where $L = R\ell$, we conclude that  for $k\in \{1, \cdots, L_N\}$, $(x,y)\in \mathcal{A}_c$  and $N\ge N_{\ell}$,
\[
\Upsilon_{2,N,\ell}(x,y)  \le  C 
\sum_{\substack{L <k\le L_N}}
e^{-\alpha( \gamma + \eta) k} \big( \Upsilon_{2,N}^{1,e^{-k}} +  \Upsilon_{2,N}^{2,e^{-k}} \big)(x,y)
\le C e^{2\ell \gamma^2} \sum_{k> R\ell} e^{- \tfrac{\eta^2 k}{2}} 
\le C_{\gamma, g,R} e^{-D \ell}
\]
upon choosing $R\in \N$ sufficiently large (depending only on the parameters $\eta, D$). 

This completes the proof of \eqref{Upsilon2}.


\subsubsection{Asymptotics of $\Upsilon_{1,N,\ell}$.} \label{sec:mixed2}
We consider a new probability measure with density proportional to $\nu_N(x) \mu_N(y)$ that is, for $x, y \in\Omega$,
\[
\frac{d\mathbb{P}^3_{N,x,y}}{d\P_N}  =  \frac{\1\{\X_N(x) \ge \gamma \log N +g(x)\}e^{\gamma \X_N(y)}}{\E_N\big[\mathbf 1\{\X_N(x) \ge \gamma \log N +g(x)\} e^{\gamma \X_N(y)}\big]} . 
\]
Recall \eqref{def:Y}, according to \eqref{Upsilon1}, we have
\begin{equation} \label{Ups1}
\Upsilon_{1,N,L}(x,y)=
\mathbb{P}^3_{N,x,y}\big[ \mathbf{Y}_{2k-1} , \mathbf{Y}_{2k}  \le ( \gamma + \eta) k : k\in [\ell, L]  \big]   \E_N\big[\nu_N(x) \mu_N(y)\big]    e^{-\gamma g(y)} .  
\end{equation}

The goal of this section is to obtain the following result. 

\begin{proposition} \label{prop:mixed3}
For $x,y\in\Omega$, $x\neq y$, under $\mathbb{P}^3_{N,x,y}$, the random vector  $\mathbf{Y}$ converges in distribution as $N\to\infty$ toward a multivariate Gaussian $\mathcal{N}(\mathbf{m}_{x,y}, \Sigma_{x,y})$. Moreover,
\begin{equation} \label{V1}
\lim_{N\to\infty}  \E_N\big[\nu_N(x) \mu_N(y)\big]  = e^{\gamma^2 C_\X(x,y)} e^{-\gamma g(x)}  . 
\end{equation}
\end{proposition}

By \eqref{Ups1}, Proposition~\ref{prop:mixed3} directly implies that the limit of $\Upsilon_{1,N,L}$ is given by \eqref{T3macro},  concluding the proof of Proposition \ref{prop:macro}.

\begin{proof}
First observe that according to Proposition~\ref{prop:mixed1} with $\alpha=0$, we have for $x,y\in\Omega$, $x\neq y$,
\[
\E_N\big[\nu_N(x) \mu_N(y)\big]= \frac{\E_{N}\big[ e^{\gamma\X_N(x)+\gamma \X_N(y)}\big]}{\E_{N}[e^{\gamma \X_N(x)}]\E_{N}[e^{\gamma \X_N(y)}]}\big( e^{-\gamma g(x)}+ \underset{N\to\infty}{o(1)} \big) . 
\]
Hence, by \eqref{W}--\eqref{W2}, we obtain  \eqref{V1}. 
To prove the first claim, we consider the Laplace transform of the random vector $\mathbf{Y}$ that is, for $\xi\in\R^{q}$, 
\begin{equation} \label{LaplaceQ2}
\mathbb{P}^3_{N,x,y}\big[ e^{\x \cdot  \mathbf{Y} } \big]  = 
\frac{\E_N\big[ \1\{\X_N(x) \ge \gamma \log N +g(x)\} e^{\gamma \X_N(y)+ \langle \X_N , \mathrm{f}_{\delta,z}  \rangle} \big]}{\E_N\big[ \1\{\X_N(x) \ge \gamma \log N +g(x)\}e^{\gamma \X_N(y)}\big] }
\end{equation}
where we used that $\xi\cdot\mathbf{Y} = \langle \X_N , \mathrm{f}_{\delta,z}  \rangle$ 
according to Definition~\ref{def:mS}.

We rely again on the method from the Appendix~\ref{sec:1dapprox} to compute this quantity. 
For this purpose, we consider the Laplace transform of the random variable  $X_N = \big(\X_N(x) - \gamma \log N\big)\epsilon_N$ with  $\epsilon_N  = 1/\sqrt{\log N}$  under the biased measure $\mathbb{L}=\mathbb{L}_{N,x,y,\xi}$, \eqref{Fmeasure}, that is, for $\chi\in\C$, 
\[
\mathbb{L}\big[ e^{i\chi X_N} \big]
=  \frac{\E_N\big[ e^{\zeta \X_N(x) +\gamma\X_N(y)+ \langle \X_N , \mathrm{f}_{\delta,z}  \rangle } \big] }{ \E_N\big[e^{\gamma \X_N(x)+ \gamma \X_N(y)+\langle \X_N , \mathrm{f}_{\delta,z}  \rangle } ]} e^{- i\gamma\chi \epsilon_N^{-1}} 
\qquad\text{where}\qquad  \zeta= \gamma + i\chi \epsilon_N  . 
\]
Using \eqref{mom1}, we can rewrite for $\zeta \in \mathcal{D}_\infty$, 
\[
\mathbb{L}\big[ e^{i\chi X_N} \big]
= \psi_N(i \chi \epsilon_N) e^{-\chi^2/2} , \qquad 
\psi_N(w)
=  \frac{\Psi_{N,\delta}^{(\gamma+w,\gamma,\xi)}(x,y;z)}{\Psi_{N,\delta}^{(\gamma,\gamma,\xi)}(x,y;z)}
\exp\big( w \varphi(x,y) \big) 
\]
where $\varphi(x,y) = \gamma C_\X(x,y)+\int  C_\X(x,u)  \mathrm{f}_{\delta,z}(u) d u$. 
By Assumption~\ref{ass:meso}, this Laplace function satisfies the condition \eqref{ass1}. 
Hence, applying Lemma~\ref{lem:1} with $\beta=\gamma$, we obtain for fixed $\xi\in\R^{q}$ and $x,y\in\Omega$ with $x\neq y$, as $N\to\infty$, 
\[
\frac{\E_N\big[ \1\{\X_N(x) \ge \gamma \log N +g(x)\} e^{\gamma \X_N(y)+ \langle \X_N , \mathrm{f}_{\delta,z}  \rangle} \big]}{\E_N\big[e^{\gamma \X_N(y)+ \gamma \X_N(x) +\langle \X_N , \mathrm{f}_{\delta,z}  \rangle}  ]}
\sim \frac{N^{-\gamma^2} e^{-\gamma g(x)}}{\gamma\sqrt{2\pi \log N}}
\]
Going back to \eqref{LaplaceQ2}, 
 using this twice (once with $\mathrm{f}=0$), implies that for $x,y\in\Omega$ with $x\neq y$, 
\[
\mathbb{P}^3_{N,x,y}\big[ e^{\x \cdot  \mathbf{Y} } \big]  \sim
\frac{ \E_N\big[e^{\gamma \X_N(x)+ \gamma \X_N(y)+\langle \X_N , \mathrm{f}_{\delta,z}  \rangle} ]}{\E_N\big[e^{\gamma \X_N(x)+ \gamma \X_N(y)}\big]}
\]
We already computed the limit of this quantity in the proof of Proposition~\ref{prop:T1lim}, so this completes the proof of Proposition~\ref{prop:mixed3}. 
\end{proof}


\section{Verification of assumptions for Gaussian fields}\label{sec:gauss}

A natural task is to verify that the assumptions from Section~\ref{sec:gen} hold for a large class of convolution approximations for a Gaussian log-correlated field, since this is arguably the most basic way of regularizing  it. 
In addition, this allows us to prove Theorem~\ref{thm:gauss}.

\subsection{Convention and GMC convergence} 
Throughout this section, we assume that $\Omega \subset \R^d$ is an open set and  that $\X$ is a (mean-zero) Gaussian log-correlated field with correlation kernel \eqref{eq:cov} where $h \in  \mathcal{C}\big(\Omega \times\Omega \to \R\big)$ is locally $\alpha$-H\"older continuous for a $\alpha\in(0,1]$. 
In this context, for a compact $K \subset \Omega$, we define a regularization $\big(\X_\delta(x)\big)_{x\in K, \delta\in(0,c]}$ by a convolution with a smooth mollifier $\rho$, see \eqref{eq:convapp}.
We also abuse notation and let
\begin{equation} \label{Gdef}
\X_{N} := \X_{1/N} , \qquad \X_{N,\delta} = \X_\delta \quad\text{for $ \delta\in (N^{-1},c]$} . 
\end{equation}
We recall the following classical result from \cite{Berestycki} concerning the existence of the GMC measures associated with $\X$ in the subcritical phase. 

\begin{theorem}[Berestycki, \cite{Berestycki}] \label{thm:Berestycki}
Let $\rho$ and $(\X_N)_{N\in\N}$ be as in Theorem~\ref{thm:gauss}. 
Then, for any $\gamma \in (0,\sqrt{2d})$, the random measure $\mu_{N,\gamma}$ given by \eqref{eq:GMC} converges in probability as $N\to\infty$ $($with respect to the vague topology$)$ to $\mu_{\X,\gamma}$, which is a Borel measure on $\Omega$ called a GMC measure. 
\end{theorem}

There are many other results about the existence of GMC measures (and the convergence of different Gaussian approximations): \cite{Berestycki} is very concise, \cite{RV} offers a review of the subject, while  \cite{Shamov} is perhaps the most general treatment.

\smallskip

With the convention \eqref{Gdef}, one has according to \eqref{kernel} and \eqref{mom1},  for any $\zeta_1 , \zeta_2 \in \C^2$, $\xi\in\R^q$, $\delta\in (0,c_K]^q$ and  
$x_1,x_2,z_1,\dots,z_q \in K$, 
\begin{equation} \label{GaussPsi}
\begin{aligned}
\Psi_{N,\delta}^{(\zeta_1,\zeta_2,\xi)}(x_1, x_2;z) 
=  \exp\Big( \tfrac{\zeta_1^2}{2} K_N(x_1)  +  \tfrac{\zeta_2^2}{2} K_N(x_2) + \zeta_1\zeta_2 (C_{\X,\epsilon,\epsilon}(x_1,x_2) -C_\X(x_1,x_2))  \Big) \,  \\
\exp\Big(\sum_{j=1}^q \big[  \zeta_1 \xi_j \big(C_{\X,\delta_j,\epsilon}(z_j,x_1) - C_{\X,\delta_j}(z_j,x_1) \big)+
\zeta_2 \xi_j \big(C_{\X,\delta_j,\epsilon}(z_j,x_2) - C_{\X,\delta_j}(z_j,x_2) \big)\big]  \Big) ,  
\end{aligned}
\end{equation}
where $\epsilon=1/N$ and $K_N(x) = C_{\X,\epsilon,\epsilon}(x,x) - \log \epsilon^{-1}$. This follows from a Gaussian  computation. 

\medskip

{The goal of this section is to verify that this function $\Psi_N$ satisfies both Assumptions~\ref{ass:meso} and~\ref{ass:macro}. 
This boils down to precise estimates for the regularized kernels \eqref{kernel}. 
In particular, we will also obtain Assumption~\ref{ass:cov} at the end of the proof.

\subsection{Estimates for regularized correlation kernels}
Recall that in the context of  Theorem~\ref{thm:gauss}, for $x,y\in K$, 
\[
C_{\X}(x,y) = - \log|x-y|+h(x,y)
\]
where $h$ is $\alpha$-H\"older continuous for some $\alpha\in(0,1]$.
Recall that $\rho(du)$ is a probability measure on $\R^d$ with a continuous density with compact support.
Without loss of generality, we assume that $\operatorname{supp}(\rho) \subset \{u\in\R^d : |u| \le 1\}$. 

We define for $x,z \in K$ and $\epsilon,\delta \in [0,c]$, 
\[
\mathfrak{h}_\delta(z,x)  :=  \int h(u,x) \rho_{\delta,z}(u) d u , \qquad
\mathfrak{g}_{\epsilon,\delta}(x,z) : = 
\int h(v,u) \rho_{\delta,z}(u) \rho_{\epsilon,x}(v) d u dv,
\]
with the convention that $\mathfrak{g}_{0,\delta} = \mathfrak{h}_\delta$ and $\mathfrak{h}_0 = h$. 
We immediately verify that there is a constant $C=C_{K,h}$ so that for any
$x,z ,y \in K$ and $\epsilon,\delta \in [0,c]$, 
\begin{equation} \label{conth}
\big| \mathfrak{h}_{\epsilon}(x,z) - \mathfrak h_\delta(x,y) \big| 
\le  \int \big|  h(x+ \epsilon u,z) - h(x+\delta u,y)\big| \rho(du)
\le C\big( |\epsilon-\delta|^\alpha + |z-y|^\alpha \big)
\end{equation}
and 
\begin{equation} \label{contg}
\big| \mathfrak{g}_{\epsilon,\delta}(x,z)  -   \mathfrak{h}_{\delta}(x,z)  \big|
\le \int  \big| \mathfrak{h}_{\epsilon}(x,u) - h(x,u) \big| \rho_{\delta,z}(u) du
\le C \epsilon^\alpha. 
\end{equation}

\begin{lemma}\label{lem:cov1}
Assume that $\alpha<1$. 
For any $x,z \in K$ and $\delta \in (\epsilon,c]$, 
we have as $\epsilon\to0$, 
\[
C_{\X,\delta,\epsilon}(z,x) =  C_{\X,\delta}(z,x) +O\big( (\epsilon/\delta)^\alpha  \log \delta^{-1} \big) .
\]
\end{lemma}

\begin{proof}
By \eqref{kernel}, for $z,x \in K$ and $\delta \in (0,c]$,
\[
C_{\X,\delta}(z,x)   =  \int  \log |u-x|^{-1}  \rho_{\delta,z}(u) d u
+ \mathfrak{h}_\delta(z,x) .
\]
Since $\rho$ is uniformly bounded, for $x,y,z\in K$ and $r>0$, 
\[
\bigg| \int_{|x- u| \le 2 r} \hspace{-.5cm}  \log|u-y|^{-1}  \rho_{\delta,z}(du) \bigg| 
\le C (r/\delta)^d\log r^{-1}  . 
\]
Thus for $v,x,z \in K$ with $|v-x| \le \varepsilon \le \delta \le c$ we have
\[\Big|\int_{|x-u| < 2|x-v|} (\log |u-v| - \log |u-x|) \rho_{\delta,z}(du) \Big| \le C\frac{|x-v|^d}{\delta^d}\Big(\log \frac{\delta}{|x-v|} + \log \frac{1}{\delta}\Big),\]
which is $O\Big(\frac{|x-v|^{\alpha}}{\delta^{\alpha}} \log \frac{1}{\delta}\Big)$ when $\alpha < 1$. Together with the estimate \eqref{conth} this implies that
\[
\big| C_{\X,\delta}(z,v)  - C_{\X,\delta}(z,x)\big|
 \le \int_{|x- u| \ge 2  |x-v|} \hspace{-.5cm} \big| \log|u-v| - \log|u-x| \big| \rho_{\delta,z}(du) + O\Big(  \frac{|x-v|^\alpha}{\delta^\alpha} \log \frac{1}{\delta} \Big) .
\]
Now, using the bound $\log|1+ \theta| \le C |\theta|^\alpha$ valid for $|\theta|\le 1/2$,  we obtain
\[ \begin{aligned}
\big| C_{\X,\delta}(z,v)  - C_{\X,\delta}(z,x)\big|
& \le C  \int \bigg| \frac{v-x}{u-x} \bigg|^\alpha \rho_{\delta,z}(du) + O\Big(  \frac{|x-v|^\alpha}{\delta^\alpha} \log \frac{1}{\delta} \Big) \\
&= O\Big(  \frac{|x-v|^\alpha}{\delta^\alpha} \log \frac{1}{\delta} \Big)
\end{aligned}\]
where the implied constant depends only on $\alpha$. 
Hence, using that
\[
C_{\X,\delta,\epsilon}(x,z) =  \int C_{\X,\delta}(x,v) \rho_{\epsilon,z}(dv)   
\]
we obtain for $z,x\in K$, 
\[ \begin{aligned}
|C_{\X,\delta,\epsilon}(z,x) - C_{\X,\delta}(z,x)|
&\le   \int  \big| C_{\X,\delta}(z,v)  - C_{\X,\delta}(z,x)\big| \rho_{\epsilon,x}(dv)    \\ 
&\le \frac{C_\alpha \epsilon^\alpha  \log \delta^{-1} }{\delta^\alpha} \int |v|^\alpha \rho(d v) 
\end{aligned}\]
which proves the claim.
\end{proof}

\begin{lemma}\label{lem:cov2}
For $x,y \in K$ with $|x-y| \ge 4 \epsilon$, we have as $\epsilon\to0$, 
\[
 C_{\X,\epsilon,\epsilon}(x,y) = C_\X(x,y) +  O\bigg(  \frac{\epsilon}{|x-z|}  + \epsilon^\alpha \bigg) . 
\]
Let $\kappa(x) = - \int \log\big| v- u \big| \rho(du) \rho(dv) + h(x,x) $. 
Then, uniformly for $x\in K$, as $\epsilon\to0$,
\[
C_{\X,\epsilon,\epsilon}(x,x) =  \log \epsilon^{-1} +\kappa + O(\epsilon^{\alpha}) . 
\]
\end{lemma}

\begin{proof}
By definitions \eqref{kernel}, for $x,z\in K$, 
\[
C_{\X,\epsilon,\epsilon}(x,z) =  - \int \log|u-v|  \rho_{\epsilon,x}(du)   \rho_{\epsilon,z}(dv)   
+  \mathfrak{g}_{\epsilon,\epsilon}(x,z) 
\]
where   $\mathfrak{g}_{\epsilon,\epsilon}(x,z)  = h(x,z) +  O(\epsilon^\alpha)$; cf.~\eqref{conth}--\eqref{contg}.
By a change of variables, if $x=z$, this immediately implies the second claim. 
Then, for the first claim, we have for say $|x-z| \ge 4 \epsilon$, 
\[
\big| C_{\X,\epsilon,\epsilon}(x,z)- C_\X(x,z) \big| \le  \int \log\bigg(1+ \epsilon\bigg|\frac{u-v}{x-z} \bigg| \bigg)  \rho(du)   \rho(dv)   
+  O(\epsilon^\alpha)  .
\]
Using the bound $\log|1+ \theta| \le C |\theta|$ valid for $|\theta|\le 1/2$, this completes the proof.
\end{proof}

Let $0<\eta<1$ and assume that $\zeta_1, \zeta_2 \in\C$ with $|\zeta_1|, |\zeta_2| \le N^{\eta/2}$. 
Returning to formula \eqref{GaussPsi},
combining the estimates from Lemmas~\ref{lem:cov1} and~\ref{lem:cov2}, 
 this implies that uniformly as $N\to\infty$, in the regime  where $\delta_k \in [N^{\eta-1}, c]$ and $|x_1-x_2|N^{1-\eta} \to \infty $, it holds for any fixed $\xi\in\R^q$, 
\[\begin{aligned}
\Psi_{N,\delta}^{(\zeta_1,\zeta_2,\xi)}(x_1, x_2;z) 
 &=  \exp\Big( \tfrac{\zeta_1^2}{2} K_N(x_1)  +  \tfrac{\zeta_2^2}{2} K_N(x_2)  +  \underset{N\to\infty}{o(1)}  \Big) \\
 & = \Psi(\zeta_1,x_1) \Psi(\zeta_2,x_2) \big(1+\underset{N\to\infty}{o(1)} \big) 
\end{aligned}\]
where $\Psi(\zeta,x) = \exp\big(\tfrac{\zeta^2}{2} \kappa(x)  \big)$. 
This statement implies both Assumptions~\ref{ass:meso} and~\ref{ass:macro}. 
 
\medskip 
 
Finally, the Assumption~\ref{ass:cov} follows immediately from the following estimates. 

\begin{lemma} 
Let $\rho$ be as in Theorem~\ref{thm:gauss}. 
For $x,z \in K$ and $c \ge \delta \ge \epsilon>0$,
\[\begin{aligned}
C_{\X,\delta}(x,z)  & = - \log\big( |x-z| \vee \delta \big) 
+ \mathfrak{h}_\delta(x,z) 
+ O\big(\big( \tfrac{\delta}{|x-z|+\delta}\big)^2\big)  , \\
C_{\X,\delta,\epsilon}(x,z)  & = - \log\big( |x-z| \vee \delta \big) 
+   \mathfrak{g}_{\epsilon,\delta}(x,z) +   O\big(\big( \tfrac{\delta}{|x-z|+\delta}\big)^2\big) . 
\end{aligned}\]
\end{lemma}

\begin{proof}
To simplify the proof, we assume that the probability measure $\rho$ is rotationally-invariant, even though it is straightforward to adapt the argument.

By definition of $C_{\X,\delta}$, \eqref{kernel}, we have for $\delta\in(0,c]$ and $x,z \in K$, 
\[ \begin{aligned}
& C_{\X,\delta}(x,z) + \log(|x-z|\vee \delta)
= \int \big(  \log(|x-z|\vee \delta) - \log|x-u|\big) \rho_{\delta,z}(u) d u
+ \mathfrak{h}_\delta(x,z)   \\
&\qquad  = -  \1_{|x-z|> \delta}\int \log\big|\mathrm{e}_1 + \tfrac{\delta u}{|x-z|} \big| \rho(du)  
-  \1_{|x-z| \le \delta} \int \log\big| \tfrac{x-z}{\delta}  + u \big| \rho(du)  + \mathfrak{h}_\delta(x,z) .
\end{aligned}\]
At the second step, we made a change of variable and used that $\rho$ is rotationally invariant; $\mathrm{e}_1$ denotes the first basis vector in $\R^d$. 

Now, we have 
\begin{equation*}
\max_{|v| \le 1} \int \big| \log| v  + u| \big| \rho(du)  <\infty , 
\end{equation*}
and for $0<r<1$,
\[
\int \log\big|\mathrm{e}_1 + r u \big| \rho(du)   =  \frac 12 \int \log\big( (1 + r u_1)^2 + r^2 |u_\perp|^2 \big) \rho(du) 
\le C r^2 , 
\]
where we decompose $u\in\R^d$ as $u=(u_1,u_\perp)$ and the constant $C$ depends only on $\rho$. 

This proves the first claim. The second claim follows by the same argument.
\end{proof}
}

\subsection{Conclusion}
From the results of the previous section, as a consequence of Theorem~\ref{thm:main} 
and Theorem~\ref{thm:Berestycki}, we obtain Theorem~\ref{thm:gauss}. 


\section{Verification of assumptions for the CUE log-characteristic polynomial}\label{sec:CUE}

\subsection{Correlations of the log-characteristic polynomial}
Let us recall that in the setting of Conjecture \ref{con:fk} and Theorem~\ref{thm:cue}, 
\begin{equation} \label{def:XN}
\X_N(x)=\sqrt{2}\log |\det(I-e^{-i x}U_N)| \qquad x\in \T,\, N\in\N , 
\end{equation}
where $U_N$ is a Haar distributed random $N\times N$ unitary matrix and $\T = \R/{(2\pi \Z)}$.
As discussed in Section~\ref{sec:gen}, we consider the approximation kernels
$\rho_{\delta,x}(\theta) = \sum_{|k|\leq \delta^{-1}}e^{ik(x-\theta)}$ for $\delta\in(0,1]$, leading to
\[
\X_{N,\delta}(x)=-\Re \bigg( \sum_{k\delta\le 1} \frac{\Tr U_N^k}{k/\sqrt 2} e^{-ikx} \bigg) , 
\qquad x\in \T , \, \delta\in(0,1] , \, N\in\N . 
\]
and the following estimates;

\begin{lemma} \label{lem:covcue}
Let $\X$ be the free field on $\T$, that is, a generalized Gaussian process with covariance kernel \eqref{eq:covT}. 
Then, with $\rho_{\delta,x}$ as above, we have for $\delta,\epsilon \in (0,1]$ with $\epsilon\le \delta$ and $x,\theta\in\T$, 
\[
C_{\X,\delta,\epsilon}(\theta,x) =C_{\X,\delta}(\theta,x)  
=  \Re \bigg( \sum_{k\delta\le 1} \frac{e^{ik(\theta-x)} }{k} \bigg) .
\]
Moreover, the Assumption~\ref{ass:cov} holds with $K=\T$. 
\end{lemma}

\begin{proof}
The expressions for $C_{\X,\delta}(x,\theta)  $ and  $C_{\X,\delta,\epsilon}$ follow directly from the fact that the kernel \eqref{eq:covT} can be expressed as 
\begin{equation} \label{kernelFourier}
C_\X(\theta,x) =  \Re\bigg( \sum_{k\in\N}  k^{-1}  e^{ik(\theta-x)} \bigg)
\end{equation}
as a generalized Fourier series. 
Note that $C_{\X,\delta}$ are convolution kernels on $\T$ and 
\[
C_{\X,\delta}(\theta,x) =  \sum_{k\le 1/\delta}  k^{-1} \cos(k\Delta) , \qquad 
\Delta= \dist(x,\theta)
\]
where $\dist(x,\theta) = |x-\theta| \mod 2\pi $ is the distance function on $\T$.

To check the Assumption~\ref{ass:cov},  we separate two cases. 
Recall that 
\(
\log r =  \sum_{k\le r} k^{-1} + O(1) 
\)
as $r\to\infty$, then  if $\Delta \le \delta$, 
\[
C_{\X,\delta}(\theta,x)  - \log \delta^{-1}
=  \sum_{k\le \frac{1}{\delta}}  k^{-1} \big(\cos(k\Delta) -1\big) + \underset{\delta\to0}{O}(1)
\]
and 
\[
\sum_{k\le 1/\delta}  k^{-1} \big| \cos(k\Delta) -1\big| \le \frac{\Delta^2}{2} \sum_{k\le 1/\delta}  k
\le \frac 12 .
\]
This establishes the case where $\dist(x,\theta) \le \delta$.

On the other hand, if $\Delta \ge \delta$, with $L =  \lfloor \Delta^{-1} \rfloor$, 
\begin{align*}
C_{\X,\delta}(\theta,x)  - \log \Delta
&=  C_{\X,\Delta}(\theta,x)-\log \Delta 
+ \Re \bigg(\sum_{L < k\le 1/\delta}  k^{-1} e^{i k\Delta} \bigg)\\ 
& = \underset{\Delta\to0}{O}(1)
+  \Re \bigg(\sum_{L < k\le 1/\delta}  k^{-1} e^{i k\Delta} \bigg)  
\end{align*}
where we used the first case (in particular, the error is independent of $\delta$). 
We can control the oscillatory sum by making a summation by parts, we obtain
\[
\sum_{L< k\le 1/\delta}  k^{-1} e^{i k\Delta} =  \frac{\Theta_L}{L+1}
+   \sum_{L < k\le 1/\delta} \frac{\Theta_k}{k(k+1)}
\]
where $\Theta_k =  \sum_{k< j\le 1/\delta} e^{i j\Delta}$.

Using that $|\Theta_k | \le C/\Delta$ (for some universal constant $C$) for any $\delta>0$ and $k\in\N$ with $k\le 1/\delta$, this implies that 
\[
\bigg|  \sum_{L \le k\le 1/\delta}  k^{-1} e^{i k\Delta} \bigg|  =O(1)
\]
uniformly for $\delta>0$. This completes the proof. 
\end{proof}

\begin{remark}
For $k, n \in\Z$, with $k\neq 0$, 
\[
\E\big[ \Tr U_N^k\Tr U_N^n \big] =  \begin{cases}
{|k|} \wedge N &  \text{if }n=-k \\
0 &\text{else}
\end{cases} .
\]
Combined with \eqref{trseries}, this implies that for $\delta , \epsilon \in [1/N, 1]$, we have exactly for $x,\theta \in \T$, 
\[
\E \X_{N,\delta}(x) \X_{N,\epsilon}(\theta)  = C_{\X,\delta,\epsilon}(x,\theta)  , 
\]
while  
\[
\E \X_{N}(x) \X_{N}(\theta)  = C_{\X,1/N}(x,\theta)  + \sum_{k>N} \frac N{k^2} \cos\big(k\dist(x,\theta)\big) . 
\]
Note that the last sum is continuous on $\T^2$ and uniformly bounded by 1.
\end{remark}

\subsection{Fisher-Hartwig asymptotics}

We now turn to verifying Assumptions~\ref{ass:meso} and \ref{ass:macro}. 
First, the Assumption~\ref{ass:meso} (1) is trivial since both 
$\X_N , \X_{N,\delta} \le C N$  for $\delta \in [1/N,1]$  with a numerical constant $C>0$. Hence the function 
$(\zeta_1,\zeta_2,\xi)\mapsto \Psi_N^{(\zeta_1,\zeta_2,\xi)}(x_1,x_2,z)$ is analytic in $\{\gamma \in \C : \Re\gamma>0\}^2\times \C^q$ for $x_1,x_2,z_1,\dots, z_q\in \T$ with $x_1\neq x_2$. 
Moreover, the positivity requirement from Assumption~\ref{ass:meso} (2) as well as the    growth requirement from Assumption~\ref{ass:macro} (2) is also immediate for $\Psi$ as in Lemma ~\ref{le:morris} below. 
Indeed, the Barnes G-function $z\mapsto G(1+z)$  is analytic and does not vanish in $\{z \in \C : \Re z>-1\}$, so $\Psi(\zeta) >0 $ for $\zeta \ge 0$.

Second, we review the existing literature, in particular the Selberg-Morris integral formula (see e.g. \cite[equation (1.18)]{FW} and set there $a=b=\frac{1}{2}\zeta$, $\gamma=1$); for any $\zeta\in \C$ with $\mathrm{Re}(\zeta)>-1$ and $\theta\in \T$, 
\begin{equation}\label{eq:sm}
\E |\det(e^{i\theta}-U_N)|^\zeta=\frac{1}{N!}\prod_{j=0}^{N-1}\frac{\Gamma(1+\zeta+j)\Gamma(2+j)}{\Gamma(1+j+\frac{\zeta}{2})^2}.
\end{equation}
Note that both sides are independent of $\theta$ by rotational invariance of the Haar measure. 
Similarly, the function $\Psi(\zeta) = \Psi(\zeta,\theta)$ is  independent of $\theta\in\T$ in the CUE case. 
We obtain the following asymptotics for the Laplace transform of the CUE log characteristic polynomial.
\begin{lemma}\label{le:morris}
It holds uniformly for $\mathrm{Re}(\zeta)>-1/\sqrt{2}$ with $\zeta=o(N^{1/3})$ and $\theta\in \T$, 
\[
\E\big[ e^{\zeta\X_N(\theta)} \big]=(1+o(1))N^{\frac{\zeta^2}{2}}\Psi(\zeta) , \qquad  \Psi(\zeta) = \frac{G(1+\zeta/\sqrt 2)^2}{G(1+\sqrt2\zeta)},
\]
where $G$ is the Barnes G-function:
\[
G(1+z)=(2\pi)^ {z/2}\exp\left(-\frac{z+z^2(1+\gamma_{\mathrm{EM}})}{2}\right)\prod_{k=1}^\infty \left[\left(1+\frac{z}{k}\right)^k\exp\left(\frac{z^2}{2k}-z\right)\right],
\]
for $z\in\C$, $\gamma_{\mathrm{EM}}$ being the Euler-Macheroni constant.

Moreover, $\Psi$ satisfies the bound of Assumption \ref{ass:macro} (2); there exist constants $\varkappa>2$ and $C_\varkappa>0$ such that for $\mathrm{Re}(\zeta)\geq 0$, 
\[
|\Psi(\zeta)|\leq Ce^{|\zeta|^\varkappa} .
\]
\end{lemma}

\begin{proof}
By definition, 
from \eqref{eq:sm}, and the basic fact that $G(z+1)=\Gamma(z)G(z)$,
we verify that
\[
\E |\det(1-e^{i\theta}U_N)|^\zeta
= G(N+1)\frac{G(\zeta+N+1)}{G(\frac{\zeta}{2}+N+1)^2}\frac{G(1+\frac{\zeta}{2})^2}{G(1+\zeta)}.
\]
In particular, $\zeta\in\C\mapsto G(\zeta+1)$  is entire and without zero in this region  $\Re \zeta >-1$.

This leads us to consider 
\[
L_N(\zeta):=\log G(N+1)+\log G(\zeta+N+1)-2\log G(\tfrac{\zeta}{2}+N+1).
\]
Note that the condition $\mathrm{Re}(\zeta)>-1$ and $\zeta=o(N^{1/3})$ imply that all of the arguments of $G$ are here large (and have real part positive and of order $N$ while the imaginary part is $o(N^{1/3})$). To estimate $L_N$, we use the asymptotic expansion 
\begin{equation}\label{eq:Gasy}
\log G(1+z)=\frac{z^2}{2}\log z-\frac{3z^2}{4}+\frac{z}{2}\log (2\pi)-\frac{1}{12}\log z+\left(\frac{1}{12}-\log A\right)+O(z^{-2}),
\end{equation}
where $A$ is the Glaisher-Kinkelin constant, and the expansion is valid in any sector not containing the negative real axis -- see e.g. \cite[Theorem 1, Theorem 2, and Theorem 3]{FL} where $G$ is called the double gamma function. In particular, we find 
\begin{align*}
L_N(\zeta)&=\tfrac{N^2}{2}\log N -\tfrac{3 N^2}{4}-\tfrac{1}{12}\log N+\tfrac{(\zeta+N)^2}{2}\log(\zeta+N)-\tfrac{3(\zeta+N)^2}{4}-\tfrac{1}{12}\log(\zeta+N)\\
&\qquad -(N+\tfrac{\zeta}{2})^2 \log(N+\tfrac{\zeta}{2})+\tfrac{3(N+\tfrac{\zeta}{2})^2}{2}+\tfrac{1}{6}\log (N+\tfrac{\zeta}{2})+O(N^{-2}),
\end{align*}
and the error is uniform in the domain of $\zeta$ we are considering. 
Using that
\[
\frac{N^2}{2}+\frac{(\zeta+N)^2}{2}-(N+\tfrac{\zeta}{2})^2=\frac{\zeta^2}{4},
\]
we obtain 
\begin{align*}
L_N(\zeta)&=\tfrac{\zeta^2}{4}\log N-\tfrac{3}{8}\zeta^2+\tfrac{(\zeta+N)^2}{2}\log(1+\tfrac{\zeta}{N})-(N+\tfrac{\zeta}{2})^2\log(1+\tfrac{\zeta}{2N})+\tfrac{1}{6}\log(1+\tfrac{\zeta}{2N})\\
&\qquad -\tfrac{1}{12}\log(1+\tfrac{\zeta}{N})+O(N^{-2})\\
&=\tfrac{\zeta^2}{4}\log N+O(\tfrac{\zeta^3}{N})+O(\tfrac{\zeta}{N})+O(N^{-2})
\end{align*}
where the implied constants are universal. The claim follows from these asymptotics using that according to  \eqref{def:XN}, 
\(
\E\big[ e^{\zeta\X_N(\theta)} \big]
=\E |\det(1-e^{i\theta}U_N)|^{\sqrt{2}\zeta}
= \Psi(\zeta) \exp L_N(\sqrt 2\zeta)
\)
for $\mathrm{Re}(\zeta)>-1$. 

Finally, the bound on $\Psi$ follows immediately from \eqref{eq:Gasy};
for any $\varkappa>2$,  there exists a constant $C_\varkappa$ such that for $\mathrm{Re}(z)\geq 0$, 
\[
|G(1+z)|^{\pm1}\leq C_\varkappa\exp (|z|^\varkappa) . \qedhere
\]
\end{proof}

\begin{remark}
Let us point out that from the asymptotics of the proof, the condition that $\zeta=o(N^{1/3})$ seems to be sharp here in that one starts getting order one corrections to the asymptotics when $\zeta$ is on the scale $N^{1/3}$.
\end{remark}

The following consequence of the results of \cite{DIK} (see also \cite{Widom,Ehrhardt,BF} for closely related results) is also directly relevant\footnote{For any  function $V\in \mathcal{C}(\T\to\C)$, 
we denote $ \Tr V(U_N) = \sum_{k=1}^N V(\vartheta_k)$ where $\{e^{i\vartheta_k}\}_{k=1}^N$ denotes the eigenvalues of the random matrix $U_N$.}.

\begin{theorem}[Deift, Its, Krasovsky, \cite{DIK}] \label{thm:DIK}
Let $p\in\N_0$, $V\in \mathcal{C}^\infty(\T\to\C)$ with $\widehat{V}_0 = 0$ and $\Psi$ as in Lemma~\ref{le:morris}. 
Then, it holds locally uniformly for $\zeta_1,\dots,\zeta_p \in \{z\in\C : \Re z>-1\}$
and $x \in \{x\in \T^p  \mathrm{dist}(x_i,x_j) >0 \text{ for }1\le i<j \le p\}$, as $N\to\infty$,
\begin{equation}\label{DIK} 
\begin{aligned}
\E\big[ e^{  \sum_{j=1}^p \zeta_j \X_N(x_j) + \Tr V(U_N)} \big] 
= &\exp\Big(\sum_{1\le j \le p} \big( \frac{\zeta_j^2}{2} \log N
-  \frac{ \zeta_j}{\sqrt{2}} V(x_j) \big)
+ \sum_{k\in\N} k \widehat{V}_k  \widehat{V}_{-k}
\Big) \\
&\times\exp\Big(\sum_{1\le i<j \le p}\zeta_i\zeta_j C_{\X}(x_i,x_j) 
+o(1) \Big)
\prod_{j=1}^p\Psi(\zeta_j) .
\end{aligned}
\end{equation}
\end{theorem}

Note that according to Lemma~\ref{lem:covcue}, we have
$\X_{N,\delta}(x) =- \sqrt2 \Tr\big[ C_{\X,\delta}(U_N,x) \big]$ for $x\in\T$ and $\delta>0$ so that by 
linearity, 
$ \langle \X_N , \mathrm{f}_{\delta,z}  \rangle = \Tr V(U_N)$ 
where 
\begin{equation} \label{def:V}
V(\theta)
=  -\sqrt2  \sum_{j=1}^q \xi_j C_{\X,\delta_j}(\theta,z_j)
=  -\sqrt2\int C_\X(\theta,u) \mathrm{f}_{\delta,z}(u) d u .
\end{equation}
In particular, $V\in \mathcal{C}^\infty(\T\to\R)$ with $\widehat{V}_0 = 0$ and we claim that
\begin{equation}\label{Scov} 
\begin{aligned} 
\sum_{k\in\N} k \widehat{V}_k  \widehat{V}_{-k}
 = \int C_\X(\theta,x) V'(\theta) V'(x)  \frac{d\theta}{2\pi}   \frac{dx}{2\pi} 
= \frac12 \E \langle \X , \mathrm{f}_{\delta,z}  \rangle^2 .
\end{aligned}
\end{equation}

We can view $C_\X$, \eqref{kernelFourier},  as the kernel of a (bounded) integral operator $\mathcal{K}$ on $L^2(\T)$ whose action on the Fourier basis is given by 
$\mathcal{K}(e^{ik\theta}) =  \frac{1}{2|k|}e^{ik\theta}$ for $k\in\Z\setminus\{0\}$
and $\mathcal{K}(1)=0$.
Then, the first identity in \eqref{Scov} is equivalent to
\[
\langle V',   \mathcal{K}V' \rangle_{L^2} = \sum_{k\in\Z} \frac{1}{2|k|} \big|\widehat{V'}_k \big|^2 = \sum_{k\in\N} k \widehat{V}_k  \widehat{V}_{-k}
\]
using that $\widehat{V'}_k  = ik  \widehat{V}_k$ and $\overline{\widehat{V}_k} = \widehat{V}_{-k}$ since $V$ is real-valued. 

By definition,  $V = -\sqrt{2}\, \mathcal{K} \mathrm{f}_{\delta,z} $
so that $V' = -\frac1{\sqrt2} \mathcal{H} \mathrm{f}_{\delta,z}  $ where 
$\mathcal{H}$  is the Hilbert transform\footnote{The Hilbert transform is a (Fourier) integral operator on $L^2(\T)$ whose action on the Fourier basis is given by $\mathcal{H}(e^{ik\theta}) = i \operatorname{sgn}(k)e^{ik\theta}$ with $ \operatorname{sgn}(0)=0$. In particular $\mathcal{H}$ is unitary on the Hilbert space 
$\{f\in L^2(\T) : \widehat{f}_0=0\}$.} 
on $\T$, since $\mathcal{K}'  = \frac12 \mathcal{H}$ as (Fourier) integral operator. 
Hence, the second second identity  in \eqref{Scov} is equivalent to
\[
\langle V',   \mathcal{K}V' \rangle_{L^2}
=  \tfrac1{\sqrt2} \langle V',   \mathcal{K} \mathcal{H} \mathrm{f}_{\delta,z}  \rangle_{L^2}
=  \tfrac1{\sqrt2} \langle \mathcal{H}^* V',   \mathcal{K} \mathrm{f}_{\delta,z}  \rangle_{L^2}
= \tfrac12  \langle  \mathrm{f}_{\delta,z}  , \mathcal{K} \mathrm{f}_{\delta,z}  \rangle_{L^2} , 
\]
where we used that the operators $ \mathcal{K} ,\mathcal{H}$ commute and 
$ \mathcal{H}^* V' = -\frac1{\sqrt2} \mathrm{f}_{\delta,z}  $ since the Hilbert transform is unitary (on the appropriate subspace of $L^2(\T)$). 

\medskip

Thus, taking $V$ as above, $p=2$, and comparing the asymptotics \eqref{DIK} to \eqref{mom1}, we find that for fixed $\delta \in (0,1]^q$ and 
$\xi \in \R^q$, it holds uniformly for $z \in \T^q$ and $\zeta_1, \zeta _2 \in \mathcal{D}_R  $ and $(x_1,x_2) \in \mathcal{A}_c$, see \eqref{eq:Dr}--\eqref{eq:Ac}, 
\[
\Psi_{N,\delta}^{(\zeta_1,\zeta_2,\xi)}(x_1, x_2;z) 
=  \Psi(\zeta_1) \Psi(\zeta_2)(1+o(1)) . 
\]

Hence, in order to obtain the asymptotics from Assumption~\ref{ass:meso} (2) and Assumption~\ref{ass:macro} (1) to hold for the  log characteristic polynomial of CUE, we need to extend Theorem~\ref{thm:DIK} in two crucial ways; for any small $\eta>0$, 
\begin{itemize}[leftmargin=5mm]
\item[(1)] the asymptotics \eqref{DIK} hold for $p=2$ in the merging regime, that is uniformly for $x_1,x_2 \in \T$ with $|x_1-x_2| \ge N^{\eta-1}$ and for a trigonometric polynomial $V =V_N$ with $V_N(\theta) = \Re\big( \sum_{k \le N^{1-\eta}} \widehat{V}_k e^{i k\theta} \big)$ and $|\widehat{V}_k | \le C/k$. 
This case yields Assumption~\ref{ass:meso}. 
\item[(2)] the asymptotics \eqref{DIK} hold for $p=2$ in case $\Im \zeta_j$ is allowed to grow mildly as $N\to\infty$, that is uniformly for $\zeta_1,\zeta_2\in \mathcal D_{R_N}$ 
This case yields Assumption~\ref{ass:macro}. 
\end{itemize}

We note that these extensions are mostly technical work and, to derive them, we will rely on the method from \cite{DIK}.
In Section~\ref{sec:Toep} and Appendix~\ref{sec:OP}, we review the basics of this method which relies on the connection between the circular unitary ensembles, orthogonal polynomials on the unit circle and the associated Riemann-Hilbert problems.
In Sections~\ref{sec:trans}--\ref{sec:local}, we review the steepest descent method for these problems, including the global parametrix and local parametrix (around the singularities).
In Section~\ref{sec:R}, we present the small norm problem and the main differences with the case already treated in \cite{DIK}. 
Based on these asymptotics the proofs of  Assumption~\ref{ass:meso} (2) and Assumption~\ref{ass:macro} (1) are finalized in 
Sections~\ref{sec:Rmeso} and~\ref{sec:Rmacro} respectively.

\subsection{Toeplitz determinants and differential identities} \label{sec:Toep}

A specificity of the CUE (and other unitary invariant ensembles) is their determinantal structure and, in particular, the Heine-Szeg\H{o} identity which relates the Laplace transform of linear statistics of the random matrix $U_N$ to Toeplitz determinants (for a proof, see e.g. \cite[Theorem 1]{BD}).
For $n\in\N$, let $\{e^{i\vartheta_k}\}_{k=1}^n$ be the eigenvalues of $U_n$. 
The statement is as follows: for any $F \in L^1(\T\to\C)$, 
\begin{equation}\label{eq:hs}
\E\big[{\textstyle \prod_{k=1}^n } F(\vartheta_k)\big]=\det\left( \widehat{F}_{j-k}\right)_{j,k=0}^{n-1}=:D_n(F).
\end{equation}
This makes a connection with orthogonal polynomials on $\T$ that we now review.

Note that using the  Heine-Szeg\H{o} formula, the LHS of \eqref{mom1} equals 
\begin{equation} \label{HS}
\E\big[ e^{ \zeta_1 \X_N(x_1) + \zeta_2 \X_N(x_2)+ \sum_{j=1}^q\xi_j \X_{N,\delta_j}(z_j)} \big]   = D_N(F_{1})
\end{equation}
where, with $V$ as in \eqref{def:V}, 
\begin{equation}\label{eq:Vt}
F_{t}(\theta)=|e^{i\theta}-e^{ix_1}|^{\sqrt{2}\zeta_1}|e^{i\theta}-e^{ix_2}|^{\sqrt{2}\zeta_2}  e^{tV(\theta)}, \qquad \theta\in\T ,\, t\in [0,1].
\end{equation}
In particular, $V$ is a trigonometric polynomial of degree $\le N^{\eta-1}$ for $\delta_1,\cdots, \delta_q \ge N^{-1+\eta}$. Moreover, this symbol $F_t$, as well as $Y$ defined below, depend on the auxiliary parameters $\zeta_1,\zeta_2\in \mathcal{D}_{\infty}$, $x_1,x_2\in\T$ with $x_1\neq x_2$, $\delta\in(0,1]^q$ and $z\in\T^q$ which are allowed to vary with the dimension $N\in\N$. 
 
To analyze the asymptotics of $D_N(F_t)$, we rely on the Riemann-Hilbert problem associated with  orthogonal polynomials.

The relevant definitions are collected in the Appendix~\ref{sec:OP} and this problem reads as follows (the original connection between such problems and orthogonal polynomials is due to \cite{FIK});

\begin{problem}\label{pb:Y}
Let $F=F_{t}$ be as in \eqref{eq:Vt}{, let $\U=\{z\in \C: |z|=1\}$ be the unit circle, }
and for $n\in\N$ and $t\in[0,1]$ let $Y = Y_{n,t}$ solve;
\begin{enumerate}
\item  $Y:\C\setminus {\U} \to \C^{2\times 2}$ is analytic.
\item $Y$ has continuous boundary values $Y_+$ $(Y_-)$ from $\{|z|<1\}$ $(\{|z|>1\})$  which satisfy
\begin{equation*}\label{eq:Yjump}
		Y_+(e^{i\theta})=Y_-(e^{i\theta})\begin{pmatrix}
			1 & e^{-i n\theta}F(\theta)\\
			0 & 1
		\end{pmatrix}.
\end{equation*}
\item As $z\to \infty$, 
\begin{equation*}\label{eq:Yasy}
		Y(z)=(I+O(z^{-1}))z^{n\sigma_3} \qquad \text{where} \qquad z^{n\sigma_3}=\begin{pmatrix}
			z^n & 0\\
			0 & z^{-n}
		\end{pmatrix}.
\end{equation*}
\end{enumerate}
\end{problem}

Here, we do not assume that this problem has a solution -- this will follow from our analysis. However, this issue is directly related to the existence of certain (orthogonal) polynomials and, if it exists, this solution is unique and given explicitly by \eqref{eq:Ydef}. 
Regardless of this consideration, there are standard methods to derive the asymptotics of the matrix $Y_{n,t}$ as $n\to\infty$ through a steepest descent analysis pioneered by Deift and Zhou \cite{DZ}. This analysis is based on deforming the jump contour in a suitable neighborhood of ${\U}$ and we review it in the next sections. 
We have introduced the parameter $t\in[0,1]$ to make an interpolation (one could also consider different scheme, e.g. \cite{DIK2}).
This allows us to relate the Toeplitz determinants $D_N(F_{1})$ to $D_N(F_{0})$ (where $V=0$) through the following differential identity; 
 
\begin{lemma}\label{le:di2}
Let $V$ be a trigonometric polynomial  and $F_{t}$ be given by \eqref{eq:Vt}. 
Assume that $D_n(F_{t})\neq 0$ for all $n=1,...,N$. Then, for $t\in[0,1]$, 
\[
\partial_t \log D_N(F_{t})=\frac{1}{2\pi i}\oint_{{\U}}z^{-N}\left(Y_{N,t}(z)^{-1}\partial_z Y_{N,t}(z)\right)_{21}V(z) F_{t}(z)dz,
\]
where $Y= Y_{N,t}$ is the unique solution of Problem~\ref{pb:Y}. 
\end{lemma}

Although it is classical, the proof of Lemma~\ref{le:di2} is given in the Appendix~\ref{sec:OP} for completeness. 
This relies on the fact that the matrix $Y = Y_{N,t}$ is built from the orthogonal polynomials with respect to the (complex) weight $F_{t}$. 
\medskip

We can also relate  $D_N(F_{0})[\zeta_1,\zeta_2]$ to $D_N(F_{0})[\zeta_1,0]$ for $\zeta_1,\zeta_2\in \mathcal{D}_{\infty}$ using another differential identity from \cite{DIK2}. 
This is especially useful because $D_N(F_{0})[\zeta_1,0] = \E\big[ e^{\zeta_1\X_N(x_1)} \big]$  is explicitly given by Lemma \ref{le:morris}.

\begin{lemma}[Deift, Its, Krasovsky, \cite{DIK2}]\label{le:di}
Let $F_{0}$ be as in \eqref{eq:Vt} and  assume that $D_n(F_{0})\neq 0$ for all $n=1,...,N+1$. 
Let 
\(
\chi_N^{-2}=\frac{D_{N+1}(F_{0})}{D_N(F_{0})}
\)
for $N\in\N$. 
Then, for $\zeta_1,\zeta_2\in \mathcal{D}_{\infty}$, 
\begin{align} \label{diff2}
&\partial_{\zeta_2}\log D_N(F_{0}) =-2 \partial_{\zeta_2} \log(\chi_N)\left(N+\tfrac{\zeta_1}{{\sqrt{2}}}+\tfrac{\zeta_2}{{\sqrt{2}}}\right)\\
\notag&+ \sum_{j=1,2}\frac{\zeta_j}{{\sqrt{2}}}\left(  \partial_{\zeta_2}(\chi_NY_{11,N}(e^{i\theta_j}))e^{i\theta_j}\chi_N^{-1}Y_{22,N+1}(e^{i\theta_j})-\partial_{\zeta_2}(\chi_N^{-1}Y_{21,N+1}(e^{i\theta_j}))\chi_N Y_{12,N}(e^{i\theta_j})\right)
\end{align}
where the matrix $Y_n=Y_{n,0}$ for $n=N,N+1$  is the unique solution of Problem~\ref{pb:Y} with $t=0$. 
\end{lemma}

\begin{remark}
In formula \eqref{diff2}, it is relevant to note that for $\mathrm{Re}(\zeta_j)>0$, the entries
$Y_{11}$ and $Y_{21}$ are polynomials
and there is no issue in evaluating $Y_{12}$ and $Y_{22}$ at $e^{ i\theta_j}$
since the zero of $F_0$ at $e^{i\theta_j}$ fixes the non-integrable singularity; cf.~Appendix \ref{sec:OP} and in particular formula \eqref{eq:Ydef}. 
\end{remark}

We now turn to transforming our Riemann-Hilbert problem into a form that will eventually allow an approximate solution.

\subsection{Transforming the Riemann-Hilbert problems} \label{sec:trans}

 The idea of the steepest descent analysis of Deift and Zhou is to perform transformations to the Riemann-Hilbert problem by modifying the jump contours so that the jump matrices are close to the identity matrix and the solution is normalized to be the identity matrix at infinity. Then the problem can be solved asymptotically with a suitable Neumann series. 
In particular, the Problem~\ref{pb:Y} at hand for $n=N, N+1$ has an oscillatory jump matrix on $\{|z|=1\}$. 
Hence, the first step of this transformation procedure consists in moving this contour to the regions $\{|z|<1\}$ and $\{|z|>1\}$ where the jump matrix will be exponentially decaying. This is known as ``opening lenses" in the Riemann-Hilbert literature; we  refer the reader to e.g. \cite{Deift,DIK,DIK2} and references therein for further details.

We are interested in $V$ as in \eqref{def:V}, but we can allow a more generic (real-valued) potential\footnote{The relationship is $V(\theta) = \widehat{V}(e^{i\theta})$ for $\theta \in \T$. In particular, $\widehat{V}_{-k} = \overline{\widehat{V}_k}$ for $k\in\N$ so that $V$ is real-valued on $\U$ and  $\widehat{V}(z)$ is analytic for $z\in\C\setminus\{0\}$.} in \eqref{eq:Vt}, 
\begin{equation} \label{pot}
\widehat{V}(z)=  \sum_{1\le |k| \le \Delta} \widehat{V}_k z^k , \qquad
\text{where $|\widehat{V}_k | \le C/k$}. 
\end{equation}
where
\begin{equation}\label{eq:Delta}
\Delta^{-1} \le \min(|e^{ix_1}-e^{ix_2}|,\min_{j=1,\dots,q}\delta_j) . 
\end{equation} 
In the sequel, $\Delta$ is either $N^{ 1-\eta}$ in case of Assumption~\ref{ass:meso} or fixed in case of Assumption~\ref{ass:macro}. 

\medskip 

Let $U_{\Delta,j}=\{w\in \C: |w-e^{ix_j}|\leq \tfrac{1}{2\Delta}\}$ for $j\in\{1,2\}$ and consider the following domain, 
\begin{align}\label{eq:L}
\widetilde{L}& = \left\lbrace z\in \C:1-\tfrac{\Delta^{-1}}{4}<|z|< 1+\tfrac{\Delta^{-1}}{4}\right\rbrace \setminus \cup_{j=1}^2 U_{\Delta,j} ; 
\end{align}
 see Figure \ref{fig:L2}.
We enlarge this set $\widetilde L$ by connecting it  to the points $e^{ix_j}$ suitably. More precisely, we draw certain contours (specifically defined in Section~\ref{sec:local}) from the points $\{(1\pm \tfrac{\Delta^{-1}}{4} ){\U}\}\cap\{\cup_{j=1}^2 U_{\Delta,j}  \}  $ to the points $e^{ix_j}$, this yields the set $L$.

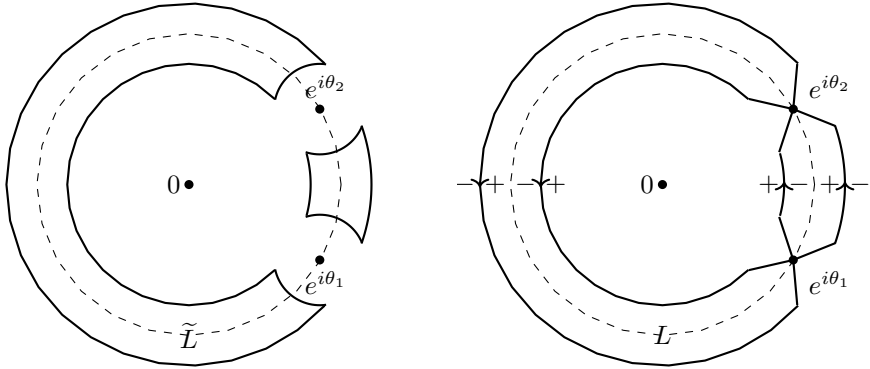
\begin{figure}
\begin{center}
\begin{tikzpicture}[xscale=0.02,yscale=0.02]
\clip (-135,-135) rectangle (170,135); 

\fill [color=black] (0,0) circle (3);

\draw [black,dashed,domain=0:360] plot ({100*cos(\x)}, {100*sin(\x)});

\draw [black,thick,domain=45:315] plot ({80*cos(\x)}, {80*sin(\x)});
\draw [black,thick,domain=42:318] plot ({120*cos(\x)}, {120*sin(\x)});

\draw [black,thick,domain=-15:15] plot ({80*cos(\x)}, {80*sin(\x)});
\draw [black,thick,domain=-19:19] plot ({120*cos(\x)}, {120*sin(\x)});

\fill [color=black] (86,50) circle (3);
\fill [color=black] (86,-50) circle (3);

\node at (0,-100) {\small $\widetilde{L}$};
\node at (-10,0) {\small $0$};
\node at (90,65) {\small $e^{i\theta_2}$};
\node at (90,-65) {\small $e^{i\theta_1}$};

\draw [black,thick,domain=82:168] plot ({86+30*cos(\x)}, {50+ 30*sin(\x)});
\draw [black,thick,domain=192:278] plot ({86+30*cos(\x)}, {-50+ 30*sin(\x)});

\draw [black,thick,domain=252:338] plot ({86+30*cos(\x)}, {50+ 30*sin(\x)});
\draw [black,thick,domain=22:108] plot ({86+30*cos(\x)}, {-50+ 30*sin(\x)});

\end{tikzpicture}
\begin{tikzpicture}[xscale=0.02,yscale=0.02]
\clip (-135,-135) rectangle (170,135); 

\draw [black,dashed,domain=0:360] plot ({100*cos(\x)}, {100*sin(\x)});

\draw [black,thick,domain=45:315] plot ({80*cos(\x)}, {80*sin(\x)});
\draw [black,thick,domain=42:318] plot ({120*cos(\x)}, {120*sin(\x)});

\draw [black,thick,domain=-15:15] plot ({80*cos(\x)}, {80*sin(\x)});
\draw [black,thick,domain=-19:19] plot ({120*cos(\x)}, {120*sin(\x)});

\fill [color=black] (86,50) circle (3);
\fill [color=black] (86,-50) circle (3);
\draw[->,thick] (-80,2) -- (-80,-2);
\draw[->,thick] (-120,2) -- (-120,-2);

\draw[->,thick] (80,-2) -- (80,2);
\draw[->,thick] (120,-2) -- (120,2);

\node at (-70,0) {\small $+$};
\node at (-90,0) {\small $-$};

\node at (-110,0) {\small $+$};
\node at (-130,0) {\small $-$};

\node at (70,0) {\small $+$};
\node at (90,0) {\small $-$};

\node at (110,0) {\small $+$};
\node at (130,0) {\small $-$};

\fill [color=black] (0,0) circle (3);

\node at (0,-100) {\small $L$};
\node at (-10,0) {\small $0$};
\node at (110,65) {\small $e^{i\theta_2}$};
\node at (110,-65) {\small $e^{i\theta_1}$};

\draw [black, thick] ({86+30*cos(193)}, {-50+ 30*sin(193)}) -- (86,-50);
\draw [black, thick] ({86+30*cos(275)}, {-50+ 30*sin(275)}) -- (86,-50);

\draw [black, thick] ({86+30*cos(85)}, {50+ 30*sin(85)}) -- (86,50);
\draw [black, thick] ({86+30*cos(167)}, {50+ 30*sin(167)}) -- (86,50);

\draw [black, thick] ({86+30*cos(22)}, {-50+ 30*sin(22)}) -- (86,-50);
\draw [black, thick] ({86+30*cos(108)}, {-50+ 30*sin(108)}) -- (86,-50);

\draw [black, thick] ({86+30*cos(252)}, {50+ 30*sin(252)}) -- (86,50);
\draw [black, thick] ({86+30*cos(338)}, {50+ 30*sin(338)}) -- (86,50);

\end{tikzpicture}
\end{center}
\vspace{-0.3cm}
\caption{Left: An illustration of the set $\widetilde{L}$. Right: An illustration of the set $L$ along with the orientation of its boundary contour for the $S$-RHP. {In both cases, the dashed circle is the unit circle $\U$.}}
\label{fig:L2}
\end{figure}

In order to ``open lenses", we must analytically continue the symbol \eqref{eq:Vt} to a neighborhood of ${\U}$. 
Let 
$\V(z)=  \sum_{1\le |k| \le \Delta} \widehat{V}_k z^k$, this is a continuation of $V$ in $\C\setminus\{0\}$ as a Laurent polynomial. 
Then, this is a question of analytically continuing the functions $|z-e^{ix_j}|^{\zeta_j}$ into a neighborhood of ${\U}$, excluding some appropriately chosen branch cuts. We follow the construction from \cite[Section 4]{DIK}.
We consider the functions 
\begin{equation}
\begin{aligned}
\label{eq:D}
\mathcal D_{\rm in}(z) & =e^{t\sum_{k=1}^\Delta \widehat{V}_k  z^k}\frac{(z-e^{ix_1})^{\zeta_1/\sqrt{2}}(z-e^{ix_2})^{\zeta_2/\sqrt{2}}}{e^{i\zeta_1x_1/\sqrt{2}}e^{i\pi \zeta_1/\sqrt{2}} e^{i\zeta_2x_2/\sqrt{2}}e^{i\pi\zeta_2/\sqrt{2}}} ,
\\
\mathcal D_{\rm out}^{-1}(z) & =e^{t\sum_{k=1}^\Delta \widehat{V}_{-k}  z^{-k}}\frac{(z-e^{ix_1})^{\zeta_1/\sqrt{2}}(z-e^{ix_2})^{\zeta_2/\sqrt{2}}}{z^{\zeta_1/\sqrt{2}} z^{\zeta_2/\sqrt{2}}} ,
\end{aligned}
\end{equation}
where the branches of the roots are defined as follows: the cut of $(z-e^{ix_j})^{\zeta_j/2}$ is taken to be on the half line $e^{ix_j}\times [1,\infty)$ and the branch is fixed by requiring that $\mathrm{arg}(z-e^{ix_j})=2\pi$ on the half line parallel to the real axis going from $e^{ix_j}$ to the right. For $z^{\zeta_j/\sqrt2}$ we choose the cut to be the half line $e^{ix_j}\times [0,\infty)$ and one fixes the branch by requiring that  $\mathrm{arg}(z)\in (x_j,x_j+2\pi)$. 

With these conventions, the functions \eqref{eq:D} are both analytic on $\C\setminus (e^{ix_1}\times [0,\infty)\cup e^{ix_2}\times [0,\infty))$. 
More specifically, $\mathcal D_{\rm in}$ is analytic in $|z|<1$ and $\mathcal D_{\rm out}$ is analytic and non-zero in $|z|>1$.

Then, by the Sokhotski-Plemelj identity, one has 
\begin{equation}\label{eq:f}
\begin{aligned}
f(z) & = |z-e^{ix_1}|^{\sqrt{2}\zeta_1}|z-e^{ix_2}|^{\sqrt{2}\zeta_2}  e^{t\V(z)} \\
& =\mathcal D_{\rm in}(z)  \mathcal D_{\rm out}^{-1}(z) , 
&z\in\C\setminus \{e^{ix_1}\times [0,\infty)\cup e^{ix_2}\times [0,\infty)\}
\end{aligned}
\end{equation}
This provides the required analytic continuation of $F_{t}$ in a neighborhood of ${\U}$. Note that we do not emphasize that this function (and all other quantities defined in terms of it, such as $S$ below) depends on the parameters $t\in[0,1]$, $\zeta_1,\zeta_2 \in \mathcal{D}_\infty$, $x_1,x_2 \in \T$ with $x_1\neq x_2$, and $\Delta$ (which may also depend on~$N$). 

{
Moreover,  according to \eqref{pot}, $\sum_{1\le |k| \le \Delta} |\widehat{V}_k| (1+\frac{1}{4\Delta})^k \le C \log \Delta $
and $|z-e^{ix_j}|^{\pm |\zeta|} \le \Delta^{R_N}$ for $z\in\partial \widetilde L$ and $\zeta\in \mathcal{D}_{R_N}$. 
So, if $R_N= o(\log N)$, uniformly for $z\in\partial \widetilde L$ (and all other relevant parameters), 
\begin{equation} \label{Dest}
|\mathcal D_{\rm in}(z)|^{\pm 1},|\mathcal D_{\rm out}(z)|^{\pm 1}=\exp O\big((\log N)^2 \big).
\end{equation}
}

In terms of terms of the analytic function $f$ and the set $L$, we consider the following Riemann-Hilbert problem; 
\begin{problem}\label{pb:S}
Let   $\Sigma_S=\partial L\cup {\U}$ be oriented as in Figure \ref{fig:L2}. 
Let $S = S_{n}$ for $n\in\N$ solve;
\begin{enumerate}
\item  $S:\C\setminus \Sigma_S\to \C^{2\times 2}$ is analytic.
\item $S$ has continuous boundary values on $\Sigma_S\setminus \{e^{i\theta_1},e^{i\theta_2}\}$, denoted $S_+,S_-$  which satisfy
\begin{align*}
S_+(z) &=S_-(z)\begin{pmatrix}
0 & f(z)\\
-f(z)^{-1} & 0
\end{pmatrix}, & z\in {\U}\setminus \{e^{ix_1},e^{ix_2}\}, \\
S_+(z) &=S_-(z)\begin{pmatrix}
1 & 0\\
z^nf(z)^{-1} & 1
\end{pmatrix}, & z\in \partial L, \quad |z|<1, \\
S_+(z)&=S_-(z)\begin{pmatrix}
1 & 0\\
z^{-n}f(z)^{-1} & 1
\end{pmatrix}, & z\in \partial L, \quad |z|>1, \\
\end{align*} 

\item $S(z)=I+O(z^{-1})$ as $z\to \infty$.
\item  For $j=1,2,$ as $z\to e^{ix_j}$, 
\begin{equation*}\label{eq:Sloc}
S(z)=\begin{cases}
\begin{pmatrix}O(|z-e^{ix_j}|^{-\sqrt{2}\mathrm{Re}\zeta_j}) & O(1)\\
O(|z-e^{ix_j}|^{-\sqrt{2}\mathrm{Re}\zeta_j}) & O(1)
\end{pmatrix}, & z\in L\setminus {\U}\\
O(1), & z\notin L
\end{cases}.
\end{equation*}
These $O(\cdot)$-terms involve only $z$ and do not require any uniformity in $n$,  $t$, etc.
\end{enumerate}
\end{problem}

It is a standard fact (that the reader will have no difficulty verifying) that 
Problem~\ref{pb:Y} and Problem~\ref{pb:S} are related by the following transformation,
\begin{equation} \label{eq:S}
S(z)=\begin{cases}
Y(z), & z\notin L \quad \text{and} \quad |z|<1\\
Y(z)z^{-n\sigma_3}, & z\notin L \quad \text{and} \quad |z|>1\\
Y(z) \begin{pmatrix}
1 & 0 \\
-z^n f(z)^{-1} & 1
\end{pmatrix}, & z\in L \quad \text{and} \quad |z|<1\\
Y(z)z^{-n\sigma_3} \begin{pmatrix}
1 & 0 \\
z^{-n} f(z)^{-1} & 1
\end{pmatrix}, & z\in L \quad \text{and} \quad |z|>1
\end{cases}.
\end{equation}
In particular, if the solution $Y$ of Problem~\ref{pb:Y} exists, then it is unique and $S$ given by \eqref{eq:S} is the unique solution of Problem~\ref{pb:S}, including the asymptotic conditions  (3) and (4). 
Conversely, these conditions guarantee uniqueness of a solution of Problem~\ref{pb:S} (which would not hold otherwise -- this is common behavior in Riemann-Hilbert problems with singular symbols; for further discussion, see e.g. \cite[Section 5]{Kuijlaars} for the case when the symbol is supported on an interval) and, by solving Problem~\ref{pb:S}, one can recover $Y$. 

\medskip

In the next sections, we explain how to construct a solution of Problem~\ref{pb:S}. 
Note that if $\Delta \ll N$ for $n\in\{N,N+1\}$, by (2), the jumps of $\partial L$ are ``exponentially small" and can be neglected except in neighborhoods of the singularities $\{e^{ix_1},e^{ix_2}\}$.
Then, this construction involves two ingredients;
\begin{itemize}[leftmargin=5mm]
\item a \emph{global parametrix} which models  the jumps over ${\U}$ (cf.~Section~\ref{sec:global}).
\item \emph{local parametrices} to adjust for the jumps in neighborhoods of $e^{ix_1}$ and $e^{ix_2}$ (cf.~Section~\ref{sec:local}).
\end{itemize}
As a final step, one patches together these parametrices into a \emph{small norm problem} (cf.~Section~\ref{sec:R}) to obtain an (asymptotic) solution of Problem~\ref{pb:S}, including conditions  (3) and (4).

\subsection{The global parametrix} \label{sec:global}
Let us ignore the jumps across $\partial L$ in order to find an approximate of Problem~\ref{pb:S}. 
In terms of \eqref{eq:f}, define
\begin{equation}\label{eq:global}
\mathcal N(z)=\begin{cases}
\mathcal D_{\mathrm{in}}(z)^{\sigma_3}\begin{pmatrix}
0 & 1\\
-1 & 0
\end{pmatrix}, & |z|<1\\
\mathcal D_{\mathrm{out}}(z)^{\sigma_3}, & |z|>1
\end{cases}.
\end{equation}
By construction, this function is analytic off of ${\U}$, and since $\lim_{z\to \infty}\mathcal D_{\mathrm{out}}(z)={1}$, we have that  $\mathcal N(z)=I+O(z^{-1})$ as $z\to \infty$. 
Moreover, this function has a jump on ${\U}\setminus\{e^{ix_1},e^{ix_2}\}$, by \eqref{eq:f}, 
\[
\mathcal N_+(z)=\mathcal N_-(z)f(z)^{\sigma_3}\begin{pmatrix} 0 & 1\\
-1 & 0\end{pmatrix} =\mathcal N_-(z)\begin{pmatrix} 0 & f(z)\\
-f(z)^{-1} & 0\end{pmatrix},
\]
which is exactly the same as the jump of $S$ across ${\U}$.

We emphasize again that \eqref{eq:global} only provides a good approximation for $S$ away from the singularities $\{e^{ix_1},e^{ix_2}\}$ and we need  different parametrices there.

\subsection{The local parametrix} \label{sec:local}
We  turn now to the approximations at $e^{ix_j}$. 
There are various equivalent ways to represent the local parametrix. In \cite{MMS}, a solution is constructed in terms of Bessel and Hankel functions, while in \cite{DIK}, one is constructed in terms of confluent hypergeometric functions. In \cite{CK}, there is a slightly different representation in terms of hypergeometric functions. We will follow \cite[Section 4.2]{DIK}, since we will rely heavily on other related results proven in \cite{DIK,DIK2}.

Recall that $U_{\Delta,j}=\{w\in \C: |w-e^{ix_j}|\leq \tfrac{1}{2\Delta}\}$ for $j\in\{1,2\}$. 
Our goal is to construct a function $P:U_{\Delta,j} \to \C^{2\times 2}$ such that $P$ has the same jumps as $S$, 
$z\mapsto S(z)P(z)^{-1}$ is analytic in $U_{\Delta,j}$, and $P(z)\mathcal N(z)^{-1}=I+o(1)$ uniformly in $z\in \partial U_{\Delta,j}$ (and the other relevant parameters -- we will be more precise later on). It is this last part which requires extra care compared to the existing literature; e.g.~\cite{DIK,DIK2}. The point being that in \cite{DIK,DIK2}, the analysis is performed for fixed $\zeta_1,\zeta_2$, while in the context of  Assumption~\ref{ass:macro}, $|\Im\zeta_1|,|\Im\zeta_2|$ are allowed to grow mildly with $N$.

\medskip

The first step is to specify, how we define the set $L$, \eqref{eq:L}, inside $U_{\Delta,j}$. For this purpose, let us define inside $U_{\Delta,j}$ the function $\xi_{N,j}(z) = N\log(z e^{-ix_j})$, where we consider the principal branch of the logarithm. This variable will serve as \emph{local conformal coordinate} in $U_{\Delta,j}$. 
We require that $\partial L$, which consists of four (simple) curves connecting  the points $\{(1\pm \tfrac{\Delta^{-1}}{4} ){\U}\}\cap \{z: |z-e^{ix_j}|=\tfrac{\Delta^{-1}}{2}\}$ to  $e^{ix_j}$, which gets mapped to (parts of) the rays $e^{ik \pi/4}\times(0,\infty)$ (with $k=1,3,5,8$) -- see Figure \ref{fig:local}

\begin{figure}
\begin{center}
\begin{tikzpicture}[xscale=0.02,yscale=0.02]
\clip (-135,-135) rectangle (170,135); 

\fill [color=black] (0,0) circle (3);

\draw [black,dashed,domain=0:360] plot ({100*cos(\x)}, {100*sin(\x)});
\draw [black,thick,domain=-30:30] plot ({-300+300*cos(\x)}, {300*sin(\x)});

\draw [black,dashed,domain=-30:30] plot ({-370+300*cos(\x)}, {300*sin(\x)});
\draw [black,dashed,domain=-30:30] plot ({-230+300*cos(\x)}, {300*sin(\x)});

\draw [black,thick,domain=0:95] plot ({-70+70*cos(\x)}, {64*sin(\x)});
\draw [black,thick,domain=180:99] plot ({70+70*cos(\x)}, {80*sin(\x)});

\draw [black,thick,domain=0:-95] plot ({-70+70*cos(\x)}, {64*sin(\x)});
\draw [black,thick,domain=180:260] plot ({70+70*cos(\x)}, {80*sin(\x)});

\node at (-20,0) {\small $e^{ix_j}$};   
\end{tikzpicture}
\begin{tikzpicture}[xscale=0.02,yscale=0.02]
\clip (-135,-135) rectangle (170,135); 

\fill [color=black] (0,0) circle (3);

\draw[black, thick] (-100,-100) -- (100,100);
\draw[black, thick] (-100,100) --   (100,-100);

 \draw [-{Stealth[length=3mm]}] (-50,50) -- (-50.1,50.1);
 \draw [-{Stealth[length=3mm]}] (50,50) -- (50.1,50.1);
 \draw [-{Stealth[length=3mm]}] (-50.1,-50.1) -- (-50,-50);
 \draw [-{Stealth[length=3mm]}] (50.1,-50.1) -- (50,-50);
 \draw [-{Stealth[length=3mm]}] (0,50) -- (0,50.1);
 \draw [-{Stealth[length=3mm]}] (0,-50.1) -- (0,-50);

\draw[black, thick] (0,150) -- (0,-150);
\node at (-20,0) {\small $0$};   
\end{tikzpicture}
\end{center}
\caption{Left: an illustration of the jump contour of $S$ (and $P$) inside of $U_{\Delta,j}$ (solid) as well as $\partial U_{\Delta,j}$ and the circles $(1\pm \frac{1}{4\Delta})\U$ (dashed). Right: an illustration of the image of this jump contour under the mapping $\xi_{N,j}$.}\label{fig:local}
\end{figure}
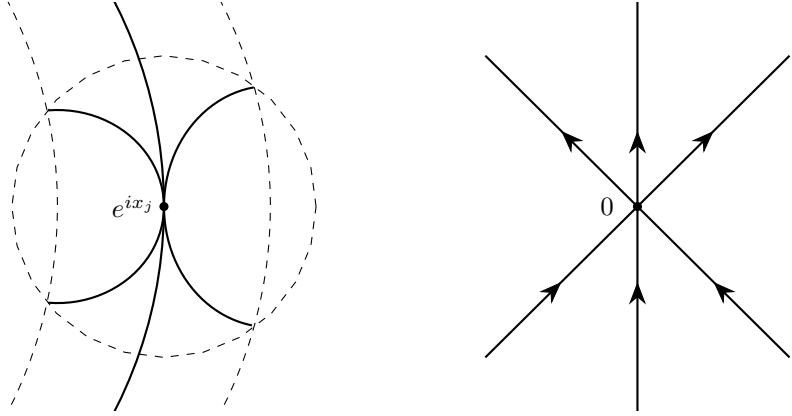

Another important ingredient to define the local parametrix is a function that is in a sense an analytic continuation of $f^{1/2}$ off of the unit circle, but different from the factorization \eqref{eq:f}. More precisely, let us write $I,II,III,IV,V,VI,VII,VIII$ for the octants of the complex $\xi$-plane ordered in the counter clockwise direction, and enumerated such that $I=\{\xi\in \C: \arg(\xi)\in (\pi/2,3\pi/4)\}$ and so on {(see Figure \ref{fig:oct})}.

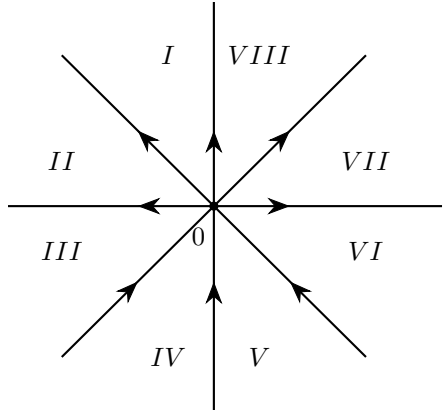
\begin{figure}
\begin{center}
\begin{tikzpicture}[xscale=0.02,yscale=0.02]
\clip (-135,-135) rectangle (170,135); 

\fill [color=black] (0,0) circle (3);

\draw[black, thick] (-100,-100) -- (0,0) -- (100,100);
\draw[black, thick] (-100,100) -- (100,-100);

\draw[black, thick] (-150,0) -- (150,0);

\draw[black, thick] (0,150) -- (0,-150);
\node at (-10,-20) {\small $0$};   

\node at (-30,100) {\small $I$};   
\node at (-100,30) {\small $II$};   
\node at (-100,-30) {\small $III$};   
\node at (-30,-100) {\small $IV$};   

\node at (30,100) {\small $VIII$};   
\node at (100,30) {\small $VII$};   
\node at (100,-30) {\small $VI$};   
\node at (30,-100) {\small $V$};   

 \draw [-{Stealth[length=3mm]}] (-50,50) -- (-50.1,50.1);
 \draw [-{Stealth[length=3mm]}] (50,50) -- (50.1,50.1);
 \draw [-{Stealth[length=3mm]}] (-50.1,-50.1) -- (-50,-50);
 \draw [-{Stealth[length=3mm]}] (50.1,-50.1) -- (50,-50);
 \draw [-{Stealth[length=3mm]}] (0,50) -- (0,50.1);
 \draw [-{Stealth[length=3mm]}] (0,-50.1) -- (0,-50);
 \draw [-{Stealth[length=3mm]}] (50,0) -- (50.1,0);
 \draw [-{Stealth[length=3mm]}] (-50,0) -- (-50.1,0);

\end{tikzpicture}
\end{center}
\caption{The splitting of the complex plane into octants relevant for the function $\Psi$ in the local parametrix as well as the orientation of the corresponding jump contours.}\label{fig:oct}
\end{figure}

 Define for $z\in U_{\Delta,j}$
\begin{equation}\label{eq:F}
F_{N,j}(z)=e^{\frac{t}{2}V(z)}h_{1}(z)h_{2}(z)\times \begin{cases}
e^{-i\pi \zeta_j/{\sqrt{2}}}, & \xi_{N,j}(z)\in I,II,V,VI\\
e^{i\pi \zeta_j/{\sqrt{2}}}, & \xi_{N,j}(z)\in III,IV,VII,VIII
\end{cases},
\end{equation}
where
\begin{equation}
h_{j}(z)=\frac{(z-e^{ix_j})^{\zeta_j/{\sqrt{2}}}}{(z e^{ix_j}e^{i\ell_j(z)})^{\zeta_j/2{\sqrt{2}}}} , \qquad \ell_{j}(z)=\begin{cases}
3\pi, & 0<\mathrm{arg}(z)<x_j\\
\pi, & x_j<\mathrm{arg}(z)<2\pi
\end{cases}.
\end{equation}
with a suitable choice of the cut (see the discussion around \cite[(4.13)]{DIK} for details).

It is then argued in \cite[(4.21) and (4.22)]{DIK} that 
\begin{equation}\label{eq:Fsq}
\begin{aligned}
F_{N,j}(z)^2&=f(z)e^{-\pi i {\sqrt{2}} \zeta_j}, & \xi_{N,j}(z)\in I,II,V,VI ,\\
F_{N,j}(z)^2&=f(z)e^{\pi i{\sqrt{2}} \zeta_j}, & \xi_{N,j}(z)\in III,IV,VII,VIII , 
\end{aligned}
\end{equation}
and $F_{N,j}$ is analytic in $U_{\Delta,j} \setminus  \cup_{k=1}^8\{ \xi_{N,j} \in \Gamma_k\}$ where 
\[
\Gamma_k=e^{\frac{k+1}{8}2\pi i}\times (0,\infty) , \qquad \text{for $k=1,...,8$.}
\]
We orient $\Gamma_4,\Gamma_5,\Gamma_6$ towards the origin, and the others away from it ({see Figure \ref{fig:oct}}). Again, it is argued in \cite[Section 4.2]{DIK} that if we take the $+/-$-side to be the left/right side of the contour then $F_{N,j}$ has continuous boundary values and satisfies the jump conditions; 
\begin{align*}
F_{N,j,+}&=F_{N,j,-}e^{-\pi i  {\sqrt{2}}\zeta_j}, & \xi_{N,j}\in \Gamma_1,\\
F_{N,j,+}&=F_{N,j,-}e^{\pi i {\sqrt{2}}\zeta_j}, & \xi_{N,j}\in \Gamma_5,\\
F_{N,j,+}&=F_{N,j,-}e^{\pi i \zeta_j/{\sqrt{2}}},&   \xi_{N,j}\in \Gamma_3\cup \Gamma_7.
\end{align*}

The next  key ingredient in the construction of the local parametrix is a function built from the confluent hypergeometric function of the second kind (or Tricomi confluent hypergeometric function) --  let us write $\psi=\psi(a,c,\xi)$ for this hypergeometric function (often $\Psi$ or $U$ in the literature). More precisely, we define 
\[
\psi(a,c,\xi)=\frac{\Gamma(1-c)}{\Gamma(a+1-c)}M(a,c,\xi)+\frac{\Gamma(c-1)}{\Gamma(a)}\xi^{1-c}M(a+1-c,2-c,\xi)
\]
where
\[
M(a,c,\xi)=\sum_{n=0}^\infty \frac{a^{(n)}}{c^{(n)}n!}\xi^n, \qquad  \text{with} \qquad a^{(n)}=a(a+1)\cdots (a+n-1),
\]
and the branch of the root $\xi^{1-c}$ is fixed by requiring that $\arg(\xi)\in (0,2\pi)$ (with a cut on the positive real axis). Moreover $a,c$ are complex numbers and we assume that $b$ is not an integer (integer cases can be dealt with by a limiting procedure). See e.g. \cite[Appendix]{IK} for a review of the basic theory of these functions. We mention here nevertheless that $M$ is an entire function of $\xi$ so the only singularity of $\psi$ is the branch cut along the positive real axis.

We then define for $\xi\in I$, 
\begin{align}\label{eq:Psi}
\Psi(\xi)=\begin{pmatrix}
\xi^{\tfrac{\zeta_j}{{\sqrt{2}}}}\psi(\tfrac{\zeta_j}{{\sqrt{2}}},1+{\sqrt{2}}\zeta_j,\xi)e^{i\pi \tfrac{\zeta_j}{{\sqrt{2}}}}e^{-\xi/2} & -\xi^{\tfrac{\zeta_j}{{\sqrt{2}}}}\psi(1+\tfrac{\zeta_j}{{\sqrt{2}}},1+{\sqrt{2}}\zeta_j,-\xi) e^{i\pi \tfrac{\zeta_j}{{\sqrt{2}}}}e^{\xi/2}\tfrac{\zeta_j}{{\sqrt{2}}}\\
-\xi^{-\tfrac{\zeta_j}{{\sqrt{2}}}}\psi(1-\tfrac{\zeta_j}{{\sqrt{2}}},1-{\sqrt{2}}\zeta_j,\xi)e^{-i\pi3 \tfrac{\zeta_j}{{\sqrt{2}}}}e^{-\xi/2}\tfrac{\zeta_j}{{\sqrt{2}}} & \xi^{-\tfrac{\zeta_j}{{\sqrt{2}}}}\psi(-\tfrac{\zeta_j}{{\sqrt{2}}},1-{\sqrt{2}}\zeta_j,-\xi) e^{-i\pi \tfrac{\zeta_j}{{\sqrt{2}}}}e^{\xi/2}
\end{pmatrix}.
\end{align}
In other regions, $\Psi$ is constructed to have prescribed jump conditions. For example, in region $II$, one defines
\begin{align*}
\Psi(\xi)&=\begin{pmatrix}
\xi^{\tfrac{\zeta_j}{{\sqrt{2}}}}\psi(\tfrac{\zeta_j}{{\sqrt{2}}},1+{\sqrt{2}}\zeta_j,\xi)e^{i\pi \tfrac{\zeta_j}{{\sqrt{2}}}}e^{-\xi/2} & -\xi^{\tfrac{\zeta_j}{{\sqrt{2}}}}\psi(1+\tfrac{\zeta_j}{{\sqrt{2}}},1+{\sqrt{2}}\zeta_j,-\xi) e^{i\pi \tfrac{\zeta_j}{{\sqrt{2}}}}e^{\xi/2}\tfrac{\zeta_j}{{\sqrt{2}}}\\
-\xi^{-\tfrac{\zeta_j}{{\sqrt{2}}}}\psi(1-\tfrac{\zeta_j}{{\sqrt{2}}},1-{\sqrt{2}}\zeta_j,\xi)e^{-i\pi3 \tfrac{\zeta_j}{{\sqrt{2}}}}e^{-\xi/2}\tfrac{\zeta_j}{{\sqrt{2}}} & \xi^{-\tfrac{\zeta_j}{{\sqrt{2}}}}\psi(-\tfrac{\zeta_j}{{\sqrt{2}}},1-{\sqrt{2}}\zeta_j,-\xi) e^{-i\pi \tfrac{\zeta_j}{{\sqrt{2}}}}e^{\xi/2}
\end{pmatrix}\\
&\quad \times \begin{pmatrix} 1 & 0\\
e^{-i\pi {\sqrt{2}}\zeta_j} & 1\end{pmatrix}, 
\end{align*}
so on $\Gamma_{2}$, we have the jump condition, 
\[
\Psi_+(\xi)=\Psi_-(\xi)\begin{pmatrix} 1 & 0\\
e^{-i\pi {\sqrt{2}}\zeta_j} & 1\end{pmatrix}.
\]
For further details about the definition of $\Psi(\xi)$ for all $\xi\in \C$ and the Riemann-Hilbert problem satisfied by $\xi$, we refer to \cite[Section 4.2]{DIK}.

We are now in a position to define the local parametrix. Once again, we refer the reader to \cite[Section 4.2]{DIK}, where it is proven that for $z\in U_{\Delta,j}$,  the function 
\begin{equation}\label{eq:local}
P(z)=E(z)\Psi(\xi_{N,j}(z))F_{N,j}(z)^{-\sigma_3}\times \begin{cases}
z^{N\sigma_3/2}, & |z|<1\\
z^{-N\sigma_3/2}, & |z|>1
\end{cases},
\end{equation}
where $z \in U_{\Delta,j}\mapsto E(z)$ is a certain analytic function we will return to shortly. {As a result of our construction, $P$}  has the same jumps as $S$ in $U_{\Delta,j}$, and in fact, $z\mapsto S(z)P(z)^{-1}$ is analytic in $U_{\Delta,j}$ for all $t\in [0,1]$.

The final ingredient, the function $E$ is relevant for the matching condition -- as mentioned, this is the main step where we cannot rely directly on \cite{DIK}, since we allow $\zeta_j$ to grow with $N$. Nevertheless, it varies slowly enough that we have essentially the same asymptotics as in the fixed $\zeta_j$-case. In particular, we will use the same $E$-function as \cite{DIK}. We define 
\begin{align*}
E(z)&=\mathcal N(z)F_{N,j}(z)^{\sigma_3}e^{-iN\frac{x_j}{2}\sigma_3}\begin{pmatrix}
e^{-i\pi \zeta_j/{\sqrt{2}}} & 0\\
0 & e^{i\pi{\sqrt{2}} \zeta_j}
\end{pmatrix}, & \xi_{N,j}(z)\in I,II\\
E(z)&=\mathcal N(z)F_{N,j}(z)^{\sigma_3}e^{-iN\frac{x_j}{2}\sigma_3}\begin{pmatrix}
e^{-i\pi {\sqrt{2}} \zeta_j} & 0\\
0 & e^{i\pi 3\zeta_j/{\sqrt{2}}}
\end{pmatrix}, & \xi_{N,j}(z)\in III,IV\\
E(z)&=\mathcal N(z)F_{N,j}(z)^{\sigma_3}e^{iN\frac{x_j}{2}\sigma_3}\begin{pmatrix}
0 & e^{i\pi 3\zeta_j/{\sqrt{2}}} \\
-e^{-i\pi {\sqrt{2}}\zeta_j} & 0
\end{pmatrix}, & \xi_{N,j}(z)\in V,VI\\
E(z)&=\mathcal N(z)F_{N,j}(z)^{\sigma_3}e^{iN\frac{x_j}{2}\sigma_3}\begin{pmatrix}
0 & e^{i\pi \zeta_j/{\sqrt{2}}} \\
-e^{-i\pi \zeta_j/{\sqrt{2}}} & 0
\end{pmatrix}, & \xi_{N,j}(z)\in V,VI.
\end{align*}

We are now concerned with the matching condition, namely we wish to understand asymptotics of
$P(z)\mathcal N(z)^{-1}$
for $z\in \partial U_{\Delta,j}$, where $\mathcal N$ is as in \eqref{eq:global}. Note that for $z\in \partial U_{\Delta,j}$,
\begin{equation} \label{xiest}
|\xi_{N,j}(z)|\simeq N|z-e^{ix_j}|\ge c N \Delta^{-1}  \ge c N^{\eta}
\end{equation}
for some numerical constant $c$, so we will need to understand large $\xi$ asymptotics of $\Psi(\xi)$. For simplicity, we will only do this in sector $I$ and leave the remaining sectors to the reader (though one must be slightly careful due to the branch cut so the asymptotics are slightly more complicated in regions $VI$ and $VII$). For a reference on the relevant asymptotics, we refer the reader to \cite[Chapter 13.7]{NIST} and \cite[Section 6 -- Section 9]{Olver} (though note that in the latter reference, the results are expressed in terms of Whittaker functions which are readily expressed in terms of $\psi$ -- we leave the details of this to the reader). The upshot is that for $\xi\in I$ and for $|c-2a|=1$ (bounded would work just as well, but for us, $|c-2a|=1$)
\[
\psi(a,c,\xi)=\xi^{-a}\left(1-a(1+a-c)\frac{1}{\xi}\right)+\varepsilon(\xi),
\]
where for say $|\xi|\geq 1$,
\[
|\varepsilon(\xi)|\leq C_1\frac{|a|^4}{|\xi^{a+2}|}\exp\left(C_2\frac{|a|^2}{|\xi|}\right) \qquad \text{and} \qquad |\varepsilon'(\xi)|\leq C_1\frac{|a|^4}{|\xi^{a+2}|}\exp\left(C_2\frac{|a|^2}{|\xi|}\right)
\]
for some universal constants $C_1,C_2>0$.

We thus find that in region $\partial U_{\Delta,j} \cap I$, if we assume that $|\zeta_j|^2=\mathcal O(|\xi_{N,j}(z)|)$ (this is the case for $\zeta_j \in \mathcal{D}_{R_N}$ if $R_N = o(N^\alpha)$ for any $\alpha>0$), then
\begin{align*}
P(z)\mathcal N(z)^{-1}&=\mathcal N(z)F_{N,j}(z)^{\sigma_3} e^{-iN \frac{x_j}{2}\sigma_3}\begin{pmatrix}
e^{-i\pi \zeta_j/{\sqrt{2}}} & 0\\
0 & e^{i\pi {\sqrt{2}}\zeta_j}
\end{pmatrix}\\
&\quad \times \begin{pmatrix}
1+\frac{\zeta_j^2}{{2}\xi_{N,j}(z)}+ O(\tfrac{|\zeta_j|^4}{\xi_{N,j}(z)^2}) & (\frac{\zeta_j}{{\sqrt{2}}\xi_{N,j}(z)}+ O(\tfrac{|\zeta_j|^4}{\xi_{N,j}(z)^2}))e^{2{\sqrt{2}}\pi i \zeta_j}\\
(-\frac{\zeta_j}{{\sqrt{2}}\xi_{N,j}(z)}+ O(\tfrac{|\zeta_j|^4}{\xi_{N,j}(z)^2}))e^{-2{\sqrt{2}}\pi i \zeta_j} & 1-\frac{\zeta_j^2}{{2}\xi_{N,j}(z)}+O(\tfrac{|\zeta_j|^4}{\xi_{N,j}(z)^2})\end{pmatrix}
\\
&\qquad \times e^{-\xi_{N,j}(z)\sigma_3/2}\begin{pmatrix}
e^{i\pi \zeta_j/{\sqrt{2}}} & 0\\
0 & e^{-i\pi {\sqrt{2}}\zeta_j}.
\end{pmatrix}F_{N,j}(z)^{-\sigma_3}z^{N\sigma_3/2}\mathcal N(z)^{-1}\\
&=M(z)\begin{pmatrix}
1+\frac{\zeta_j^2}{{2}\xi_{N,j}(z)}+ O(\tfrac{|\zeta_j|^4}{\xi_{N,j}(z)^2}) & (\frac{\zeta_j}{{\sqrt{2}}\xi_{N,j}(z)}+ O(\tfrac{|\zeta_j|^4}{\xi_{N,j}(z)^2}))e^{2{\sqrt{2}}\pi i \zeta_j}\\
(-\frac{\zeta_j}{{\sqrt{2}}\xi_{N,j}(z)}+ O(\tfrac{|\zeta_j|^4}{\xi_{N,j}(z)^2}))e^{-2{\sqrt{2}}\pi i \zeta_j} & 1-\frac{\zeta_j^2}{{2}\xi_{N,j}(z)}+O(\tfrac{|\zeta_j|^4}{\xi_{N,j}(z)^2})\end{pmatrix} M(z)^{-1},
\end{align*}
where the implied constants are uniform in everything and 
\[
M(z)=\mathcal N(z)F_{N,j}(z)^{\sigma_3}e^{-iN \frac{x_j}{2}\sigma_3}\begin{pmatrix}
e^{-i\pi \zeta_j/{\sqrt{2}}} & 0\\
0  & e^{i\pi {\sqrt{2}}\zeta_j}
\end{pmatrix}.
\]
Recalling that 
\(
\mathcal N(z)=\mathcal D_{\rm in}(z)^{\sigma_3}\begin{pmatrix}0 & 1\\
-1 & 0\end{pmatrix},
\)
for $\xi_{N,j}(z)\in I$, 
we see that 
\begin{align*}
M(z)&=\begin{pmatrix}
0 & \mathcal D_{\rm in}(z)F_{N,j}(z)^{-1}e^{iN\frac{x_j}{2}}e^{i\pi {\sqrt{2}}\zeta_j}\\
-\mathcal D_{\rm in}(z)^{-1}F_{N,j}(z)e^{-iN\frac{x_j}{2}}e^{-i\pi \zeta_j/{\sqrt{2}}} & 0
\end{pmatrix}
\end{align*}
and 
\begin{align*}
&P(z)\mathcal N(z)^{-1}\\
&=\begin{pmatrix}
1-\frac{\zeta_j^2}{{2}\xi_{N,j}(z)}+ O(\tfrac{|\zeta_j|^4}{\xi_{N,j}(z)^2}) & -\frac{\mathcal D_{\rm in}(z)^2}{ F_{N,j}(z)^{2}}e^{iNx_j}e^{-\frac{\pi}{{\sqrt{2}}}i \zeta_j}\left(-\frac{\zeta_j}{{\sqrt{2}}\xi_{N,j}(z)}+ O(\tfrac{|\zeta_j|^4}{\xi_{N,j}(z)^2})\right)\\
-\frac{F_{N,j}(z)^{2}}{\mathcal D_{\rm in}(z)^{2}}e^{-iNx_j}e^{\frac{\pi}{{\sqrt{2}}}i \zeta_j}\left(\frac{\zeta_j}{{\sqrt{2}}\xi_{N,j}(z)}+ O(\tfrac{|\zeta_j|^4}{\xi_{N,j}(z)^2})\right)& 1+\frac{\zeta_j^2}{{2}\xi_{N,j}(z)}+O(\tfrac{|\zeta_j|^4}{\xi_{N,j}(z)^2})\end{pmatrix}.
\end{align*}

Recalling \eqref{eq:f} and \eqref{eq:Fsq}, this can be written as, for $z\in \partial U_{j,\Delta}\cap \xi_{N,j}^{-1}(I)$,
\begin{align}
&P(z)\mathcal N(z)^{-1}\label{eq:match1}\\
&=\begin{pmatrix}
1-\frac{\zeta_j^2}{{2}\xi_{N,j}(z)}+ O(\tfrac{|\zeta_j|^4}{\xi_{N,j}(z)^2}) & -\mathcal D_{\rm in}(z)\mathcal D_{\rm out}(z)e^{iNx_j}e^{\frac{\pi}{{\sqrt{2}}}i \zeta_j}\left(-\frac{\zeta_j}{{\sqrt{2}}\xi_{N,j}(z)}+ O(\tfrac{|\zeta_j|^4}{\xi_{N,j}(z)^2})\right)\\
-\frac{e^{-iNx_j}e^{-\frac{\pi}{{\sqrt{2}}}i \zeta_j}}{\mathcal D_{\rm in}(z)\mathcal D_{\rm out}(z)}\left(\frac{\zeta_j}{{\sqrt{2}}\xi_{N,j}(z)}+ O(\tfrac{|\zeta_j|^4}{\xi_{N,j}(z)^2})\right)& 1+\frac{\zeta_j^2}{{2}\xi_{N,j}(z)}+O(\tfrac{|\zeta_j|^4}{\xi_{N,j}(z)^2})\end{pmatrix},\notag
\end{align}
where the implied constants are universal in the regime $|\zeta_j|^2/|\xi_{N,j}(z)|\leq 1$ 
(we recall that for $z\in \partial U_{j,\Delta}$, $|\xi_{N,j}(z)| \ge cN/\Delta$ for a $c>0$).

We see that the key thing is to estimate $\mathcal D_{\mathrm{in}}(z)\mathcal D_{\mathrm{out}}(z)e^{\frac{\pi}{{\sqrt{2}}}i\zeta_j}$ for $z$ in the appropriate domain. 
In particular, this is where the condition $\zeta_1,\zeta_2=o(\log N)$  comes in play.
We summarize the required fact in the following lemma.

\begin{lemma}\label{le:Dest}
For $z\in \partial U_{j,\Delta}\cap \xi_{N,j}^{-1}(I)$ and $\zeta_1,\zeta_2 \in \mathcal{D}_{R_N}$, \eqref{eq:Dr}, with $R_N = o(\log N)$ as $N\to\infty$, 
\[
|\mathcal D_{\mathrm{in}}(z)\mathcal D_{\mathrm{out}}(z)e^{\frac{\pi}{{\sqrt{2}}}i\zeta_j}|^{\pm 1}= \exp O(R_N)
\]
where the implied constant only depends on $C$ in \eqref{pot}.
\end{lemma}
\begin{proof}
Note that by definition (namely \eqref{eq:D}), for $t\in[0,1]$, 
\[
\mathcal D_{\mathrm{in}}(z)\mathcal D_{\mathrm{out}}(z)e^{\frac{\pi}{{\sqrt{2}}}i\zeta_j}=e^{t\sum_{k=1}^\Delta(\widehat V_k z^k-\widehat V_{-k}z^{-k})} \frac{z^{\zeta_1/\sqrt{2}} z^{\zeta_2/\sqrt{2}}e^{\frac{\pi}{\sqrt{2}}i\zeta_j}}{e^{i\zeta_1 x_1/\sqrt{2}} e^{i \pi \zeta_1 /\sqrt{2}}e^{i\zeta_2 x_2/\sqrt{2}} e^{i \pi \zeta_2 /\sqrt{2}}} .
\]
If $z\in \partial U_{j,\Delta}$, \eqref{eq:Delta}, and $\zeta_1,\zeta_2 \in \mathcal{D}_{R_N}$, we have for some numerical constant
\[
\bigg| \frac{z^{\zeta_1/\sqrt{2}} z^{\zeta_2/\sqrt{2}}e^{\frac{\pi}{\sqrt{2}}i\zeta_j}}{e^{i\zeta_1 x_1/\sqrt{2}} e^{i \pi \zeta_1 /\sqrt{2}}e^{i\zeta_2 x_2/\sqrt{2}} e^{i \pi \zeta_2 /\sqrt{2}}}\bigg|^{\pm1}
=  \exp O(R_N).  
\]

Thus it remains to control the $V$-term. For this, we note that if we write $\widehat z=\frac{z}{|z|}$, then 
\[
\sum_{k=1}^\Delta(\widehat V_k z^k-\widehat V_{-k}z^{-k})=\sum_{k=1}^\Delta [\widehat V_k (z^k-\widehat z^k)-\widehat V_{-k}(z^{-k} -\widehat z^{-k})]+\sum_{k=1}^\Delta (\widehat V_k \widehat z^k-\widehat V_{-k}\widehat z^{-k}).
\]
Since $V$ is real-valued $\widehat V_{-k}=\overline{\widehat V_k}$, and since $\overline{\widehat z}=\widehat{z}^{-1}$, we see that
\[
\sum_{k=1}^\Delta (\widehat V_k \widehat z^k-\widehat V_{-k}\widehat z^{-k})\in i\R.
\]
Thus its contribution to the exponential has size 1. 
We turn to bound for $z\in  U_{\Delta,j}$, 
\[
\sum_{k=1}^\Delta \widehat V_k (z^k-\widehat{z}^k) =  \sum_{k=1}^\Delta\widehat V_k  \widehat{z}^k \big(|z|^k -1\big).
\]
The same argument will work for the sum with $k$ replaced by $-k$.
Then, by \eqref{pot}, using that $|z|^k \le e^{k/\Delta}$ for $z\in \partial U_{\Delta,j}$,
we have
\begin{align*}
\left|\sum_{k=1}^\Delta \widehat V_k (z^k-\widehat{z}^k)\right|
\leq C\sum_{k=1}^\Delta \frac{1}{k}\left( e^{k/\Delta}-1\right) =O(1).
\end{align*}
where the implied constant only depends on $C$ in \eqref{pot}. 
This concludes the proof.
\end{proof}

In particular, by \eqref{eq:match1} and Lemma~\ref{le:Dest}, we conclude that there is a fixed $\alpha>0$ ($\alpha=\eta$ if $R_N=R$ is fixed and $\alpha<1$ if $R_N = o(\log N)$ and $\Delta$ is fixed; see \eqref{eq:Delta})
such that uniformly for $z\in \cup_{j=1}^2 \partial U_{\Delta,j}$, 
\begin{equation}\label{eq:match2}
P(z)\mathcal N(z)^{-1}=I+O(N^{-\alpha})
\end{equation}
uniformly in all the relevant parameters ($t\in [0,1]$, $x_1,x_2\in\T$ with  $|e^{ix_1}-e^{ix_2}|\ge \Delta^{-1}$,  $\zeta_1,\zeta_2 \in  \mathcal{D}_{R_N}$ with $R_N =o(\log N)$ and $\Delta \le  C N^{\eta-1}$).

While we will need the more precise matching condition \eqref{eq:match1} for parts of our argument (in Section~\ref{sec:Rmacro}), this is already sufficient for us to discuss the ``small norm analysis".

\subsection{The small norm problem} \label{sec:R}

We now briefly review some basic facts about the analysis of ``small norm" Riemann-Hilbert problems. For details (which we omit), we refer the reader to e.g. \cite[Section 7.2]{DKMLVZ} and \cite[Theorem 3.1 and Section 9]{Kuijlaars}. There are several key underlying ideas, and we will not go into detail about them. 

\begin{figure}
\begin{center}
\begin{tikzpicture}[xscale=0.02,yscale=0.02]
\clip (-135,-135) rectangle (170,135); 

\fill [color=black] (0,0) circle (3);

\draw [black,thick,domain=45:315] plot ({80*cos(\x)}, {80*sin(\x)});
\draw [black,thick,domain=42:318] plot ({120*cos(\x)}, {120*sin(\x)});

\draw [black,thick,domain=-15:15] plot ({80*cos(\x)}, {80*sin(\x)});
\draw [black,thick,domain=-19:19] plot ({120*cos(\x)}, {120*sin(\x)});

\fill [color=black] (86,50) circle (3);
\fill [color=black] (86,-50) circle (3);

\node at (-10,0) {\small $0$};
\node at (90,65) {\small $e^{i\theta_2}$};
\node at (90,-65) {\small $e^{i\theta_1}$};

\draw [black,thick,domain=0:360] plot ({86+30*cos(\x)}, {50+ 30*sin(\x)});
\draw [black,thick,domain=0:360] plot ({86+30*cos(\x)}, {-50+ 30*sin(\x)});

\draw [-{Stealth[length=3mm]}] (120,0) -- (120,0.1);
\draw [-{Stealth[length=3mm]}] (80,0) -- (80,0.1);
\draw [-{Stealth[length=3mm]}] (-120,0) -- (-120,-0.1);
\draw [-{Stealth[length=3mm]}] (-80,0) -- (-80,-0.1);

\draw [-{Stealth[length=3mm]}] (110,68) -- (110.1,67.9);
\draw [-{Stealth[length=3mm]}] (115,-62.9) -- (115,-63);

\end{tikzpicture}
\end{center}
\vspace{-0.3cm}
\caption{An illustration of the jump contour $\Gamma_\mathcal R$.}
\label{fig:L3}
\end{figure}
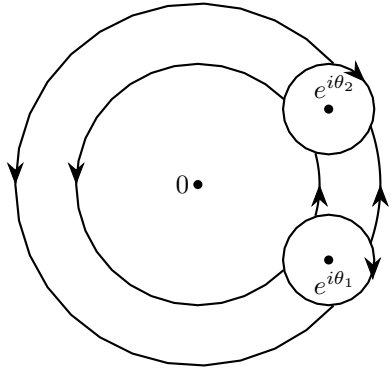

Our starting point is to define 
\begin{equation} \label{eq:R}
\mathcal R(z)=\begin{cases}
S(z)\mathcal N(z)^{-1}, & z \in \C  \setminus  \cup_{j=1}^2 U_{\Delta,j}\\
S(z)P(z)^{-1}, & z\in  \operatorname{int}\big(\cup_{j=1}^2 U_{\Delta,j}\big)
\end{cases} 
\end{equation}
and the contour  
$\Gamma_\mathcal R=\partial \widetilde L \, \cup_{j=1}^2 \partial U_{\Delta,j}$ (see Figure~\ref{fig:L3}),
where both circles are oriented in a clockwise manner. 
By construction, $S$ is a solution of Problem~\ref{pb:S} if and only if $\mathcal R$ satisfies the following Riemann-Hilbert problem; 

\begin{problem}\label{pb:R}
 let  $\Gamma_\mathcal R^\circ$ denote $\Gamma_\mathcal R$ without the self-intersection points. 
Let $\mathcal R = \mathcal R_{n}$ for $n\in\N$ solve;
\begin{enumerate}
\item  $\mathcal R:\C\setminus \Gamma_\mathcal R\to \C^{2\times 2}$ is analytic.
\item $\mathcal R$ has continuous boundary values on $\Gamma_\mathcal R^\circ$, denoted by $\mathcal R_+,\mathcal R_-$  which satisfy
\begin{align*}
\mathcal R_+(z) &= \mathcal R_-(z)
\begin{pmatrix}
1 & 0\\
z^n f(z)^{-1} & 1
\end{pmatrix}, & z\in \partial \widetilde L \cap \{|z|<1\}\setminus \cup_{j=1}^2  \{|z-e^{ix_1}|=\tfrac{1}{2\Delta}\} \\ 
\mathcal R_+(z) &= \mathcal R_-(z) 
\begin{pmatrix}
1 & 0\\
z^{-n} f(z)^{-1} & 1
\end{pmatrix}  , & z\in \partial \widetilde L \cap \{|z|>1\}\setminus \cup_{j=1}^2  \{|z-e^{ix_1}|=\tfrac{1}{2\Delta}\} \\ 
\mathcal R_+(z) &= \mathcal R_-(z) P(z)\mathcal N(z)^{-1} , 
& z\in \cup_{j=1}^2  \{|z-e^{ix_1}|=\tfrac{1}{2\Delta}\}
\end{align*} 
\item $\mathcal R(z)=I+O(z^{-1})$ as $z\to \infty$.
\item  $\mathcal R(z)=O(1)$ as $z\to \Gamma_\mathcal R$.
\end{enumerate}
\end{problem} 

We now argue (mainly by referring to the literature) that this problem can be solved (uniquely) by a Neumann series if $N$ is large enough. 
Importantly, this will establish that Problem~\ref{pb:S} also has a solution for large enough $N$. Thus, this yields that $Y$ exists (without assuming a priori that Problem~\ref{pb:Y} admits a solution) and it provides the asymptotics of this solution as $N\to\infty$ (with the required uniformity). 
Moreover, we recall that this solution is unique and given explicitly by \eqref{eq:Ydef},  although we will not need this fact. 

\medskip 

Since the uniformity is different in cases of Assumptions~\ref{ass:meso} and~\ref{ass:macro}, we have to treat these situations separately. Let us first consider the case of Assumption \ref{ass:meso}. 

\begin{lemma}\label{le:mesoR}
Fix a (small) $\eta>0$ and assume that $\widehat{V}$ satisfies \eqref{pot}. Let also $\alpha>0$ be as in \eqref{eq:match2}.
The (unique) solution $\mathcal R = \mathcal R_{N}$ of Problem~\ref{pb:R} satisfies uniformly for $x_1,x_2\in\T$ with  $|e^{ix_1}-e^{ix_2}|\ge \Delta^{-1}$, 
$\zeta_1,\zeta_2 \in \mathcal{D}_{R_N}$, $t\in [0,1]$ and locally uniformly for  $z\in \big\{z\in \C: \dist(z,\Gamma_\mathcal R)\geq \frac{1}{\Delta}\big\}$, 
\begin{equation*} 
\mathcal R(z)=I+O\big(N^{-\alpha}\big)
\end{equation*}
and 
\begin{equation} \label{bdR2}
\mathcal R'(z)=O\big(N^{-\alpha} \max_{j=1,2}|z-e^{ix_j}|^{-1}\big).
\end{equation}
In particular, there exists a $N_{R,\alpha}\in\N$ such that for $N\geq N_{R,\alpha}$, the solution $Y= Y_N$ of Problem~\ref{pb:Y} exists.
\end{lemma}

\begin{proof}
This is very standard in the Riemann-Hilbert literature, so we will simply refer the reader to the relevant references on various points. As a general reference, see e.g. \cite[Section 7.2]{DKMLVZ}, and for something closer to our setting, see \cite[Section 9]{NSW}.

First of all, the jump condition for $\mathcal R$ can be written as 
\begin{equation}\label{eq:Rjumpdif}
\mathcal R_+(z)-\mathcal R_-(z) =\mathcal R_-(z)(J_\mathcal R(z)-I), \qquad z\in \Gamma_R^\circ , 
\end{equation}
where the jump matrix satisfies, by \eqref{eq:match2}, $J_\mathcal R(z)-I=O(N^{-\alpha})$ for $z\in \cup_{j=1}^2 \partial U_{\Delta,j}$. 
On the remainder of $\Gamma_R^\circ$, by \eqref{eq:L}, $|z|=1\pm \frac{1}{4\Delta}$ so that $|z^{\mp N}|=O(e^{-c N\Delta^{-1}})$ for some universal constants $c>0$.  
Hence, by \eqref{Dest}, one readily checks that $J_{\mathcal R}(z)-I=O(e^{-cN^\eta})$ for a fixed constant $c>0$ on $\partial\widetilde L$.
In particular, this quantity is negligible.

Using the conditions from Problem~\ref{pb:R}, one can use the Sokhtoski-Plemelj identity to solve \eqref{eq:Rjumpdif} in terms of $\mathcal R_-$; for $z\in \C \setminus \Gamma_\mathcal R$,
\begin{equation} \label{Rsol}
\mathcal R(z)=I+\frac{1}{2\pi i}\int_{\Gamma_\mathcal R}\mathcal R_-(w)(J_\mathcal R(w)-I)\frac{dw}{w-z}.
\end{equation}
Taking the boundary values of this equation from the $-$-side, this implies that $\mathcal R_-$ satisfies the singular integral equation
\[
\mathcal R_-=I+ \mathcal C_{-}(\mathcal R_-(J_\mathcal R-I))=:I+ \mathcal C_\Delta \mathcal R_-,
\]
where $\mathcal C_-$ denotes the boundary values from the $-$-side of the Cauchy integral operator associated with $\Gamma_\mathcal R$. We rewrite this as 
\[
(I- \mathcal C_\Delta)(\mathcal R_--I)= \mathcal C_\Delta I.
\]
The main point is that the operator $I- \mathcal C_\Delta$ on $L^2(\Gamma_\mathcal R,\C^{2\times 2})$ is invertible.
This follows from the relationship between $\mathcal C_\Delta$ and the weighted Hilbert transform on $\Gamma_\mathcal R$, so one can control the norm of the operator $\mathcal C_\Delta$ in order to invert $I- \mathcal C_\Delta$ via Neumann-series techniques. To be more specific, $\mathcal C_-$ is a bounded operator on $L^2(\Gamma_\mathcal R,\C^{2\times 2})$ whose norm is uniformly bounded in $\Delta$ (which depends on $N$ and control the other relevant parameters) -- this relies on a celebrated result of David combined with a simple argument allowing for moving contours (see \cite[Lemma 9.2]{NSW} for more details). Since we have established that $J_\mathcal R-I=O(N^{-\alpha})$ on  $\Gamma_\mathcal R^\circ$, this implies that the operator norm of $\| \mathcal C_\Delta\| = O(N^{-\alpha})$ too.
Hence, $I-\mathcal C_\Delta$ is invertible for large enough $N$.

This allows writing on $ \Gamma_\mathcal R$, 
\[
\mathcal R_-=I+((I-\mathcal C_\Delta)^{-1} \mathcal C_\Delta I) ,
\]
which is uniformly bounded by a numerical constant (for $N$ large enough)
and then to bound $\mathcal R$ using formula \eqref{Rsol}; we conclude that for $z\in \C \setminus \Gamma_\mathcal R$,
\begin{equation} \label{Rbound}
\big\| \mathcal R(z) - I \big\|
\le \int_{\Gamma_{\mathcal R}}\frac{1}{|w-z|}\|J_\mathcal R(w)-I\||dw|,
\end{equation}
where $\|\cdot \|$ is any suitable matrix norm, and $|dw|$ means integration with respect to arc-length measure.
Then, we have seen that $\|J_\mathcal R(w)-I\|= O(N^{-\alpha})$ for $w\in\partial U_{j,\Delta}$, 
so that if $\dist(z,\Gamma_\mathcal R)\geq \frac{1}{2\Delta}$, 
\[
\int_{\partial U_{j,\Delta}}\frac{1}{|w-z|}\|J_\mathcal R(w)-I\||dw|
=O\big(N^{-\alpha}\big) . 
\]
On the remainder of $\Gamma_\mathcal R$, we have seen that $\|J_\mathcal R-I\|=O(e^{-cN^\eta})$, so a similar bound shows that the integral on $\partial\widetilde L$ is negligible as $N\to\infty$. This yields the first estimate. 

Then, to prove the second, one can use that $\mathcal R$ is analytic on $\C\setminus \Gamma_\mathcal R$, by Cauchy integral formula, 
\[
\mathcal R'(z) =\int_{\gamma_z}\frac{\mathcal R(w)-I}{(w-z)^2}d w,
\]
where $\gamma_z = \{|w-z| =  \frac{1}{2\Delta} \}$ is a (simple) loop  around $z$ (not intersecting $\Gamma_\mathcal R$ since we assume that $\dist(z,\Gamma_\mathcal R)\geq \frac{1}{\Delta}$). Running a similar argument again (the main contributions coming from $\partial U_{j,\Delta}$), using the previous estimate for $\|\mathcal R-I\|$ instead, this yields uniformly for $\dist(z,\Gamma_\mathcal R)\geq \frac{1}{\Delta}$, 
\[
\mathcal R'(z) =O(\Delta \|\mathcal R-I\|) = O\big(N^{-\alpha} \max_{j=1,2}|z-e^{ix_j}|^{-1}\big) , 
\]
also uniformly in the other relevant parameters.

\medskip

This argument does not only provide the relevant asymptotics, it also shows that Problem~\ref{pb:R} has a unique solution $\mathcal R = \mathcal R_{N}$ if $N\geq N_{R,\alpha}$. 
Hence, by inverting the transformations \eqref{eq:R} and \eqref{eq:S}, one constructs a solution $Y=Y_N$ to the Riemann-Hilbert problem~\ref{pb:Y}.
\end{proof}

In the setting of Assumption \ref{ass:macro}, the corresponding statement is as follows.
Recall that $\mathcal{A}_c\subset \big\{(x,y)\in \T^2: |e^{ix}-e^{iy}| \ge c  \big\}$ for a fixed (small) $c>0$. 

\begin{lemma}\label{le:macroR}
Let $\Delta\ge c^{-1}$ be fixed and assume that $R_N=o(\log N)$ as $N\to\infty$.
For any $\epsilon>0$, the (unique) solution $\mathcal R = \mathcal R_{N}$ of the Problem~\ref{pb:R} satisfies uniformly for $\zeta_1,\zeta_2\in \mathcal D_{R_N}$, $x_1,x_2\in \mathcal{A}_c$, $t\in [0,1]$, and  $z\in\C\setminus \Gamma_R$,
\[
\mathcal R(z)=I+\frac{1}{N}\sum_{j=1}^2\bigg(\frac{A_j}{z-e^{ix_j}}-\Theta_j(z) \1\{|z-e^{ix_j}|<\Delta^{-1}\}\bigg) +O(N^{\epsilon-2})
\]
and 
\[
\mathcal R'(z)
=-\frac{1}{N}\sum_{j=1}^2 \bigg(\frac{A_j}{(z-e^{ix_j})^2}-\Theta_j'(z)\1\{|z-e^{ix_j}|<\Delta^{-1}\} \bigg) +O(N^{\epsilon-2}) 
\]
Here, $\Theta_j$ is a meromorphic function in $U_{j,\Delta}$ with a simple pole at $e^{ix_j}$ (it is explicit and depends on $N$ -- see the proof) and $A_j$ is another explicit quantity (see e.g. \cite[(4.71)]{DIK}).

In particular, there exists a $N_{c}\in\N$ such that for $N\geq N_{c}$, the solution $Y= Y_N$ of Problem~\ref{pb:Y} exists.
\end{lemma}

\begin{proof}
The proof follows along the same lines as that of Lemma \ref{le:mesoR}, so we also give few details. 
The key point is again estimating the jump matrix $J_\mathcal R-I$. 

By \eqref{eq:match1}, using that $|\zeta_j| \le R_N$ and \eqref{xiest} with $
\eta=1$ (here, we assume that $\Delta$ is fixed), for any $\epsilon>0$, it holds for $z\in \partial U_{j,\Delta}$ with say $\xi_{N,j}(z)\in I$,
\begin{align*}
J_{\mathcal R}(z)-I&= \frac 1N \begin{pmatrix}
-\frac{\zeta_j^2}{2N\xi_{N,j}(z)} & \mathcal D_{\rm in}(z)\mathcal D_{\rm out}(z)e^{iNx_j} e^{\frac{\pi}{\sqrt{2}}i\zeta_j}\frac{\zeta_j}{\sqrt{2}N\xi_{N,j}(z)}\\
-\frac{e^{-iNx_j} e^{-\frac{\pi}{\sqrt{2}}i\zeta_j}}{\mathcal D_{\rm in}(z)\mathcal D_{\rm out}(z)}\frac{\zeta_j}{\sqrt{2}N\xi_{N,j}(z)} & \frac{\zeta_j^2}{2N\xi_{N,j}(z)}
\end{pmatrix} +O(N^{\epsilon-2}) .
\end{align*}
Note that we used Lemma~\ref{le:Dest} and the assumption $R_N=o(\log N)$ to control the error term (uniformly in all the relevant parameters).
We denote the main term by $\Theta_j(z)$, it is basically of order 1 and according to \eqref{eq:D}, 
\begin{align*}
\mathcal D_{\rm in}(z)\mathcal D_{\rm out}(z)&=e^{t\sum_{j=1}^\Delta(\widehat{V}_{j}z^j-\widehat{V}_{-j}z^{-j})}(ze^{-ix_1})^{\zeta_1/\sqrt{2}}(ze^{-ix_2})^{\zeta_2/\sqrt{2}}e^{-i\pi \zeta_1/\sqrt{2}}e^{-i\pi\zeta_2/\sqrt{2}},
\end{align*}
so that this quantity is analytic in a neighborhood of $e^{ix_j}$.
Hence, if $\Delta$ is sufficiently small (see \eqref{eq:Delta}), $\Theta_j(z)$ is meromorphic in $U_{j,\Delta}$ with a simple pole at $z=e^{ix_j}$ coming from the local coordinate $\xi_{N,j}$. 

On the remainder of the jump contour, by \eqref{eq:f} and a similar argument, there is a constant  $c>0$ depending on $\Delta$ so that 
\[
\|J_{\mathcal R}(z)-I\|=O(|f(z)^{-1}||z|^{\pm N}) = O(e^{-c N})
\]
with the sign depending on whether $|z|<1$ or $|z|>1$.

Thus, the same argument used to control \eqref{Rbound} applies and we find that if  say $\dist(z,\Gamma_\mathcal R)>\frac{1}{2\Delta}$, 
\[
\mathcal R(z)=I+\sum_{j=1}^2\frac{1}{N}\int_{\partial U_{j,\Delta}}\frac{\Theta_j(w)}{w-z}\frac{dw}{2\pi i} + O(N^{\epsilon-2})
\]
with the required uniformity. 
To compute this integral, one needs to evaluate the residue of $\Theta_j$ at the simple pole $w=e^{ix_j}$. 
This is done in detail in \cite[Section 4.3]{DIK}, leading to the quantity $A_j$ and we omit further details. 
Note that in case  $z\in\mathrm{int}(U_{j,\Delta})$, there is an extra residue at $w=z$ which is given by  $\Theta_j(z)$.
We conclude by mentioning that the estimates can be extended near the contour $ \Gamma_R$ by the standard contour deformation argument (see \cite[Section 7]{DKMLVZ}).
The estimate for the derivative of $\mathcal R$ is obtained using the  Cauchy integral formula and the claim about $Y$ is obtained exactly as in the proof of Lemma~\ref{le:mesoR} with $\Delta$ fixed. This concludes the proof.  
\end{proof}

Now that we have good asymptotics for the function $\mathcal R$, we can proceed to integrate the differential identities of Lemmas~\ref{le:di2} and~\ref{le:di} to verify our assumptions.


{
\subsection{Verification of Assumption \ref{ass:meso} (2)} \label{sec:Rmeso}
The claim is already known if $\xi=0$; this is a combination of \cite[Theorem 1.11]{CK} (which holds if $|e^{ix_1}-e^{ix_2}|<\epsilon$ for some fixed $\epsilon>0$) and \cite[Theorem 1.1]{DIK2} (see \cite[Remark 1.4]{DIK2} for a comment on uniformity). Together, these two results exactly state that with $\Psi$ as in Lemma~\ref{le:morris}, 
\begin{equation} \label{2pt}
\Psi_{N}^{(\zeta_1,\zeta_2)}(x_1, x_2)  = \Psi(\zeta_1) \Psi(\zeta_2) \big(1+\underset{N\to\infty}{o(1)} \big)
\end{equation}
uniformly for $\zeta_1,\zeta_2 \in \mathcal{D}_R$, $x_1,x_2 \in \T$ with $|x_1-x_2| \ge N^{\eta-1}$. 
Note that by  \eqref{mom1} and  \eqref{HS}, the LHS equals
\(
D_{N}(F_0)  \exp\big(-\tfrac{\zeta_1^2+ \zeta_2^2}{2} \log N-\zeta_1\zeta_2 C_\X(x_1,x_2)\big)
\).
\smallskip

We now extend these asymptotics to $\xi \neq 0$ fixed by using \eqref{HS} and integrating the differential identity of Lemma~\ref{le:di2}. 
In this case, the potential $V$ is given by \eqref{def:V} where for $\delta\in(0,1]$,  $e^{i\theta} \mapsto C_{\X,\delta}(\theta,x)$ is a trigonometric polynomial as in Lemma~\ref{lem:covcue}.
In particular, the assumptions \eqref{pot}--\eqref{eq:Delta} hold if $\Delta= C N^{1-\eta}$ for a $C\ge 1$. 
There is a standard analyticity argument that we omit and refer instead to e.g.~\cite[Section 5.1 and Section 5.3]{DIK2}  (see in particular \cite[(5.23)]{DIK2})
which ensures that the condition $D_n(F_{t})\neq 0$ for all $n\in \N$ from  Lemma~\ref{le:di2} holds for all but finitely many $t\in [0,1]$. 
Then, with $Y=Y_{N,t}$ and $F=F_t$, we have
\[
\log \frac{D_{N}(F_1)}{D_{N}(F_0)}=\frac{1}{2\pi i}\int_0^1 \oint_{\U}z^{-N}(Y(z)^{-1} Y'(z))_{21}V(z)F(z)dzdt.
\]
A straightforward calculation using the jump conditions for $Y$ (Problem~\ref{pb:Y}) implies that for $z\in \U$, 
\[
z^{-N}(Y(z)^{-1}Y'(z))_{21}F(z)=(Y(z)^{-1}Y'(z))_{11,+}-(Y(z)^{-1}Y'(z))_{11,-}.
\]
Thus we can deform our integration contour; let $L_\pm$ be two circles (positively oriented) of radius $1\mp1/\Delta$ so that we can use Lemma \ref{le:mesoR}. 
In particular, we obtain
\begin{align*}
\log \frac{D_{N}(F_1)}{D_{N}(F_0)}=\frac{1}{2\pi i}\int_0^1 \bigg( \oint_{L_+}(Y(z)^{-1} Y'(z))_{11}\widehat{V}(z)dz -  \oint_{L_-}(Y(z)^{-1} Y'(z))_{11}\widehat{V}(z)dz \bigg)dt  
\end{align*}
where 
\[
Y=  \begin{cases}
\mathcal R \mathcal N &\text{on $L_+$} \\
\mathcal R \mathcal N z^{N\sigma_3}&\text{on $L_-$}  
\end{cases} , 
\]
$\mathcal N$ is given by \eqref{eq:global}, and  $\mathcal R$  is as in Lemma \ref{le:mesoR}. 
This implies that there exists $\alpha>0$ so that for $z\in L_+$
\[
(Y(z)^{-1} Y'(z))_{11}=\frac{\mathcal D_{\mathrm{in}}'(z)}{\mathcal D_{\mathrm{in}}(z)}+O\big(N^{-\alpha} \max_{j=1,2}|z-e^{ix_j}|^{-1}\big) .
\]
Note that when the derivative hits $\mathcal R$, we are conjugating the matrix $\mathcal R^{-1}\mathcal R'$ by $\mathcal N$ and this does not affect the diagonal entries, which yields the error term. 
By \eqref{eq:D}, we obtain for $z\in L_+$, 
\[
(Y(z)^{-1} Y'(z))_{11}=
t\sum_{k=1}^\Delta k \widehat{V}_k  z^{k-1} +\frac{\zeta_1}{\sqrt{2}}\frac{1}{z-e^{ix_1}}+\frac{\zeta_2}{\sqrt{2}}\frac{1}{z-e^{ix_2}} +O\big(N^{-\alpha} \max_{j=1,2}|z-e^{ix_j}|^{-1}\big) 
\]
uniformly in all relevant parameters, $t\in[0,1]$, $z\in L_+$, $\zeta_1,\zeta_2 \in \mathcal{D}_R$, $x_1,x_2 \in \T$ with $|x_1-x_2| \ge N^{\eta-1}$. 
Now, we use Cauchy's theorem to evaluate the resulting integral, 
\[
\frac{1}{2\pi i} \oint_{L_+}(Y(z)^{-1} Y'(z))_{11}\widehat{V}(z)dz 
= t\sum_{k=1}^\Delta k \widehat{V}_k  \widehat{V}_{-k} +O(N^{-\alpha}).
\]
In particular, the residue at $e^{i x_j}$ do not contribute to the integral over $ L_+$ and the error is integrable.
Similar reasoning shows that for $z\in L_-$, 
\[
(Y(z)^{-1} Y'(z))_{11}=
- t\sum_{k=1}^\Delta k\widehat{V}_{-k}  z^{-k-1} -\frac{\zeta_1}{\sqrt{2}}\frac{1}{z-e^{ix_1}}-\frac{\zeta_2}{\sqrt{2}}\frac{1}{z-e^{ix_2}}+\frac{\zeta_1+\zeta_2}{\sqrt{2}}\frac{1}{z}+\frac{N}{z}+O(N^{-\alpha})
\]
so that since $\widehat{V}_0=0$, 
\[
-\frac{1}{2\pi i} \oint_{L_-}(Y(z)^{-1} Y'(z))_{11}\widehat{V}(z)dz 
= t\sum_{k=1}^\Delta k \widehat{V}_k  \widehat{V}_{-k}
+ \tfrac{\zeta_1}{\sqrt{2}}V(x_1)+\tfrac{\zeta_2}{\sqrt{2}} V(x_2)
 +O(N^{-\alpha}).
\]

Putting everything together, this shows that as $N\to\infty$, 
\[
\log \frac{D_{N}(F_1)}{D_{N}(F_0)}=
 \sum_{k=1}^\Delta k \widehat{V}_k  \widehat{V}_{-k}
+ \tfrac{\zeta_1}{\sqrt{2}}V(x_1)+\tfrac{\zeta_2}{\sqrt{2}} V(x_2)
 +O(N^{-\alpha})
\]
with the required uniformity. 
Note that according to formulae \eqref{def:V}--\eqref{Scov}, this implies that 
\[
D_{N}(F_1) = D_{N}(F_0)
\exp\Big( \int \big( \zeta_1 C_\X(x_1,u)  
+ \zeta_2 C_\X(x_2,u) \big) \mathrm{f}_{\delta,z}(u) d u 
+ \frac{\E \langle \X , \mathrm{f}_{\delta,z}  \rangle^2}{2}  +O(N^{-\alpha})\Big) .
\]

Then, according to \eqref{HS} and \eqref{mom1}, we conclude from these asymptotics that as $N\to\infty$
\[
 \Psi_{N,\delta}^{(\zeta_1,\zeta_2,\xi)}(x_1, x_2;z) 
 = \Psi_{N}^{(\zeta_1,\zeta_2)}(x_1, x_2)  \exp\Big( O(N^{-\alpha}) \Big)  
 \]
uniformly for $\zeta_1,\zeta_2 \in \mathcal{D}_R$, $x_1,x_2 \in \T$ with $|x_1-x_2| \ge N^{\eta-1}$, $z\in \T^q$ and $\delta_1,\cdots, \delta_q \ge N^{-1+\eta}$.  

Combining these asymptotics and \eqref{2pt}, this completes the proof of Assumption \ref{ass:meso} (2). \qed
}

{

\subsection{Verification of Assumption \ref{ass:macro} (1)} \label{sec:Rmacro}
Starting from  Lemma \ref{le:morris} (which states that independently of $x_1\in\T$, $\Psi_{N}^{(\zeta_1)}(x_1)= \Psi(\zeta_1)(1+o(1))$ provided that $\zeta_1=o(N^{1/3})$  as $N\to\infty$), we will integrate the differential identity from Lemma \ref{le:di} to obtain an approximation of 
\begin{equation} \label{psi2pt}
\Psi_{N}^{(\zeta_1,\zeta_2)}(x_1, x_2)  =D_{N}(F_0)  \exp\big(-\tfrac{\zeta_1^2+ \zeta_2^2}{2} \log N-\zeta_1\zeta_2 C_\X(x_1,x_2)\big)
\end{equation}
for $\zeta_1,\zeta_2\in \mathcal D_{R_N}$ and $(x_1,x_2) \in \mathcal{A}_c=\big\{(x,y)\in \T^2: |e^{ix}-e^{iy}| \ge c  \big\}$. 
Here, the parameter $\Delta \ge c^{-1}$ is fixed, \eqref{eq:Delta}, so the arguments are exactly as in  \cite{DIK,DIK2} and we will be rather brief with details, just providing the appropriate references and emphasizing the main differences.  

Lemma \ref{le:macroR}  provides the relevant asymptotics of the matrices $Y_{N}, Y_{N+1}$ (using the relationships \eqref{eq:S} and \eqref{eq:R}) and also of $\chi_N$ (through e.g.~formula \eqref{eq:Ydef}\footnote{
According to the Appendix \ref{sec:OP}, it holds for $z\in\U$, 
\(
\chi_NY_{11,N}(z) = p_N(z),
\)
\(
\chi_N^{-1}Y_{22,N+1}(z) = \int_\T \frac{q_{N}(e^{-i\theta})}{z-e^{i\theta}}F_0(\theta)\frac{d\theta}{2\pi},
\)
\(
\chi_N^{-1}Y_{21,N+1}(z)= z^N q_N(z^{-1}) 
\)
and 
\(
\chi_N Y_{12,N}(z)= \int_\T\frac{e^{-iN\theta}p_N(e^{i\theta})}{1-ze^{-i\theta}}F_0(\theta)\frac{d\theta}{2\pi}.
\)
In particular, there is no issue in evaluating these quantities at $e^{i\theta_j}$ in formula \eqref{diff2}. Namely, the previous integrals converge by definition of $F_0$,~\eqref{eq:Vt}.
}); see e.g.~\cite[the proof of Theorem 1.8 in Section 5]{DIK}. 
Moreover, these quantities are analytic for $\zeta_1, \zeta_2 \in \mathcal{D}_\infty$
so we can apply Cauchy's formula to also obtain asymptotics of their derivatives with respect to $\zeta_2$. 
Note that in this case, it is crucial  that the asymptotics of Lemma \ref{le:macroR}  hold up to $o(N)$ since the main term in  formula~\eqref{diff2} involves $N \partial_{\zeta_2} \log(\chi_N)$. 
Regarding the assumptions of Lemma \ref{le:di}, they can also be dispensed with by an analyticity argument as explained in Section~\ref{sec:Rmeso}. 
In fact, repeating (word for word, noting that $\beta_j=0$ in this case) the arguments in the proof of \cite[Proposition 5.1]{DIK2}, one deduce from \eqref{diff2} that if $R_N=o(\log N)$ as $N\to\infty$, for any $\epsilon>0$, 
\begin{align*}
\partial_{\zeta_2}\log D_N(F_0)&=\zeta_2\left(1+\log N+\sqrt{2}\partial_{\zeta_2}\bigg(\frac{\Gamma(1+ \zeta_2/{\sqrt{2}})}{\Gamma(1+\sqrt{2}\zeta_2)}\bigg)\right)-\zeta_1\log |e^{ix_1}-e^{ix_2}|+O(N^{\epsilon-1}) 
\end{align*}
uniformly for $\zeta_1,\zeta_2\in \mathcal D_{R_N}$ and $(x_1,x_2) \in \mathcal{A}_c$ (with $\epsilon$ as in Lemma~\ref{le:macroR}).

To integrate this formula, we use the identity (see the discussion surrounding \cite[(5.27)]{DIK2})
\[
2z\left(1+\partial_z\frac{\Gamma(1+z)}{\Gamma(1+2z)}\right)  = \partial_z\log \frac{G(1+z)^2}{G(1+2z)} = \sqrt2  \partial_\zeta\big( \log \Psi(\zeta)\big)|_{\zeta=\sqrt 2z}
\]
where $\Psi$ is as in Lemma~\ref{le:morris}. 
According to \eqref{psi2pt} and \eqref{mom1}, this implies that  as $N\to\infty$, 
\[
\log\bigg(\frac{\Psi_{N}^{(\zeta_1,\zeta_2)}(x_1, x_2)}{\Psi_{N}^{(\zeta_1)}(x_1)}\bigg)
= \log\Psi(\zeta_2) +O(N^{\epsilon-1}) .
\]
Together with Lemma~\ref{le:morris}, this yields that the asymptotics \eqref{2pt} also hold uniformly $\zeta_1,\zeta_2\in \mathcal D_{R_N}$ and $(x_1,x_2) \in \mathcal{A}_c$. 

Finally, we can use the arguments\footnote{The proof, in particular  Lemma \ref{le:mesoR} applies if  $\Delta$, $\xi\in\C^q$ are fixed and $R_N= o(\log N)$ as $N\to\infty$.} of Section~\ref{sec:Rmeso} to relate 
$\Psi_{N,\delta}^{(\zeta_1,\zeta_2,\xi)}(x_1, x_2;z)$ to $\Psi_{N}^{(\zeta_1,\zeta_2)}(x_1, x_2)$  (or equivalently $D_{N}(F_1)$ to $D_{N}(F_0)$) by integrating the differential identity from Lemma~\ref{le:di2}. 
This completes the proof of Assumption \ref{ass:macro} (1) -- the error being of order $N^{\epsilon-1}$ in this regime. 
}


\appendix
\section{Approximation of distribution functions} \label{sec:Fourier}

In this appendix we consider some general probabilistic methods for estimating the tail of a probability distribution from information about characteristic functions. In the one dimensional case, this is a classical problem where the basic fact we rely on is due to Feller (see \cite[Chap.~XVI.3]{Feller71}). We will also need a multi-dimensional version of the result. 

We fix a small parameter $\mathcal W>0$ and let 
\[
\mathcal{S} = \{z \in\C  :  |\Re z| < \mathcal W\} . 
\]

\subsection{1d case.} \label{sec:1dapprox}
We consider a sequence of (real-valued) random variables $(X_N)_{N\in\N}$ whose distributions depend on an external set of parameters $\lambda\in{\Lambda}$ where for simplicity, ${\Lambda}\subset \R^m$ is {an open} set. {In our applications to $\Psi_N$ from \eqref{mom1} (though this is the multidimensional case we turn to shortly), the parameter $\lambda$ corresponds to the variables $x_i,\delta,\xi$.}   We assume that the characteristic function of $X_N$ satisfies for $\xi\in\R$,
\[
\E[e^{i \xi X_N}] = \psi_{N,\lambda}(i \xi \epsilon_N) e^{-\xi^2/2}
\]
where $\epsilon_N\to0$, $\psi_{N,\lambda}$ are analytic functions in $\mathcal{S}$ such that  $(\lambda,z) \mapsto \psi_{N,\lambda}(z)$ are {locally} bounded and {there exists compacts $A_N\subset \Lambda$ such that}
\begin{equation} \label{ass1}
\psi_{N, \lambda} =  \psi_{\lambda}+o(1)
\text{ locally uniformly in $\mathcal{S}$ and uniformly for $\lambda\in {A_N}$ as } N\to\infty. 
\end{equation}
Of course due to the local uniform convergence, the limit  $\psi_\lambda$ is also analytic in $\mathcal S$ and $\lambda\mapsto \psi_\lambda$ is {bounded}. 
 
Note that we assume that $\psi_{N, \lambda}$ is analytic in the symmetric strip $\mathcal{S}$ while the Laplace transform considered in Section~\ref{sec:gen} (see e.g.~\eqref{mom1}) is assumed to be analytic in $\mathcal{D}_\infty= [0,\sqrt{2d}]\times \R$.
This is not inconsistent because we always apply the results from this section under a measure which is biased by a $e^{\zeta X_N}$ with $\Re \zeta >0$ fixed.
Moreover, while the distinction between $A_N$ and $\Lambda$ may seem irrelevant, we emphasize that in our applications, we are interested in situations where the parameter space may depend on $N$, e.g. $|x-y|\geq N^{-1+\eta}$ as in Assumption \ref{ass:meso}.

In this setting,  we want to obtain a uniform approximation for the distribution functions $F_N$ of  $X_N$ in terms of 
\[
G_{N,\lambda}(x) = \int_{-\infty}^x (1 + \epsilon_N \psi_{N,\lambda}'(0) u) \frac{e^{-u^2/2}}{\sqrt{2\pi}} d u .
\]
This is a perturbation of the distribution function of a standard Gaussian. 
Note that if $N$ is large enough, for any $\lambda\in {A_N}$, $G_{N,\lambda}$ takes values in $[0, 2]$ and $ \|G_{N,\lambda}'\|_{L^\infty} \le  1$. 

Then, it is proved in \cite[Chap.~XVI.3]{Feller71} that  there exists a universal constant $C$ 
so that for any $\varrho>0$, 
\begin{equation} \label{Feller}
\max_\R |F_N-G_N| \le  \int_{-\varrho}^{\varrho} \bigg| \frac{\E[e^{i \xi X_N}] - \widehat{G_N'}(-\xi)}{\xi} \bigg| d\xi + C\varrho^{-1} . 
\end{equation}
To estimate this, note that 
\[
\widehat{G_{N,\lambda}'}(\xi)=(1-i\xi \epsilon_N \psi_{N,\lambda}'(0)\xi)e^{-\frac{\xi^2}{2}}.
\]
Since $\psi_{N,\lambda}(0) =1$ for any $N\in\N$, using our assumption \eqref{ass1}, we see that by Taylor's theorem, for any $R\geq 1$, there is a constant $C_R>0$ so that if  $\lambda\in {A_N}$ and  $|\xi| \le R\epsilon_N^{-1}$, 
\begin{equation} \label{charass}
\bigg| \frac{\E[e^{i \xi X_N}] - \widehat{G_{N,\lambda}'}(-\xi)}{\xi} \bigg| 
\le C_R \epsilon_N^2  |\xi| e^{-\xi^2/2} 
\end{equation}
By choosing  $\varrho =R\epsilon_N^{-1}$, letting $N \to\infty$ and then $R\to\infty$, we conclude from \eqref{Feller} that uniformly for $\lambda\in {A_N}$,
\begin{equation} \label{uni1}
\max_\R |F_N- G_N| = o(\epsilon_N) . 
\end{equation}

\medskip

This type of approximation is directly relevant in our context if $\X_N$ is an approximately Gaussian random variable with variance $\sqrt{\log N}$ to obtain the asymptotics of probability of the form $\P[\X_N \ge \gamma \log N]$ for $\gamma >0$.
 Such approximations are obtained e.g.~in \cite[Chap.~4]{FMN16}. 
However, the (local) uniformity over the various parameters of the distribution is crucial for our applications, so we formulate a general result.  

Below one should think of $\P_{N,\theta}$ being some parametrized family of probability measures giving rise to a family of approximately Gaussian random variables $\X_N=\X_{N,\theta}$.

\begin{lemma} \label{lem:1}
Let $\P_{N,\theta}$ be a sequence of probability measures depending on $N\in\N$ and $\theta \in \Theta$ for some parameter space $\Theta$.  For $\beta>0$, define a new measure $\P_{N,\theta,\beta}$ by $\frac{d \P_{N,\theta,\beta}}{d \P_{N,\theta}} = \frac{e^{\beta \X_N}}{\E_{N,\theta}[e^{\beta \X_N}]}$ and consider the random variable $X_N  = \epsilon_N(\X_N - \gamma \log N)$ with $\epsilon_N = 1/\sqrt{\log N}$.
If the characteristic function $\xi\in\R \mapsto \E_{N,\theta,\beta}\big[e^{i \xi X_N}\big] = $ satisfies \eqref{ass1} with $\lambda = (\gamma,\beta,\theta) \in {A_N}$, a compact subset of $(0,\infty) \times (0,\infty) \times \Theta$, then locally uniformly in $g\in \R$ and uniformly for  $(\gamma,\beta,\theta) \in {A_N}$, 
\[
\P_{N,\theta}[\X_N \ge \gamma \log N+g]   = \frac{\E_{N,\theta}[e^{\beta \X_N}] N^{-\gamma\beta}}{\beta\sqrt{2\pi \log N}} e^{-\beta g}\big( 1+ \underset{N\to\infty}{o(1)} \big) .
\]
\end{lemma}

\begin{proof}
{We consider first the case $g=0$.} We can rewrite for $\gamma \ge 0$, 
\[
\P_{N,\theta}[\X_N \ge \gamma \log N]    = {\E}_{N,\theta,\beta}\big[\mathbf 1\{ X_N \ge 0\} e^{- \beta \epsilon_N^{-1}X_N } \big] \E_{N,\theta} [e^{\beta \X_N}] N^{-\gamma\beta} . 
\]
Using that the characteristic function of $X_N$ under  $\P_{N,\theta,\beta}$ satisfies  \eqref{ass1},  we can use the uniform approximation \eqref{uni1} so that integrating by parts,
\begin{equation} \label{ibp}
\begin{aligned}
\E_{N,\theta,\beta}\big[\mathbf 1\{ X_N \ge 0\} e^{- \beta \epsilon_N^{-1} X_N } \big]  
& = -  \int_{0}^\infty   \P_{N,\theta,\beta}\big[0\le X_N \le x\big]  d(e^{- \beta \epsilon_N^{-1} x}) \\
& = \int_{0}^\infty  G_N'(x)e^{- \beta \epsilon_N^{-1} x} d x +  \underset{N\to\infty}{o(\epsilon_N)} , 
\end{aligned}
\end{equation}
uniformly  $\lambda\in {A_N}$.

By a change of variable, this implies that for $\beta>0$
\[
\E_{N,\theta,\beta}\big[\mathbf 1\{ X_N \ge 0\} e^{- \beta \epsilon_N^{-1} X_N } \big]    = \beta^{-1}\epsilon_N \bigg( 
\int_{0}^\infty   \left(1 + O(\beta^{-1}\epsilon_N^2 |u|)\right)  \frac{e^{-\epsilon_N^2u^2/2\beta^2}}{\sqrt{2\pi}}  e^{-u}d u +  \underset{N\to\infty}{o(1)}  \bigg) . 
\]
By Lebesgue's dominated convergence theorem, this completes the proof with the required uniformity if $g=0$.
Moreover, since these asymptotics are locally uniform in $\gamma>0$, replacing $\gamma\leftarrow \gamma + g/\log N$, we obtain the claim {for general $g\in \R$.} 
\end{proof}

The asymptotics from Lemma~\ref{lem:1} are instrumental in several arguments of this paper. To close this section, let us make a few comments on the method.

\begin{remark} The upper-bound \eqref{Feller} relies on \cite[Chap.~XVI.3]{Feller71} and the previous argument only requires \eqref{charass} which is weaker than the locally uniform limit \eqref{ass1}.
However, characteristic functions usually have an analytic extension in a small strip in the complex plane, so that the assumption \eqref{ass1} is rather natural to check in applications. 
\end{remark}


\subsection{Multidimensional case.} \label{sec:2dapprox}
We now adapt the previous argument in arbitrary dimension, This requires stronger assumptions on the characteristic function of the random vector $X_N$. 
Fix $n\in\N$ with $n\ge2$. 
We consider a sequence of continuous random $\R^n$-valued vectors $(X_N)_{N\in\N}$ whose probability distributions depend on an external parameter $\lambda\in {\Lambda}$ for an {open} set $\Lambda$. 

We assume that the characteristic function of $X_N$ satisfies for $\xi\in\R^n$,
\begin{equation} \label{cf}
\E[e^{i \xi \cdot X_N}] = \psi_{N,\lambda}(i \xi \epsilon_N) e^{-|\xi|^2/2}
\end{equation}
where $\epsilon_N\to0$, 
$\psi_{N,\lambda}$ are analytic functions in $\mathcal{S}^n$, $\lambda \mapsto \psi_{N,\lambda}$ are {locally bounded} ({uniformly on a compact} set of $\mathcal{S}^n$). 

Now, instead of assuming that $\psi_{N,\lambda}$ {can be approximated locally uniformly} in $\mathcal{S}^n$, we assume that there is a $\eta>0$ and {compact $A_N\subset \Lambda$} so that 
\begin{equation} \label{ass2}
\psi_{N, \lambda}=  \psi_{\lambda} + \underset{N\to\infty}{o(1)} \quad 
\text{ uniformly in $\big\{\xi \in\mathcal S^{n}:   |\Im  \xi_j|  \le \epsilon_N^{-\eta}, \text{ for } j\in[n] \big\} $  and  $\lambda\in {A_N}$}. 
\end{equation}

Moreover, we need an assumption about the growth of $\psi_\lambda$. More precisely, we assume that there exist constants $c>0$ and  $\varkappa \ge 2$) so that for $\xi\in\R^n$ and $\lambda\in {A_N}$, 
\begin{equation} \label{ass3}
|\psi_\lambda(i \xi )|  \le c e^{ |\xi|^\varkappa} . 
\end{equation}

For $x,y\in\R^n$, we write $x\ge y$ if $x_i \ge y_i$ for all $i\in\{1,\dots,n\}$ and $\1 =(1,\cdots,1) \in \R^n$ {(and $\mathbf 2=2\times \1$)}.  

We claim that under these assumptions, we can obtain uniform approximations for the ``\emph{distribution function}"\footnote{It turns out that it is $F_N(x)$ instead of $\mathbb P(X_{N}\leq x)$, $x\in\R^n$ which is more relevant for our purposes.} of $X_N$,
\[
F_N(x)  : =   \P\big[ 0 \le X_{N} \le \epsilon_N x \big] , \qquad x \ge 0 
\]

In a similar way, we let 
\[
G_N(x)   =  \P_{\mathcal{N}(0,\mathrm{I_n})}\big[ 0\le X \le \epsilon_N x \big]  , \qquad x \ge 0 
\] 
be the ``\emph{distribution function}" of a standard Gaussian. Our main result is then as follows.

\begin{proposition} \label{prop:approx}
If $\alpha>0$ is small enough depending the parameters $\varkappa,\eta,n$, then under the  assumptions \eqref{cf}--\eqref{ass3}, one has uniformly for  $\lambda\in {A_N}$, 
\[
\max_{x\in [0,\epsilon_N^{-\alpha}]^n} \big|F_N(x)-G_N(x)\big|  =  \underset{N\to\infty}{o(\epsilon_N^n)}.
\]
\end{proposition}

\begin{remark}
Note that the quantity  $G_N(x)$ is of order $\epsilon_N^n$ for a fixed $x$, so it is crucial to have an error term $o(\epsilon_N^n)$.  
For our purposes, we could actually replace $\epsilon_N^{-\alpha}$ by $\log \epsilon_N^{-c}$  for a large $c>0$, so it suffices to check \eqref{ass2} for any fixed,  arbitrarily small, $\eta>0$. 
\end{remark}

The basic step in the proof of Proposition~\ref{prop:approx} relies on adapting Feller's Lemma 1 in  \cite[Chap.~XVI.3]{Feller71}.

\begin{lemma} \label{lem:Feller}
Let $\phi:\R^n\to \R_+$ be a Schwartz-function such that $\int_{\R^n}\phi=1$ and its Fourier transform $\widehat\phi$ has compact support.
Let $Z$ be a random vector (taking values in $\R^n$, independent of $X_N$) with probability density function $\phi$.
Let $\beta, \kappa>0$ and define, for $x\in\R^n$, $x\ge 0$, 
\[ \begin{aligned}
\Delta_N^\pm(x)
& =  \P\big[ \mp \epsilon_N^\beta\1 \le \epsilon_N^{-1} X_{N} - \epsilon_N^{\beta+\kappa} Z \le x \pm \epsilon_N^\beta\1 \big]  -\P_{\mathcal{N}(0,\mathrm{I_n})}\big[ \mp \epsilon_N^\beta\1 \le \epsilon_N^{-1} X - \epsilon_N^{\beta+\kappa} Z \le x\pm \epsilon_N^\beta\1 \big] .
\end{aligned}\]
Let $\alpha\in(0,\frac{\beta}{n-1})$, then it holds uniformly for $x\in [0,\epsilon_N^{-\alpha}]^n$,
\[
\Delta_N^-(x)-o(\epsilon_N^n) \le 
F_N(x)  -G_N(x)   \le  \Delta_N^+(x)
+ o(\epsilon_N^n).
\]
\end{lemma}

\begin{proof}
Since $F_N$ is an increasing function (in every coordinate), we have for $t\in\R^n$ with $|t|\le 1$, 
\begin{equation} \label{dfub}
F_N(x)  -G_N(x)   \le  \P\big[ - \epsilon_N^\beta(\1-t) \le \epsilon_N^{-1} X_{N} \le x+ \epsilon_N^\beta(\1+t) \big] 
- \P_{\mathcal{N}(0,\mathrm{I_n})}\big[ 0 \le \epsilon_N^{-1} X \le x\big]  . 
\end{equation}
Observe that since the Gaussian p.d.f. is uniformly bounded on $\R^n$, we have 
\[
\max_{x\in [0,\epsilon_N^{-\alpha}]^n}  \P_{\mathcal{N}(0,\mathrm{I_n})}\big[  \epsilon_N^{-1}X  \in \big( [- \mathbf{2}\epsilon_N^\beta , x+\mathbf{2}\epsilon_N^\beta]  \setminus [0,x] \big) \big] 
\le C \epsilon_N^{n+\beta-\alpha(n-1)}
\]
for a numerical constant $C>0$. 
In particular, if $\beta> \alpha(n-1)$, then the RHS is $o(\epsilon_N^n)$. 
This implies that uniformly for $x\in [0,\epsilon_N^{-\alpha}]^n$, 
\[\begin{aligned}
\Upsilon_N(x)  &= \P_{\mathcal{N}(0,\mathrm{I_n})}\big[ -\epsilon_N^\beta \1  \le \epsilon_N^{-1}X - \epsilon_N^{\beta+\kappa} Z \le x+\1 \epsilon_N^\beta \big]  
- \P_{\mathcal{N}(0,\mathrm{I_n})}\big[ 0 \le \epsilon_N^{-1} X \le x\big]  \\
&\le \P\big[ |Z| > \epsilon_N^{-\kappa }\big] + \P_{\mathcal{N}(0,\mathrm{I_n})}\big[  \epsilon_N^{-1}X  \in \big( [- \mathbf{2}\epsilon_N^\beta , x+\mathbf{2}\epsilon_N^\beta]  \setminus [0,x] \big) \big]  \\
& =o(\epsilon_N^n) 
\end{aligned}\]
where we used that $ \P\big[ |Z| > \epsilon_N^{-\kappa }\big] \le C_k \epsilon_N^k$ for any $k\in\N$, since the p.d.f.~$\phi$ decays faster than any polynomial. 
Using the same argument, replacing $t\leftarrow \epsilon_N^\kappa t$ and integrating both sides of \eqref{dfub} against $\phi(t)$ on $\{ t \in\R^n : |t| \le \epsilon_N^{-\kappa} \}$, we obtain after dividing by $\P\big[  |Z| \le \epsilon_N^{-\kappa } \big]$, 
\[\begin{aligned}
F_N(x)  -G_N(x)  & \le
\P\big[ \big\{ -\epsilon_N^\beta \1  \le \epsilon_N^{-1}X_N - \epsilon_N^{\beta+\kappa} Z \le x+\1 \epsilon_N^\beta \big\} \cap \big\{  |Z| \le \epsilon_N^{-\kappa } \big\} \big]/\P\big[  |Z| \le \epsilon_N^{-\kappa } \big]   \\
&\qquad - \P_{\mathcal{N}(0,\mathrm{I_n})}\big[ -\epsilon_N^\beta \1  \le \epsilon_N^{-1}X - \epsilon_N^{\beta+\kappa} Z \le x+\1 \epsilon_N^\beta \big] {+} \Upsilon_N(x)  \\
& = \Delta_N^+(x)+o(\epsilon_N^n) 
\end{aligned}\]
where the error is controlled uniformly for $x\in [0,\epsilon_N^{-\alpha}]^n$. 

We can use the same strategy to obtain a lower bound, using that for $t\in\R^n$ with $|t|\le 1$, 
\[
F_N(x)  -G_N(x)  \ge   \P\big[  \epsilon_N^\beta(\1 +t) \le \epsilon_N^{-1} X_{N} \le x - \epsilon_N^\beta(\1-t) \big] 
- \P_{\mathcal{N}(0,\mathrm{I_n})}\big[ 0 \le \epsilon_N^{-1} X \le x\big]  . 
\]
The same bounds allow us to conclude that
\(
F_N(x)  -G_N(x)   \ge \Delta_N^-(x)-o(\epsilon_N^n) . 
\)
\end{proof}

We are now ready to give the proof of Proposition~\ref{prop:approx}.

\begin{proof}
Let $\alpha,\beta,\kappa>0$. 
The characteristic function of the random variable $ \epsilon_N^{-1} X_{N} - \epsilon_N^{\beta+\kappa} Z$ is given by, according to \eqref{cf}, 
\[
\xi\in\R^n \mapsto \E\big[e^{i\epsilon_N^{-1} X_{N}  \cdot \xi }\big] \widehat\phi(\epsilon_N^{\beta+\kappa}\xi)
=  \psi_{N,\lambda}(i \xi) e^{-|\xi|^2/2\epsilon_N^2} \widehat\phi(\epsilon_N^{\beta+\kappa}\xi) .
\]
In particular, it has compact support (for $|\xi| \le \epsilon_N^{-\beta-\kappa}$ as we may assume that $\widehat{\phi}(u)$ is supported in $\{u\in\R^n :|u|\le 1\}$) and, by Fourier's inversion formula, the p.d.f.~of this random variable is smooth and given by
\[
u \in\R^n \mapsto
\frac{1}{(2\pi)^n} \int_{\R^n} e^{i \xi \cdot u}  \psi_{N,\lambda}(i \xi) e^{-|\xi|^2/2\epsilon_N^2} \widehat\phi(\epsilon_N^{\beta+\kappa}\xi) d\xi .
\]
We have a similar expression (with $\psi_{N,\lambda}\leftarrow 1$) for the p.d.f.~of $ \epsilon_N^{-1} X - \epsilon_N^{\beta+\kappa} Z$ where $X\sim \mathcal{N}(0,\mathrm{I_n})$ independent of $Z$.

According to the convention of Lemma~\ref{lem:Feller}, this implies that 
\[
\Delta_N^\pm(x)
= \frac{1}{(2\pi)^n} \int_{\R^n} \big( \psi_{N,\lambda}(i \xi)-1 \big) e^{-|\xi|^2/2\epsilon_N^2} \widehat\phi(\epsilon_N^{\beta+\kappa}\xi) \Gamma_N^{\pm}(\xi;x) d\xi .
\]
where 
\[
\Gamma_N^\pm(\xi;x) = 
\int_{\R^n}  \1\{ \mp \epsilon_N^\beta\1 \le u \le x \pm \epsilon_N^\beta\1\} e^{i \xi \cdot u} d u 
\]
is uniformly bounded by $2 \epsilon_N^{-n\alpha} $ for $x\in [0,\epsilon_N^{-\alpha}]^n$ and $\xi\in\R^n$. 

Moreover, since $\psi_{N,\lambda}(0) = \psi_\lambda(0) =1$, using the conditions  \eqref{ass2}--\eqref{ass3},
\[
\big| \psi_{N,\lambda}(i \xi)-1 \big| 
\le  |\xi| \big( C  e^{|\xi|^\varkappa} + \underset{N\to\infty}{o(1)} \big) .
\]
uniformly for $\xi\in\R^n$ with $|\xi| \le  \epsilon_N^{-\eta}$ and $\lambda\in {A_N}$ 
(this estimate relies on the fact that  $\psi_{N,\lambda}$ is analytic in $\mathcal{S}^n$).
Hence, if we assume that $\eta > \beta$  (so that we can pick $\kappa>0$ with $\eta > \beta+\kappa$), we obtain
\[
\big|\Delta_N^\pm(x)\big|
\le C  \epsilon_N^{-n\alpha} \int_{\{ |\xi| \le  \epsilon_N^{-\eta}\}} |\xi | \big(e^{|\xi|^\varkappa}+ \underset{N\to\infty}{o(1)}  \big)e^{-|\xi|^2/2\epsilon_N^2}  d \xi
\]
where we used the uniform bound for $ \Gamma_N^\pm$, combined with the fact that $\| \widehat\phi\|_{\infty,\R^n} <\infty$.

This yields, 
\[
\max_{x\in [0,\epsilon_N^{-\alpha}]^n} \big|\Delta_N^\pm(x)\big|
\le C  \epsilon_N^{1+n(1-\alpha)} \bigg(  \int_{\{ |\xi| \le  \epsilon_N^{-1-\eta}\}} 
|\xi | e^{\epsilon_N^\varkappa|\xi|^\varkappa-|\xi|^2/2}  d\xi + \underset{N\to\infty}{o(1)}  \bigg)
\]
and provided that $\eta<\frac{2}{\varkappa-2}$ (if $\varkappa=2$, there is no constraint; otherwise we can always reduce $(\eta,\alpha)$), the RHS integral is $O(1)$ as $N\to\infty$.

Hence, if $\alpha< \frac 1n \wedge \frac{\eta}{n-1} \wedge \frac{2}{(\varkappa-2)(n-1)}$ (so that we can choose $\eta>\beta >  \alpha(n-1))$, we conclude  by Lemma~\ref{lem:Feller} that 
\[\begin{aligned}
\max_{x\in [0,\epsilon_N^{-\alpha}]^n} \big|  F_N(x)  -G_N(x)   \big| 
& \le \max_{x\in [0,\epsilon_N^{-\alpha}]^n} \big|\Delta_N^\pm(x)\big|
+ o(\epsilon_N^n) \\
&\le  o(\epsilon_N^n) .
\end{aligned}\]
This completes the proof of Proposition~\ref{prop:approx}. 
\end{proof}

\section{Consequence of the Assumption \ref{ass:meso}}

Our main Assumption \ref{ass:meso} imply that $\X_N$ is really an approximation to the limiting Gaussian log-correlated field $\X$ down to arbitrary mesoscopic scales; 

\begin{lemma}\label{lem:approx} 
Under the Assumption \ref{ass:meso}, 
for any compact $K\subset \Omega$ and (small) $\eta>0$, it holds uniformly for $(x,y)\in K$ with $|x-y| \ge N^{\eta-1}$,  
\[
{\rm Cov}\big[ \X_N(x),\X_N(y)\big] = C_\X(x,y)\big(1+\underset{N\to\infty}{o(1)} \big) .
\]
Moreover, if one assumes that \eqref{PsiN} holds uniformly for $\xi\in \C^q$ with $|\xi_j|\le c$ for some $c>0$, then for any test function $g\in\mathcal{C}^1_c(\Omega)$ and $q\in\N$,
\[
\E_N\big[\langle \X_N , g  \rangle^q\big]  \to \E\big[\langle \X , g  \rangle^q\big]
\qquad\text{as $N\to\infty$.}
\]
In particular, $\X_N \to \X$ as a random Schwartz distribution on $\Omega$. 
\end{lemma}

\begin{proof}
Recall the notation \eqref{mom1} and observe that 
\[
\E_N\big[ \X_N(x_1)\X_N(x_2)\big]
= \partial_{\zeta_1}\partial_{\zeta_2} \big(\E_N\big[ e^{ \zeta_1 \X_N(x_1) + \zeta_2 \X_N(x_2)}\big] \big)_{\zeta_1=\zeta_2=0} 
\]
where the RHS can equivalently be written using a double contour-integral formula (by analyticity of $\Psi_N^{(\zeta_1,\zeta_2)}(x_1,x_2)$). 
Hence, since the asymptotics \eqref{PsiN} are uniform in the relevant parameters, one obtains as $N\to\infty$, 
\[\begin{aligned}
&\E_N\big[ \X_N(x_1)\X_N(x_2)\big] \\
&= \partial_{\zeta_1}\partial_{\zeta_2}\Big(  \Psi(\zeta_1,x_1) \Psi(\zeta_2,x_2)
\exp\Big( \tfrac{\zeta_1^2+ \zeta_2^2}{2} \log N +\zeta_1\zeta_2 C_\X(x_1,x_2)\Big)\Big)_{\zeta_1=\zeta_2=0}  \big(1+{o(1)} \big) \\
&=  \Big( C_\X(x_1,x_2) \Psi(0,x_1) \Psi(0,x_2) + \Psi'(x_1) \Psi'(x_2) \Big)\big(1+{o(1)} \big)
\end{aligned}\]
where $\Psi'(x) = \partial_\zeta\big( \Psi(\zeta,x)\big)_{\zeta=0}$  and the error is uniform for  for $(x_1,x_2)\in K$ with $|x_1-x_2| \ge N^{\eta-1}$. 
A similar computation with $\zeta_2=0$ shows that $\E_N\big[ \X_N(x_1)\big] =  \Psi'(x_1)+{o(1)} $ as $N\to\infty$. Since $\Psi(0,x)=1$  for all $x\in\Omega$, this proves the first claim.

The same argument (taking $\partial_{\zeta_1}^q$ for $q\in2\N$ with $\zeta_2=\xi=0$) also shows that there are (non-decreasing) constants $C_q$ so that $\E_N\big[\X_N(x_1)^q\big] \le C_q (1+\log N)^{q/2}$ for $x_1\in K$. By Jensen's inequality, this implies that for any $g\in\mathcal{C}(K)$ and $q\in2\N$,
\begin{equation} \label{XNmombound}
\E_N \big[\langle \X_N ,g\rangle^q\big] \le \|g\|_{L^1}^{q-1} \int \E_N\big[\X_N(x_1)^q\big]|g(x)| dx
\le C_q  \|g\|_{L^1}^{q}  (1+\log N)^{q/2}. 
\end{equation}
For  $g\in\mathcal{C}_c^1(\Omega)$, let $g_\delta(x) = \langle \rho_{\delta,x}, g\rangle$ for $x\in\Omega$, $\delta\in(0,1]$ and observe that $\|g-g_\delta\|_{L^1} \le C_g \delta$ with $C_g \ge \|g\|_{L^1}$ 
Then, by H\"older's inequality and using \eqref{XNmombound} twice, 
\[ \begin{aligned}
\big| \E_N \big[\langle \X_N ,g\rangle^q\big]- \E_N \big[\langle \X_N ,g_\delta\rangle^q\big] \big|
& \le \sum_{\ell=1}^q {q \choose \ell}  \E_N \big[ |\langle \X_N ,g-g_\delta\rangle|^\ell  |\langle \X_N ,g\rangle|^{q-\ell}\big] \\ 
& \le C_g^q C_q^q  (1+\log N)^{q/2}\big( e^{q\delta} -1\big). 
\end{aligned}\]
In particular, choosing $\delta(N) = N^{-1+\eta}$ for some $\eta \in(0,1)$, we obtain 
\begin{equation} \label{exch}
 \E_N \big[\langle \X_N ,g\rangle^q\big] =  \E_N \big[\langle \X_N ,g_\delta\rangle^q\big] +o(1)\qquad\text{as $N\to\infty$.}
\end{equation}

We can now proceed to prove for the second claim starting from the fact that, choosing $\delta(N) = N^{-1+\eta}$ for some $\eta \in(0,1)$, one has for $\xi\in\C^q$, uniformly in a neighborhood of 0, and locally uniformly for  $z \in \Omega^q$, 
\[
\E_N\big[\exp\big({\textstyle \sum_{j=1}^q\xi_j \X_{N,\delta}(z_j)}\big)\big] = \exp\big(\tfrac12{\E\big(\textstyle \sum_{j=1}^q\xi_j \X_{\delta}(z_j)}\big)^2\big)  \big(1+{o(1)} \big) \qquad\text{as $N\to\infty$,}
\]
using that $\Psi_{N,\delta}^{(0,0,\xi)}(x_1, x_2; z) = 1+o(1)$ is independent of $x_1,x_2$ (cf.~Assumption \ref{ass:meso}). 
This implies that locally uniformly for $z \in \Omega^q$,  as $N\to\infty$,
\[
\E_N \big[{\textstyle \prod_{j=1}^q}  \X_{N,\delta}(z_j)\big] = \E \big[{\textstyle \prod_{j=1}^q} \X_{\delta}(z_j)\big]   \big(1+{o(1)} \big) .
\]
Hence, we can integrate this against $\prod_{j=1}^q g(z_j)$ to get 
\[\begin{aligned}
\E_N \big[\langle \X_N ,g_{\delta}\rangle^q\big]
&= \int  \E \big[{\textstyle \prod_{j=1}^q} \X_{\delta}(z_j)\big] \big(1+{o(1)} \big)   {\textstyle \prod_{j=1}^q g(z_j) dz_j} \\
& = \E \big[\langle \X ,g_{\delta}\rangle^q\big] +o(1) , 
\end{aligned}\]
as $N\to\infty$. Here, we used that $\X$ is a Gaussian log-correlated field, so that $\E \big[{\textstyle \prod_{j=1}^q} \X_{\delta}(z_j)\big] $ is bounded on the set $\{z\in K^q : |z_i-z_j|\ge \epsilon, \forall\, i\neq j\}$ by a constant depending only on $q,\epsilon$ and not on $\delta(N)$. 
It is also standard that for any  $g\in\mathcal{C}_c^1(\Omega)$ and $q\in\N$,  $\E \big[\langle \X ,g_{\delta}\rangle^q\big] \to  \E \big[\langle \X ,g\rangle^q\big]$  as $\delta\to0$. 
Thus, combining the previous asymptotics with \eqref{exch}, we conclude that 
\[
 \E_N \big[\langle \X_N ,g\rangle^q\big] =  \E_N \big[\langle \X ,g\rangle^q\big] +o(1)\qquad\text{as $N\to\infty$,}
 \]
which is the second claim.
\end{proof}

To conclude, we record the following consequence of Lemma~\ref{lem:1}. 

\begin{lemma}[F\'eray et al. (Section 4.2)] \label{lem:prob}
Under the Assumptions \ref{ass:meso}, it holds locally uniformly for $\gamma\in(0,\sqrt{2d})$ and $x\in \Omega$,  we have as $N\to\infty$,  
\begin{equation*} 
\P_N[\X_N(x) \ge \gamma \log N]  = \frac{\Psi(\gamma,x)}{\gamma \sqrt{2\pi \log N}} \exp\big(- \tfrac{\gamma^2}{2} \log N + \underset{N\to\infty}{o(1)} \big) . 
\end{equation*}
\end{lemma}

\begin{proof}
Let  $X_N  = \frac{\X_N(x) - \gamma \log N}{\sqrt{\log N}}$.
Consider the characteristic function of $X_N$ under the biased measure $\P_{N,x,\gamma}$ given by 
$\frac{d \P_{N,x,\gamma}}{d \P_N} = \frac{e^{\gamma \X_N(x)}}{\E_N[e^{\gamma \X_N(x)}]}$. Recalling  \eqref{mom1} (with $\zeta_1=\gamma+ i\xi \epsilon_N$, $\zeta_2=\xi=0$ and $\epsilon_N=\frac{1}{\sqrt{\log N}}$), we can express this characteristic function as
\[
\E_{N,x,\gamma}\big[e^{i \xi X_N}\big] =   \psi_{N,x}(i \xi \epsilon_N) e^{-\xi^2/2} ,\qquad 
\psi_{N,x}(z) = \frac{\Psi_{N}^{(\gamma+z)}(x) }{\Psi_{N}^{(\gamma)}(x) } . 
\]
By  \eqref{W1}, the function $ \psi_{N,x}$ satisfies the condition \eqref{ass1} uniformly in $x$ in a compact subset of $\Omega$, for $\gamma \in (0,\sqrt{2d})$ and $z\in\mathcal{S} = \{z \in\C  :  |\Re z| < \mathcal W\}$. Hence, we can apply Lemma~\ref{lem:1}  with $\beta=\gamma$, this yields
\begin{equation} \label{prob3}
\P_N[\X_N(x) \ge \gamma \log N]  =\frac{\E_N[e^{\gamma \X_N(x)}] N^{-\gamma^2}}{\gamma\sqrt{2\pi \log N}} \big( 1+ \underset{N\to\infty}{o(1)} \big) .
\end{equation}
Using \eqref{W1} with $\zeta=\gamma$, this completes the proof with the required uniformity. 
\end{proof}

\begin{remark}
Let us emphasize that the proof only relies on  \eqref{W1} and that the asymptotics from Lemma~\ref{lem:prob} are restricted to $\gamma\in(0,\sqrt{2d})$ because of our choice of $\mathcal{D}_\infty$. 
\end{remark}

\section{Orthogonal polynomials and Riemann-Hilbert problems}
\label{sec:OP}

In this appendix we review the basic connection between orthogonal polynomials and Riemann-Hilbert problems.

Let $F\in L^1(\T)$ and recall the notation \eqref{eq:hs}.
In terms of $F$, we define a sequence of polynomials $(p_m)_{m\in\N}$ of increasing degree, for every $m\in\N$ and $z\in\C$,
\begin{equation}\label{eq:op}
p_m(z)=\frac{1}{\sqrt{D_m(F)D_{m+1}(F)}}\begin{vmatrix}
		\int_{0}^{2\pi}F(e^{i\theta})\frac{d\theta}{2\pi} & \cdots & \int_{0}^{2\pi}e^{im\theta}F(e^{i\theta})\frac{d\theta}{2\pi}\\
		\vdots & \ddots & \vdots \\
		\int_0^{2\pi}e^{-(m-1)\theta}F(e^{i\theta})\frac{d\theta}{2\pi} & \cdots & \int_0^{2\pi}e^{-(m-1)\theta}e^{im\theta}F(e^{i\theta})\frac{d\theta}{2\pi}\\
		1 & \cdots & z^m
	\end{vmatrix},
\end{equation}
where the branch of the square root is the principal one. 

We assume that $D_{m}(F)\neq 0$ for $m=0,...,N$,
then the leading coefficients of these polynomials are non-zero (for $m\le N$) and 
\[
p_m(z)= \chi_m z^m+O(z^{m-1}) , \qquad 
\chi_m^{-2}=\frac{D_{m+1}(F)}{D_m(F)} . 
\]

Moreover, if $F\ge 0$, these form a family of orthogonal polynomial. 
By multi-linearity of \eqref{eq:op}, for $0\leq j<m$, 
\[
\int_0^{2\pi} p_m(e^{i\theta})e^{-ij\theta}F(e^{i\theta})\frac{d\theta}{2\pi}
\]
is a determinant with two identical rows, so it must vanish, that is $p_m$ is orthogonal to monomials of lower order. 
By the same argument, 
\[
\int_0^{2\pi}p_m(e^{i\theta})e^{-im\theta}F(e^{i\theta})\frac{d\theta}{2\pi}=\frac{D_{m+1}(F)}{\sqrt{D_m(F)D_{m+1}(F)}}=\frac{1}{\chi_m} .
\]

To summarize, under the conditions $D_{m}(F)\neq 0$ for $m=0,...,N$,  we have for $0\leq j\leq m \le N$, 
\begin{equation}\label{eq:porto}
	\int_0^{2\pi}p_m(e^{i\theta})e^{-ij\theta}F(e^{i\theta})\frac{d\theta}{2\pi}=\frac{\boldsymbol{\delta}_{j,m}}{\chi_m}
\end{equation}
and by a telescopic argument (and a suitable interpretation for $j=0$), we have 
\begin{equation}\label{eq:toeppoly}
D_{N+1}(F)= {\textstyle \prod_{j=1}^{N}\chi_j^{-2}} .
\end{equation}
Conversely, if the polynomials $p_j$ exist for $j<N$, we can recover the Toeplitz determinant $D_{N}(F)$ from them. 
There is no general result which guarantees that they do exist, but we can argue that apart from a countable number of values of $\zeta_1,\zeta_2,\varphi$ they do exist, and then work from this.

In addition, one often introduces a family of dual polynomials. Again, if $D_{m}(F)\neq 0$ for $m=0,...,N$, we define for $m\le N$ and $z\in\C$, 
\begin{align}\label{eq:qop}
q_m(z)&=\frac{1}{\sqrt{D_m(F)D_{m+1}(F)}}\begin{vmatrix}
	\int_{0}^{2\pi}F(e^{i\theta})\frac{d\theta}{2\pi} & \cdots & \int_{0}^{2\pi}e^{i(m-1)\theta}F(e^{i\theta})\frac{d\theta}{2\pi} & 1\\
	\vdots & \ddots & \vdots & \vdots \\
	\int_0^{2\pi}e^{-i(m-1)\theta}F(e^{i\theta})\frac{d\theta}{2\pi} & \cdots & \int_0^{2\pi}e^{-i(m-1)\theta}e^{i(m-1)\theta}F(e^{i\theta})\frac{d\theta}{2\pi}& z^{m-1}\\
	\int_0^{2\pi}e^{-im\theta}F(e^{i\theta})\frac{d\theta}{2\pi}  & \cdots & \int_0^{2\pi}e^{-im\theta}e^{i(m-1)\theta}F(e^{i\theta})\frac{d\theta}{2\pi}& z^m
\end{vmatrix}.
\end{align}
One readily checks that $q_m(z)=\chi_m z^m +O(z^{m-1})$ and for $0\leq j\leq m\le N$, 
\begin{equation}\label{eq:qorto}
\int_0^{2\pi}q_m(e^{-i\theta})e^{ij\theta}F(e^{i\theta})\frac{d\theta}{2\pi}=\frac{\delta_{j,m}}{\chi_m}.
\end{equation}
Note that if $F$  is real valued (for us it may not be), then we have $q_m(z^{-1})=\overline{p_m(z)}$ for $z\in \T$. 

\medskip

To analyze the orthogonal polynomials asymptotically, we encode them into a Riemann-Hilbert problem. For $j\in\N$, if $p_j$ and $q_{j-1}$ exist, we define for $|z|\neq 1$ 
\begin{equation}\label{eq:Ydef}
	Y(z)=Y_{j}(z)=\begin{pmatrix}
		\frac{1}{\chi_j}p_j(z) & \frac{1}{\chi_j}\int_\T\frac{e^{-i(j-1)\theta}p_j(e^{i\theta})}{e^{i\theta}-z}F(\theta)\frac{d\theta}{2\pi}\\
		-\chi_{j-1}z^{j-1}q_{j-1}(z^{-1}) & -\chi_{j-1}\int_\T\frac{q_{j-1}(e^{-i\theta})}{e^{i\theta}-z}F(\theta)\frac{d\theta}{2\pi}
	\end{pmatrix}.
\end{equation}
The basic result in this business that everything builds on is the following result of Fokas, Its, and Kitaev \cite{FIK} (which is in a different setting, and has later been adapted to many different questions). See also \cite[Chatper 3]{Deift} for a proof in yet another setting -- the argument readily adapts to this case (once one replaces Sobolev-theory by Hölder-theory) and we omit the proof.

\begin{proposition}\label{pr:YRHP}
For $j\in\N$, if the polynomials $p_j,q_{j-1}$ exist $($i.e.~$D_m(F)\neq 0$ for $m\in\{j-1,j,j+1\}$$)$, then the matrix $Y$, \eqref{eq:Ydef}, is the unique solution to the  Problem~\ref{pb:Y}. 
\end{proposition}

By applying the Deift-Zhou steepest descent analysis, one can invert this statement and, under suitable conditions,  construct a function $Y$ which solves this problem if $j$ is sufficiently large. Then, one can argue that the polynomials $p_j$, $q_{j-1}$ have to exist. 

\medskip

While variants of Lemma~\ref{le:di2} (e.g. for Hankel determinants) have certainly appeared in the literature, we provide a brief proof this differential identity since we do not know of a perfect reference.

\begin{proof}[Proof of Lemma~\ref{le:di2}]
Using \eqref{eq:toeppoly}, a short calculation (some of whose details we are omitting) shows that   
\begin{align*}
\partial_t \log D_{N}(F)&=-2\sum_{j=0}^{N-1}\frac{\partial_t \chi_j(t)}{\chi_j(t)}\\
&=-2\sum_{j=0}^{N-1}\int_{0}^{2\pi}\partial_t (p_j(e^{i\theta},t))q_j(e^{-i\theta},t)F(e^{i\theta})\frac{d\theta}{2\pi}\\
&=-\int_0^{2\pi}\partial_t \sum_{j=0}^{N-1} p_j(e^{i\theta},t)q_j(e^{-i\theta},t) F(e^{i\theta})\frac{d\theta}{2\pi}\\
&=\int_0^{2\pi}\sum_{j=0}^{N-1} p_j(e^{i\theta},t)q_j(e^{-i\theta},t) \partial_t F(e^{i\theta})\frac{d\theta}{2\pi}
\end{align*}
where the polynomials are the ones orthogonal with respect to the weight $F(e^{i\theta})$. Also we note that $\partial_t F=V F$ so if we can relate the sum above to $Y$ in a suitable way, then we will be done.

The first step in this is to use the Christoffel-Darboux identity (see \cite[Lemma 2.3]{DIK} for a result that holds also for the  complex weights we need): for $z\neq 0$
\begin{align*}
\sum_{j=0}^{N-1}p_j(z,t)q_j(z^{-1},t)&=-N p_N(z,t)q_N(z^{-1},t)+z\left(q_N(z^{-1},t)p_N'(z,t)-p_N(z,t)\frac{d}{dz}q_N(z^{-1},t)\right).
\end{align*}

Let us on the other hand look at the object that we want to see: by \eqref{eq:Ydef} (and drop the $t$ from our notation for now)
\[
z^{-N}[Y(z)^{-1}\partial_z Y(z)]_{21}=z^{-N}\left(\frac{\chi_{N-1}}{\chi_N}z^{N-1}q_{N-1}(z^{-1})p_N'(z)-\frac{\chi_{N-1}}{\chi_N}p_N(z)\frac{d}{dz}(z^{N-1}q_{N-1}(z^{-1}))\right).
\]
The next step is to apply a recursion relation for the polynomials. More precisely \cite[Lemma 2.2, (2,4)]{DIK} says that 
\begin{equation}\label{eq:rec}
\chi_{N-1}z^{-1}q_{N-1}(z^{-1})=\chi_Nq_N(z^{-1})-q_N(0)z^{-N}p_N(z).
\end{equation}
This implies that 
\[
z^{-N}\frac{\chi_{N-1}}{\chi_N}z^{N-1}q_{N-1}(z^{-1})p_N'(z)=q_N(z^{-1})p_N'(z)-\frac{q_N(0)}{\chi_N}z^{-N}p_N(z)p_N'(z).
\]
Moreover, multiplying \eqref{eq:rec} by $z^N$ and differentiating, we see that 
\[
\frac{d}{dz}(\chi_{N-1}z^{N-1}q_{N-1}(z^{-1}))=\chi_N \frac{d}{dz}(z^N q_N(z^{-1}))-q_N(0)p_N'(z).
\]
We see that the $q_N(0)$-terms cancel, so these considerations lead to 
\begin{align*}
z^{-N}(Y(z)^{-1}\partial_z Y(z))_{21}&=q_N(z^{-1})p_N'(z)-z^{-N}p_N(z)\frac{d}{dz}(z^N q_N(z^{-1}))\\
&=-N z^{-1}p_N(z)q_N(z^{-1})+q_N(z^{-1})p_N'(z)-p_N(z)\frac{d}{dz}q_N(z^{-1})\\
&=z ^{-1}\sum_{j=0}^{N-1}p_j(z)q_j(z^{-1}).
\end{align*}
Noting that if we make the change of variables $z=e^{i\theta}$ in our integral, we have $z^{-1}\frac{dz}{2\pi i }= \frac{d\theta}{2\pi}$ so (reintroducing $t$ to our notation)
\begin{align*}
\partial_t \log D_N(F)&=\int_0^{2\pi}\sum_{j=0}^{N-1}p_j(e^{i\theta},t)q_j(e^{-i\theta},t)V(e^{i\theta})F(e^{i\theta})\frac{d\theta}{2\pi}\\
&=\int_{\T}z^{-1}\sum_{j=0}^{N-1}p_j(z,t)q_j(z^{-1},t)V(z)F(z)\frac{dz}{2\pi i}\\
&=\int_{\T}z^{-N}(Y(z,t)^{-1}\partial_z Y(z,t))_{21}V(z)F(z)\frac{dz}{2\pi i},
\end{align*}
which was the claim.
\end{proof}


\begin{thebibliography}{99}
\bibitem{ABB17}
L.-P. Arguin, D. Belius, and P. Bourgade: Maximum of the characteristic polynomial of random unitary matrices. Comm. Math. Phys., 349 (2017), no. 2, 703--751. 
\bibitem{ABBRS} L. P. Arguin,  P. Bourgade, and M. Radziwi\l\l. The Fyodorov-Hiary-Keating Conjecture. I. Preprint arXiv:2007.00988.
\bibitem{AHK} L.-P. Arguin, L. Hartung, and N. Kistler: High points of a random model of the Riemann-zeta function and Gaussian multiplicative chaos. Stochastic Process. Appl. 151 (2022), 174–190
\bibitem{Berestycki} N. Berestycki: An elementary approach to Gaussian multiplicative chaos. Electron. Commun. Probab. 22 (2017), Paper No. 27, 12 pp.
\bibitem{BL} M. Biskup and O. Louidor: On intermediate level sets of two-dimensional discrete Gaussian free field.     Ann. Inst. H. Poincar\'e Probab. Statist. 55 no. 4 (2019) 1948-1987.
\bibitem{BF} P. Bourgade and H. Falconet: Liouville quantum gravity from random matrix dynamics. Preprint arXiv:2206.03029 (2022).
\bibitem{BD} D. Bump and P. Diaconis: Toeplitz minors. J. Combin. Theory Ser. A 97 (2002), no. 2, 252–271. 
\bibitem{CMN} R. Chhaibi, T. Madaule, and J. Najnudel: On the maximum of the C$\beta$E field. Duke Math. J. 167 (2018), no. 12, 2243--2345.
\bibitem{CN} R. Chhaibi and J. Najnudel: On the circle, $GMC^\gamma=\underset{\longleftarrow}{\lim}C\beta E_n$ for $\gamma=\sqrt{\frac{2}{\beta}}$, $(\gamma\leq 1)$. Preprint arXiv:1904.00578 (2019).
\bibitem{CFLW} T. Claeys, B. Fahs, G. Lambert, and C. Webb: How much can the eigenvalues of eigenvalues of a random Hermitian matrix fluctuate? How much can the eigenvalues of a random Hermitian matrix fluctuate? Duke Math. J. 170 (2021), no. 9, 2085–2235.
\bibitem{CK} T. Claeys and I. Krasovsky: Toeplitz determinants with merging singularities. Duke Math. J. 164 (2015), no. 15, 2897–2987. 
\bibitem{CLZ}  S. Coste, G. Lambert, Y. Zhu: The characteristic polynomial of sums of random permutations and regular digraphs. Preprint arXiv:2204.00524.
\bibitem{Deift}  P. A. Deift: Orthogonal polynomials and random matrices: a Riemann-Hilbert approach. Courant Lecture Notes in Mathematics, 3. New York University, Courant Institute of Mathematical Sciences, New York; American Mathematical Society, Providence, RI, 1999. viii+273 pp.
\bibitem{DIK} P. Deift, A. Its, and I. Krasovsky: Asymptotics of Toeplitz, Hankel, and Toeplitz+Hankel determinants with Fisher-Hartwig singularities. Ann. of Math. (2) 174 (2011), no. 2, 1243–1299. 
\bibitem{DIK2} P. Deift, A. Its, and I. Krasovsky: On the asymptotics of a Toeplitz determinant with singularities. Random matrix theory, interacting particle systems, and integrable systems, 93–146, Math. Sci. Res. Inst. Publ., 65, Cambridge Univ. Press, New York, 2014.
\bibitem{DKMLVZ}  P. Deift, T. Kriecherbauer, K. T.-R. McLaughlin, S. Venakides, and X. Zhou: Strong asymptotics of orthogonal polynomials with respect to exponential weights. Comm. Pure Appl. Math. 52 (1999), no. 12, 1491–1552.
\bibitem{DZ} P. Deift and X. Zhou: A steepest descent method for oscillatory Riemann-Hilbert problems. Asymptotics for the MKdV equation. Ann. of Math. (2) 137 (1993), no. 2, 295--368.
\bibitem{DE01}
P. Diaconis and S.N. Evans: Linear Functionals of Eigenvalues of Random Matrices.
Trans. Amer. Math. Soc. 353 (2001), no. 7, 2615--2633.
\bibitem{Ehrhardt} T. Ehrhardt: A status report on the asymptotic behavior of Toeplitz determinants with Fisher-Hartwig singularities. Recent advances in operator theory (Groningen, 1998), 217–241, Oper. Theory Adv. Appl.,124, Birkhäuser, Basel, 2001.
\bibitem{Feller71} W. Feller: An introduction to probability theory and its applications. Vol. II. Second edition John Wiley \& Sons, Inc., New York-London-Sydney 1971 xxiv+669 pp.
\bibitem{FMN16} V. Féray, P.-L.  Méliot, and A.  Nikeghbali: Mod-$\phi$ convergence. Normality zones and precise deviations. Springer Briefs in Probability and Mathematical Statistics. Springer, Cham, 2016. xii+152 pp.
\bibitem{FL}  C. Ferreira and J. L.  López: An asymptotic expansion of the double gamma function. J. Approx. Theory 111 (2001), no. 2, 298–314.
\bibitem{FIK} A. Fokas, A. Its, and A. Kitaev: The isomonodromy approach to matrix models in 2D
quantum gravity. Commun. Math. Phys. 147 (1992), 395–430.
\bibitem{FW} P. J. Forrester and S. O. Warnaar:
The importance of the Selberg integral. 
Bull. Amer. Math. Soc. (N.S.) 45 (2008), no. 4, 489–534. 

\bibitem{FB} Y.V. Fyodorov and J.-P. Bouchaud: Freezing and extreme-value statistics in a random energy model with logarithmically correlated potential. J. Phys. A 41 (2008), no. 37, 372001, 12 pp.
\bibitem{FK} Y.V. Fyodorov and J.P. Keating:  Freezing transitions and extreme values: random matrix theory, and disordered landscapes. Philos. Trans. R. Soc. Lond. Ser. A Math. Phys. Eng. Sci. 372 (2014), no. 2007, 20120503, 32 pp.
\bibitem{GKS} C. Genz, N. Kistler, and M. Schmidt: High points of branching Brownian motion and McKean's martingale in the Bovier-Hartung extremal process.
Electron. Commun. Probab. 23 (2018), Paper No. 86, 12 pp. 
\bibitem{IK} A. Its and I. Krasovsky. Hankel determinant and orthogonal polynomials for the Gaussian weight with a jump. Contemp. Math. 458 (2008), 215–247
\bibitem{HKO}  C.P. Hughes, J.P: Keating, N. O'Connell: On the characteristic polynomial of a random unitary matrix. Comm. Math. Phys. 220 (2001), no. 2, 429–451.
\bibitem{Jego} A. Jego: Planar Brownian motion and Gaussian multiplicative chaos. Ann. Probab. 48 (2020), no. 4, 1597–1643.
\bibitem{JS} J. Junnila and E. Saksman: Uniqueness of critical Gaussian chaos. Electron. J. Probab. 22 (2017), no. 11, 1–31.
\bibitem{JSW1} J. Junnila, E. Saksman, and C. Webb:   Imaginary multiplicative chaos: moments, regularity and connections to the Ising model. Ann. Appl. Probab. 30 (2020), no. 5, 2099–2164.
\bibitem{Kenyon} R. Kenyon: Dominos and the Gaussian free field. Ann. Probab. 29 (2001), no. 3, 1128–1137. 
\bibitem{Kuijlaars} A. Kuijlaars: Riemann-Hilbert analysis for orthogonal polynomials. Orthogonal polynomials and special functions (Leuven, 2002), 167–210, Lecture Notes in Math., 1817, Springer, Berlin, 2003.
\bibitem{Lam19} G. Lambert: Mesoscopic central limit theorem for the circular $\beta$-ensembles and applications. Electron. J. Probab. 26 (2021), article no. 7, 1–33.
\bibitem{MMS} A. Martínez-Finkelshtein, K. T.-R. McLaughlin, and E. B. Saff: Asymptotics of orthogonal polynomials with respect to an analytic weight with algebraic singularities on the circle. Int. Math. Res. Not. 2006, Art. ID 91426, 43 pp.
\bibitem{NSW} M. Nikula, E. Saksman, and C. Webb: Multiplicative chaos and the characteristic polynomial of the CUE: the L1-phase. Trans. Amer. Math. Soc. 373 (2020), no. 6, 3905–3965. 
\bibitem{NIST} NIST handbook of mathematical functions. Edited by Frank W. J. Olver, Daniel W. Lozier, Ronald F. Boisvert and Charles W. Clark. With 1 CD-ROM (Windows, Macintosh and UNIX). U.S. Department of Commerce, National Institute of Standards and Technology, Washington, DC; Cambridge University Press, Cambridge, 2010. xvi+951 pp.
\bibitem{Olver} F. W. J. Olver:  On the asymptotic solution of second-order differential equations having an irregular singularity of rank one, with an application to Whittaker functions. J. Soc. Indust. Appl. Math. Ser. B Numer. Anal. 2 (1965), 225–243. 
\bibitem{PZ}  E. Paquette, and O. Zeitouni: The extremal landscape for the C$\beta$E ensemble.
Preprint arXiv:2209.06743 (2022)
\bibitem{Remy} G. Remy: The Fyodorov-Bouchaud formula and Liouville conformal field theory. Duke Math. J. 169 (2020), no. 1, 177–211.
\bibitem{RV} R. Rhodes and V. Vargas: Gaussian multiplicative chaos and applications: a review. 
Probab. Surv. 11 (2014), 315--392.  
\bibitem{Shamov} A. Shamov: On Gaussian multiplicative chaos. J. Funct. Anal. 270 (2016), 3224–3261.
\bibitem{Simon}  B. Simon: Szeg\H{o}'s theorem and its descendants. Spectral theory for $L^2$ perturbations of orthogonal polynomials. M. B. Porter Lectures. Princeton University Press, Princeton, NJ, 2011. xii+650 pp.
\bibitem{Widom} H. Widom. Toeplitz determinants with singular generating functions. Amer. J. Math. 95 (1973), 333–383.
\end{thebibliography}
\end{document}